\documentclass[a4paper, 11pt]{amsart}

\usepackage{enumitem, amsmath, amssymb, amsthm, amsaddr, comment}
\usepackage[norelsize, ruled, boxed, commentsnumbered]{algorithm2e}
\usepackage[margin=2.44cm]{geometry}
\usepackage{graphicx,float}

\newtheorem{thm}{Theorem}[section]
\newtheorem{lemma}[thm]{Lemma}
\newtheorem{prop}[thm]{Proposition} 
\newtheorem{coro}[thm]{Corollary}
\newtheorem{claim}[thm]{Claim}
\newtheorem{setup}[thm]{Setup} 
\newtheorem{remark}[thm]{Remark}
\newtheorem{definition}[thm]{Definition}
\newtheorem{conjecture}[thm]{Conjecture}

\newcommand{\eps}{\varepsilon}
\newcommand{\vphi}{\varphi}
\newcommand{\sm}{\setminus}
\newcommand{\oHC}{\overline{\mathrm{HC}}}
\newcommand{\oHP}{\overline{\mathrm{HP}}}
\newcommand{\HC}{{\mathrm{HC}}} 
\newcommand{\HP}{{\mathrm{HP}}}
\newcommand{\Heven}{H_\textrm{even}}
\newcommand{\Hodd}{H_\textrm{odd}}

\newcommand{\T}{\mathcal{T}}   
\newcommand{\Sc}{\mathcal{S}}
\newcommand{\Part}{{\mathcal P}}
\newcommand{\Prob}{\mathbb{P}}
\newcommand{\Expect}{\mathbb{E}}
\newcommand{\N}{\mathbb{N}}
\newcommand{\red}{{\mathrm{red}}}
\newcommand{\blue}{{\mathrm{blue}}}
\newcommand{\bad}{{\mathrm{bad}}}

\title{Hamilton cycles in hypergraphs below the Dirac threshold}
\author{Frederik Garbe and Richard Mycroft}
\email{garbe@math.cas.cz, r.mycroft@bham.ac.uk}
\thanks{\emph{Postal address:} Frederik Garbe, Institute of Mathematics, Czech Academy of Sciences, \v{Z}itn\'{a} 25, 110 00, Prague, Czech Republic. Richard Mycroft, University of Birmingham, Birmingham, B15 2TT, UK. \\ \indent RM supported by EPSRC grant EP/M011771/1.}

\begin{document}
 
\begin{abstract}
We establish a precise characterisation of $4$-uniform hypergraphs with minimum codegree close to $n/2$ which contain a Hamilton $2$-cycle. As an immediate corollary we identify the exact Dirac threshold for Hamilton $2$-cycles in $4$-uniform hypergraphs. Moreover, by derandomising the proof of our characterisation we provide a polynomial-time algorithm which, given a $4$-uniform hypergraph $H$ with minimum codegree close to $n/2$, either finds a Hamilton $2$-cycle in $H$ or provides a certificate that no such cycle exists. This surprising result stands in contrast to the graph setting, in which below the Dirac threshold it is NP-hard to determine if a graph is Hamiltonian. We also consider tight Hamilton cycles in $k$-uniform hypergraphs $H$ for $k \geq 3$, giving a series of reductions to show that it is NP-hard to determine whether a $k$-uniform hypergraph $H$ with minimum degree $\delta(H) \geq \frac{1}{2}|V(H)| - O(1)$ contains a tight Hamilton cycle. It is therefore unlikely that a similar characterisation can be obtained for tight Hamilton cycles.
\end{abstract}

\clearpage\maketitle
\thispagestyle{empty}

\section{Introduction}\label{sec:intro}
The existence of Hamilton cycles in graphs is a fundamental problem of graph theory which has been an active area of research for many years. The decision problem -- given a graph $G$, determine if it contains a Hamilton cycle -- was one of Karp's famous $21$ NP-complete problems~\cite{Ka72}. This means we are unlikely to find a `nice' characterisation of Hamiltonian graphs analogous to Hall's Marriage Theorem and Edmonds's algorithm for the existence of a perfect matching in graphs. Consequently, much research has focussed on sufficient conditions which ensure the existence of a Hamilton cycle in a graph $G$, such as the classic theorem of Dirac~\cite{dirac52} that every graph on $n \geq 3$ vertices with minimum degree at least $n/2$ contains a Hamilton cycle.

In recent years a great deal of attention has been devoted towards establishing analogous results for Hamilton cycles in hypergraphs. To discuss this work we make the following standard definitions. 

A \emph{$k$-uniform hypergraph}, or \emph{$k$-graph} $H$ consists of a set of vertices $V(H)$ and a set of edges $E(H)$, where each edge consists of $k$ vertices. This generalises the notion of a (simple) graph, which coincides with the case $k=2$. Given any integer $1 \leq \ell < k$, we say that a $k$-graph $C$ is an \emph{$\ell$-cycle} if $C$ has no isolated vertices and the vertices of $C$ may be cyclically ordered in such a way that every edge of $C$ consists of $k$ consecutive vertices and each edge intersects the subsequent edge (in the natural ordering of the edges) in precisely $\ell$ vertices. It follows from the latter condition that the number of vertices of an $\ell$-cycle $k$-graph $C$ is divisible by $k-\ell$, as each edge of $C$ contains exactly $k-\ell$ vertices which are not contained in the previous edge. We say that a $k$-graph $H$ on $n$ vertices contains a \emph{Hamilton $\ell$-cycle} if it contains an $n$-vertex $\ell$-cycle as a subgraph; as above, a necessary condition for this is that $k-\ell$ divides $n$, and we assume this implicitly throughout the following discussion. It is common to refer to $(k-1)$-cycles as \emph{tight cycles} and to speak of \emph{tight Hamilton cycles} accordingly. This is the most prevalently used definition of a cycle in a uniform hypergraph, but more general definitions, such as a Berge cycle~\cite{hh-bermond78}, have also been considered. Given a $k$-graph $H$ and a set $S \subseteq V(H)$, the \emph{degree} of $S$, denoted $d_H(S)$ (or $d(S)$ when $H$ is clear from the context), is the number of edges of $H$ which contain $S$ as a subset. The \emph{minimum codegree} of $H$, denoted $\delta(H)$, is the minimum of $d(S)$ taken over all sets of $k-1$ vertices of $H$, and the \emph{maximum codegree} of $H$, denoted $\Delta(H)$, is the maximum of $d(S)$ taken over all sets of $k-1$ vertices of $H$. Note that for graphs the maximum and minimum codegree are simply the maximum and minimum degree respectively. 

\subsection{Previous work}
The study of Hamilton cycles in hypergraphs has been a thriving area of research in recent years. We briefly summarise some of this work here; for a more expository presentation we refer the reader to the recent surveys of K\"uhn and Osthus~\cite{KOSurvey}, R\"odl and Ruci\'nski~\cite{RRSurvey} and Zhao~\cite{ZhaoSurvey}. A major focus has been to find hypergraph analogues of Dirac's theorem. Specifically, the aim is to find, for each $k$ and $\ell$, the \emph{Dirac threshold}, that is, the best-possible minimum codegree condition which guarantees that a $k$-graph on $n$ vertices contains a Hamilton $\ell$-cycle. The first results in this direction were by Katona and Kierstead~\cite{hh-katona99} who established the first non-trivial bounds on the Dirac threshold for a tight Hamilton cycle in a $k$-graph for $k \geq 3$. R\"odl, Ruci\'nski and Szemer\'edi~\cite{hh-rodl06, hh-rodl08} then improved this bound by determining asymptotically the Dirac threshold for a tight Hamilton cycle, first for $k=3$ and then for any $k \geq 3$. The asymptotic Dirac threshold for any $1 \leq \ell < k$ such that $\ell$ divides $k$ follows as a consequence of this. This left those values of $\ell$ for which $k-\ell$ does not divide $k$, in which cases the Dirac threshold was determined asymptotically through a series of works by K\"uhn and Osthus~\cite{pp-kuhn06-cherry}, Keevash, K\"uhn, Mycroft and Osthus~\cite{hh-keevash11}, H\`an and Schacht~\cite{hh-han10} and K\"uhn, Mycroft and Osthus~\cite{KMO14}. These results can all be collectively described by the following theorem, whose statement gives the asymptotic Dirac threshold for any $k$ and $\ell$. 

\begin{thm}[\cite{hh-han10, hh-keevash11, KMO14, pp-kuhn06-cherry, hh-rodl06, hh-rodl08}] \label{codeg}
  For any $k \geq 3$, $1 \leq \ell < k$ and $\eta > 0$, there exists~$n_0$ such that if $n \geq
  n_0$ is divisible by $k-\ell$ and $H$ is a $k$-graph on $n$ vertices with $$\delta(H) \geq 
  \begin{cases}
    \left( \frac{1}{2} + \eta \right) n& \mbox{ if $k-\ell $ divides $k$,} \\
    \left(\frac{1}{\lceil 
    \frac{k}{k-\ell} \rceil(k-\ell)}+\eta\right) n & \mbox{otherwise,}
  \end{cases} 
  $$ 
  then $H$ contains a Hamilton $\ell$-cycle. 
\end{thm}

More recently the exact Dirac threshold has been identified in some cases, namely for $k=3, \ell=2$ by R\"odl, Ruci\'nski and Szemer\'edi~\cite{RoRuSz09}, for $k=3, \ell=1$ by Czygrinow and Molla~\cite{CM}, and for any $k \geq 3$ and $\ell < k/2$ by Han and Zhao~\cite{HZ}.

The stated motivation for many of the above works is to establish sufficient minimum-degree conditions which render the Hamilton $\ell$-cycle decision problem tractable, since the problem of determining whether a $k$-graph admits a Hamilton $\ell$-cycle is NP-hard (this can be shown by an elementary reduction from the graph case). Indeed, it is trivial to determine whether there is a Hamilton $\ell$-cycle in a $k$-graph $H$ with minimum degree above the Dirac threshold identified asymptotically in Theorem~\ref{codeg} (as by the theorem the answer must be affirmative). Moreover, for tight cycles Karpi\'nski, Ruci\'nski and Szyma\'nska~\cite{KaRuSz10} derandomised the proof of Theorem~\ref{codeg} to describe a polynomial-time algorithm which actually finds a tight Hamilton cycle in such a $k$-graph. This raises the question whether the threshold for tractability could lie substantially below the Dirac threshold. Dahlhaus, Hajnal and Karpi{\'n}ski~\cite{DHK} essentially answered this question for graphs by showing that the minimum degree needed to render the problem tractable is asymptotically equal to the Dirac threshold. That is, they showed that for any fixed $\eps > 0$ it is NP-hard to determine whether a graph $G$ with $\delta(G) \geq (\frac{1}{2}-\eps)|V(G)|$ admits a Hamilton cycle. More recently, Karpi\'nski, Ruci\'nski and Szyma\'nska~\cite{KaRuSz10} gave an analogous result for hypergraphs with minimum codegree below the lower threshold of Theorem~\ref{codeg} (actually, their statement pertained only to tight cycles, but the same construction gives the result for $\ell$-cycles for any $\ell < k$). \medskip

\begin{thm}[Karpi\'nski, Ruci\'nski and Szyma\'nska~\cite{KaRuSz10}]~\label{NPhard}
For any $1 \leq \ell < k$ and any $\eps > 0$ the following problem is NP-hard: given a $k$-graph $H$ with $\delta(H) \geq (\frac{1}{\lceil \frac{k}{k-\ell} \rceil(k-\ell)}-\eps) |V(H)|$, determine whether $H$ contains a Hamilton $\ell$-cycle.
\end{thm}

The same authors observed that for tight cycles this left a `hardness gap' of $(\tfrac{1}{k},\tfrac{1}{2}]$ for the problem of determining whether, for a fixed $c$ in this range, a $k$-graph $H$ with $\delta(H) \geq c|V(H)|$ admits a tight Hamilton cycle. More generally, there is a `hardness gap' between the results of Theorem~\ref{codeg} and Theorem~\ref{NPhard} for any $\ell$ such that $k-\ell$ divides $k$.

The reader should also note that the Dirac threshold has also been investigated for other types of degree conditions for uniform hypergraphs. Specifically, for a $k$-graph $H$ and $1\leq s < k$ we say the {\it  minimum $s$-degree of $H$} is the minimum of $d(S)$ taken over all sets of $s$ vertices of~$H$. Much less is known about $s$-degree Dirac thresholds for $s < k-1$; Bastos, Mota, Schacht, Schnitzer and Schulenburg~\cite{BMSSS16, BMSSS17} recently determined the exact $(k-2)$-degree Dirac threshold for a Hamilton $\ell$-cycle in a $k$-graph for $1 \leq \ell < k/2$, generalising previous results for $3$-graphs due to Bu\ss, H\`an and Schacht~\cite{BHS13} and Han and Zhao~\cite{hh-han14}. One significant open problem is to determine the asymptotic $1$-degree Dirac threshold for a tight Hamilton cycle in a $k$-graph. For $k = 3$ this problem was solved very recently by Reiher, R\"odl, Ruci\'nski, Schacht and Szemer\'edi~\cite{RRRSSz} (partial results were previously given by R\"odl and Ruci\'nski~\cite{RR14} and R\"odl, Ruci\'nski, Schacht and Szemer\'edi~\cite{RRSS16}). For general $k \geq 4$ the problem remains open; upper bounds were established by Glebov, Person and Weps~\cite{GPW12}.

\subsection{New results}

Our central result considers Hamilton $2$-cycles in $4$-graphs (that is, the case $k=4$ and $\ell=2$), in which case the asymptotic bound of Theorem~\ref{codeg} is the best previously known result. For this case  we provide a more detailed result than the various exact results described above: rather than merely identifying the Dirac threshold for such a cycle, we give a precise characterisation of all $4$-graphs with minimum codegree close to the Dirac threshold according to whether or not they contain a Hamilton $2$-cycle. This is the following theorem. Note for this that a bipartition of a set $V$ simply means a partition of $V$ into two sets. The precise definitions of the terms `even-good' and `odd-good' are somewhat technical, so we defer them to Section~\ref{sec:partitions}; for now the reader should be aware that each refers to the existence of certain small structures in~$H$ with respect to the bipartition of $V(H)$. Recall also that a $4$-graph can only contain a Hamilton $2$-cycle if it has even order.

\begin{thm} \label{character}
There exist $\eps, n_0 > 0$ such that the following statement holds for any even $n \geq n_0$. Let $H$ be a $4$-graph on $n$ vertices with $\delta(H)\geq n/2-\eps n$. Then $H$ admits a Hamilton $2$-cycle if and only if every bipartition of $V(H)$ is both even-good and odd-good.
\end{thm}

In particular, having established this characterisation, the exact Dirac threshold for Hamilton $2$-cycles in $4$-graphs follows by a straightforward deduction (which is given in Section~\ref{sec:dirac}).

\begin{coro}\label{dirac24}
There exists $n_1$ such that if $n \geq n_1$ is even and $H$ is a $4$-graph on $n$ vertices with 
$$
\delta(H) \geq 
  \begin{cases}
\tfrac{n}{2} - 2 & \mbox{ if $n$ is divisible by $8$,} \\
\tfrac{n}{2} - 1& \mbox{otherwise,}
  \end{cases} 
$$
then $H$ contains a Hamilton $2$-cycle. Moreover, this minimum codegree condition is best-possible in each case.
\end{coro}

It is not immediately apparent that the criterion of Theorem~\ref{character} can be tested in polynomial time, but in Section~\ref{sec:partitions} we explain why, due to the minimum codegree of $H$, this is in fact the case. Consequently, we can determine in polynomial time whether or not a $4$-graph $H$ whose codegree is close to the Dirac threshold admits a Hamilton $2$-cycle. Moreover, by derandomising the proof of Theorem~\ref{character} we can actually find such a cycle, should one exist, giving the following theorem.

\begin{thm} \label{2cycle}
There exist a constant $\eps > 0$ and an algorithm which, given a $4$-graph $H$ on $n$ vertices with $\delta(H) \geq n/2 - \eps n$, runs in time $O(n^{32})$ and returns either a Hamilton $2$-cycle in $H$ or a certificate that no such cycle exists (that is, a bipartition of $V(H)$ which is either not even-good or not odd-good).
\end{thm}

This theorem demonstrates the existence of a linear-size gap between the minimum codegree threshold which renders the Hamilton $2$-cycle problem tractable and the Dirac threshold. This provides a surprising contrast to the graph setting, for which the result of Dahlhaus, Hajnal and Karpi\'nski noted above shows that there is no such gap.

We also consider the tractability of finding a tight Hamilton cycle in a $k$-graph for $k \geq 3$. For such cycles we close the aforementioned `hardness gap' identified by Karpi\'nski, Ruci\'nski and Szyma\'nska~\cite{KaRuSz10} by proving the following theorem. 

\begin{thm} \label{tightcycle}
For any $k \geq 3$ there exists $C$ such that it is NP-hard to determine whether a $k$-graph $H$ with $\delta(H) \geq \frac{1}{2}|V(H)|-C$ admits a tight Hamilton cycle.
\end{thm}

This shows that, in stark contrast to the situation just described for $2$-cycles in $4$-graphs, the minimum codegree threshold which renders the problem tractable is asymptotically equal to the Dirac threshold (at which a tight Hamilton cycle is guaranteed to exist). 
It would be interesting to know, for other values of $k$ and $\ell$, whether the minimum codegree which renders the problem of finding a Hamilton $\ell$-cycle in a $k$-graph tractable is essentially equal to the Dirac threshold (as Theorem~\ref{tightcycle} shows is the case for tight cycles, and Theorems~\ref{codeg} and~\ref{NPhard} together show is the case for $\ell$-cycles when $(k-\ell) \nmid k$), or whether it is significantly different (as Theorem~\ref{2cycle} shows is the case for $2$-cycles in $4$-graphs).

\subsection{Organisation of the paper}

This paper is organised as follows. In Section~\ref{sec:2cycles} we establish the definitions and notation we need, including the definition of even-good and odd-good bipartitions needed for our characterisation (Theorem~\ref{character}), before giving the brief deduction of the exact Dirac threshold for Hamilton 2-cycles in 4-graphs (Corollary~\ref{dirac24}) and discussing the complexity aspects of Theorem~\ref{character}. Next, in Section~\ref{sec:proof} we give the proof of Theorem~\ref{character}, postponing a number of key lemmas to future sections. Specifically we distinguish a non-extremal case, an even-extremal case and an odd-extremal case, the necessary lemmas for which are postponed to Sections~\ref{sec:gen},~\ref{sec:evenextr} and~\ref{sec:oddextr} respectively. In parallel with the proofs of these lemmas we establish analogous algorithmic results, and in Section~\ref{sec:algo} we combine these results to formulate a polynomial-time algorithm as claimed in Theorem~\ref{2cycle}. We then consider tight Hamilton cycles in Section~\ref{sec:tight}, proving Theorem~\ref{tightcycle} via a sequence of polynomial-time reductions, before finally giving some concluding remarks in Section~\ref{sec:disc}.

A previous extended abstract of this paper~\cite{GaMy16} described many of the results of this paper but omitted all of the proofs except for those of Proposition~\ref{forwards} and the deduction of Corollary~\ref{dirac24} from Theorem~\ref{character}; for completeness we also include these proofs here. Moreover, the extended abstract only described an algorithm to determine the existence of a Hamilton $2$-cycle in a $4$-graph with minimum codegree close to the Dirac threshold; in this paper we go significantly further by derandomising the proof of Theorem~\ref{character} to give an algorithm which explicitly finds such a cycle, should it exist (Theorem~\ref{2cycle}).

This arXiv preprint ({\tt arXiv:1609.03101}) also includes an appendix with further details of some of the algorithms presented here and proofs of their correctness.

\section{A characterisation of dense 4-graphs with no Hamilton 2-cycle}\label{sec:2cycles}

\subsection{Notation}

Let $H$ be a $k$-graph. We write $e(H)$ for the number of edges of $H$. Also, for any set $S \subseteq V(H)$ we define the \emph{neighbourhood} of $S$ to be $N_H(S) := \{S' \subseteq V(H) \sm S : S \cup S' \in E(H)\}$. That is, $N_H(S)$ is the collection of all sets which together with $S$ form an edge of $H$. So $N_H(S)$ is a collection of $(k-|S|)$-sets and $|N_H(S)| = d_H(S)$. In particular, if $|S| = k-1$ then $N_H(S)$ is a set of singleton sets of vertices, in which case we identify $N_H(S)$ with the corresponding set of vertices for notational simplicity, for example writing $v \in N_H(S)$ instead of $\{v\} \in N_H(S)$.
Furthermore, when $H$ is clear from the context we write simply $N(S)$ and to avoid clutter we frequently omit braces around sets, for example writing $N(x, y, z)$ instead of $N(\{x, y, z\})$. Given a set $X \subseteq V(H)$, we write $H[X]$ to denote the subgraph of $H$ \emph{induced by $X$}, that is, the $k$-graph with vertex set $X$ and whose edges are all edges $e \in E(H)$ with $e \subseteq X$.

We define $\ell$-paths in $k$-graphs in a similar way to $\ell$-cycles. Indeed, a $k$-graph $P$ is an \emph{$\ell$-path} if $P$ has no isolated vertices and, moreover, the vertices of $P$ can be linearly ordered in such a way that every edge of $P$ consists of $k$ consecutive vertices and each edge intersects the subsequent edge in precisely $\ell$ vertices. So the number of vertices in an $\ell$-path must be congruent to $k$ module $k-\ell$. As for cycles we refer to $(k-1)$-paths as \emph{tight paths}. The \emph{length} of an $\ell$-path or $\ell$-cycle is the number of edges it contains. A \emph{segment} of an $\ell$-path $P$ or $\ell$-cycle $C$ is an $\ell$-path $P'$ which is a subgraph of $P$ or $C$ respectively. 

Now suppose that $H$ is a $4$-graph. Given a bipartition $\{A, B\}$ of $V(H)$, we say that an edge $e \in E(H)$ is \emph{odd} if $|e \cap A|$ is odd (or equivalently, if $|e \cap B|$ is odd) and \emph{even} otherwise. We denote the subgraph of $H$ consisting only of the even edges of $H$ by $\Heven$ and the subgraph of $H$ consisting only of the odd edges of $H$ by $\Hodd$. Also, we say that a pair $p$ of distinct vertices of $H$ is a \emph{split} pair if $|p\cap A|=1$, and a \emph{connate} pair otherwise. These terms are all dependent on the bipartition $\{A, B\}$ and $4$-graph $H$ in question, but these will always be clear from the context.

We use various ways of describing a $2$-path or $2$-cycle in a $4$-graph. One is is to list a sequence of vertices, that is, $(v_1, \dots, v_m)$ for some even $m \geq 4$; the edges of $C$ are then $\{v_i, v_{i+1}, v_{i+2}, v_{i+3}\}$ for each even $i$. Another is to give a sequence of pairs of vertices, that is, $p_1p_2p_3\dots p_m$ for some integer $m \geq 2$; the edges of $C$ are then $p_i \cup p_{i+1}$ for each $i$. The \emph{ends} of a $2$-path in a $4$-graph are the initial pair and final pair, that is $\{v_1, v_2\}$ and $\{v_{m-1}, v_m\}$ in the first style of notation, and $p_1$ and $p_m$ in the second style of notation. We concatenate $2$-paths in the natural way, for example, if $P$ is a $2$-path with ends $p$ and $p'$, and $Q$ is a $2$-path with ends $p'$ and $q$, and $P$ and $Q$ have no vertices in common outside $p'$, then $PQ$ is a $2$-path with ends $p$ and $q$. We sometimes say that $P$ is a path \emph{from $p$ to $q$} to mean that $P$ has ends $p$ and $q$, however, note that a path from $p$ to $q$ has the same meaning as a path from $q$ to $p$.

Given a $4$-graph $H$, the \emph{total $2$-pathlength} of $H$ is the maximum sum of lengths of vertex-disjoint $2$-paths in $H$. For example, $H$ having total $2$-pathlength $3$ indicates the presence in $H$ of three disjoint edges (\emph{i.e.} three $2$-paths of length 1), or of a $2$-path of length $3$, or of two vertex-disjoint $2$-paths, one of length 1 and one of length~2.

For an integer $k$ we write $[k]$ for the set of integers from $1$ to $k$ and, given a set $V$, we write $\binom{V}{k}$ for the set of subsets of $V$ of size $k$. Also, we write $x \ll y$ (``$x$ is sufficiently smaller than~$y$'') to mean that for any $y > 0$ there exists $x_0 > 0$ such that for any $x \leq x_0$ the subsequent statement holds. Similar statements with more variables are defined accordingly. We omit floors and ceilings throughout this paper where they do not affect the argument.

\subsection{Odd-good and even-good bipartitions of $4$-graphs} \label{sec:partitions}

Using the definitions introduced in the previous subsection, we can now give the central definition of our characterisation.

\begin{definition} \label{evengood}
Let $H$ be a $4$-graph on $n$ vertices, where $n$ is even, and let $\{A, B\}$ be a bipartition of $V(H)$. We say that $\{A, B\}$ is \emph{even-good} if at least one of the following statements holds.
\begin{enumerate}[noitemsep, label=(\roman*)]
\item $|A|$ is even or $|A| = |B|$.
\item $\Hodd$ contains edges $e$ and $e'$ such that either $e \cap e' = \emptyset$ or $e \cap e'$ is a split pair.
\item $|A| = |B|+2$ and $\Hodd$ contains edges $e$ and $e'$ with $e \cap e' \in \binom{A}{2}$.
\item $|B| = |A|+2$ and $\Hodd$ contains edges $e$ and $e'$ with $e \cap e' \in \binom{B}{2}$.
\end{enumerate}
Now let $m \in \{0, 2, 4, 6\}$ and $d \in \{0, 2\}$ be such that $m \equiv n \mod 8$ and $d \equiv |A| - |B| \mod 4$. Then we say that $\{A, B\}$ is \emph{odd-good} if at least one of the following statements holds.
\begin{enumerate}[noitemsep, label=(\roman*)]
\setcounter{enumi}{4}
\item $(m, d) \in \{(0, 0), (4, 2)\}$.
\item $(m, d) \in \{(2, 2), (6, 0)\}$ and $\Heven$ contains an edge.
\item $(m, d) \in \{(4, 0), (0, 2)\}$ and $\Heven$ has total $2$-pathlength at least two.
\item $(m, d) \in \{(6, 2), (2, 0)\}$ and either there is an edge $e \in E(\Heven)$ with $|e \cap A| = |e \cap B| = 2$ or $H_{even}$ has total $2$-pathlength at least three.
\end{enumerate}
\end{definition}

Note in particular that, if $n$ is odd, then any bipartition of $V(H)$ is neither even-good nor odd-good. We have now introduced all notation and definitions needed to understand and make use of Theorem~\ref{character}.

\begin{remark} \label{deletevs}
Suppose that $H$ has even order. If $\{A, B\}$ is a bipartition of $V(H)$ which is not even-good, then, since $\Hodd$ must not contain two disjoint edges, there exists a set $X$ of at most four vertices of $H$ such that each edge of $\Hodd$ meets $X$. Similarly, if $\{A, B\}$ is a bipartition of $V(H)$ which is not odd-good, then there exists a set $X$ of at most eight vertices of $H$ such that every edge of $\Heven$ meets $X$.
\end{remark}

Using Remark~\ref{deletevs} we can test the criterion of Theorem~\ref{character} in polynomial time in graphs of high minimum codegree. A special case of a result of Keevash, Knox and Mycroft~\cite[Lemma~2.2]{KeKnMy13} states that, given a $k$-graph $H$ on $n$ vertices with minimum codegree $\delta(H) \geq n/3$, we can list in time $O(n^5)$ all bipartitions $\{A, B\}$ of $V(H)$ with no odd edge and all bipartitions $\{A, B\}$ of $V(H)$ with no even edge, and that there are at most a constant number of such bipartitions. Hence we can first check whether the order of $H$ is even and, if so, we can establish all candidates for a bipartition $\{A, B\}$ which is not even-good or not odd-good in the following way: for each set $X$ of eight vertices of $H$, delete the vertices of $X$ from $H$ to form $H'$, and list all bipartitions $\{A', B'\}$ of $V(H')$ with no even edge or no odd edge, then for each such $\{A', B'\}$ list each of the $2^8$ possible extensions to a bipartition of $V(H)$. Clearly we can check in polynomial time whether a given bipartition of $V(H)$ is even-good and odd-good, so we can test the criterion of Theorem~\ref{character} by checking this for all listed bipartitions. Together with Theorem~\ref{character} this proves that we can determine in polynomial time whether a $4$-graph $H$ satisfying the minimum degree condition of Theorem~\ref{character} contains a Hamilton $2$-cycle. More precisely, we have the following corollary (see~\cite{GaMy16} for a more detailed description of this algorithm and proof of its correctness under the assumption that Theorem~\ref{character} holds).

\begin{coro}\label{algodecide}
There exist a constant $\eps>0$ and an algorithm which, given a $4$-graph $H$ on $n$ vertices with $\delta(H) \geq n/2 - \eps n$, determines in time $O(n^{25})$ whether $H$ contains a Hamilton $2$-cycle. Furthermore, if $H$ does not contain a Hamilton $2$-cycle, then the algorithm returns a bipartition $\{A, B\}$ of $V(H)$ which is not even-good or odd-good.
\end{coro}

Theorem~\ref{2cycle} supersedes Corollary~\ref{algodecide} by showing that we can actually find a Hamilton $2$-cycle in such a $4$-graph in polynomial time (if such a cycle exists). This does not follow from Theorem~\ref{character} directly, but instead follows by giving algorithms for each step involved in the proof of Theorem~\ref{character}. We do this in parallel with the results needed for the proof of Theorem~\ref{character} in Sections~\ref{sec:gen},~\ref{sec:evenextr} and~\ref{sec:oddextr}, before presenting the algorithm claimed in Theorem~\ref{2cycle} in Section~\ref{sec:algo}.

\subsection{The exact Dirac threshold for Hamilton $2$-cycles in $4$-graphs.} \label{sec:dirac}
We now make use of our characterisation to deduce Corollary~\ref{dirac24}. First, we give a construction to show that the degree bound of Corollary~\ref{dirac24} is best-possible. So fix an even integer $n \geq 6$ and let $A$ and $B$ be disjoint sets with $|A \cup B| = n$ such that $|A| = \tfrac{n}{2}-1$ if $8$ divides $n$ and $|A| = \tfrac{n}{2}$ otherwise. We define $H^*$ to be the $4$-graph with vertex set $A \cup B$ whose edges are all $4$-sets $e\subseteq A \cup B$ such that $|e \cap A|$ is odd. Then it is easy to see that $\delta(H^*) = \tfrac{n}{2} - 3$ if $8$ divides~$n$ and $\tfrac{n}{2}-2$ otherwise. Furthermore, the size of $A$ implies that the bipartition $\{A, B\}$ of $V(H^*)$ is not odd-good, as $H^*$ has no even edges. Hence by Theorem~\ref{character} there is no Hamilton $2$-cycle in $H^*$.

\begin{proof}[Proof of Corollary~\ref{dirac24}]
Choose $\eps, n_0$ as in Theorem~\ref{character}, and set $n_1 := \max\{n_0, \frac{2}{\eps}, 100\}$. Let $n \geq n_1$ be even and let $H$ be a $4$-graph on $n$ vertices which satisfies the minimum codegree condition of Corollary~\ref{dirac24}. Also let $\{A, B\}$ be a bipartition of $V(H)$, and assume without loss of generality that $|A| \leq \tfrac{n}{2}$. By Theorem~\ref{character} it suffices to prove that $\{A, B\}$ is even-good and odd-good. This follows immediately from Definition~\ref{evengood} if $|A| < 18$, so we may assume that $18 \leq |A| \leq |B|$ (this ensures that we may choose distinct vertices as required in what follows). Note that if $8$ divides $n$ and $|A| = \tfrac{n}{2}$, then $\{A, B\}$ is even-good by Definition~\ref{evengood}(i) and odd-good by Definition~\ref{evengood}(v). So we may assume that if $8$ divides $n$, then $|A| \leq \tfrac{n}{2}-1$ and $\delta(H) \geq \tfrac{n}{2}-2$, whilst otherwise we have $|A| \leq \tfrac{n}{2}$ and $\delta(H) \geq \tfrac{n}{2}-1$. Either way, we must have $\delta(H) \geq |A|-1$. 

To see that $\{A, B\}$ must be even-good, arbitrarily choose vertices $x_1, x_2, y_1, y_2, z_1, z_2 \in A$. Then $|N(x_1, y_1, z_1)\cap B|, |N(x_2, y_2, z_2)\cap B| \geq \delta(H) - (|A| - 3) \geq 2$, so we may choose distinct $w_1, w_2 \in B$ with $w_1 \in N(x_1, y_1, z_1)\cap B$ and $w_2 \in N(x_2, y_2, z_2)\cap B$. The sets $\{x_1, y_1, z_1, w_1\}$ and $\{x_2, y_2, z_2, w_2\}$ are then disjoint odd edges of $H$, so $\{A, B\}$ is even-good by Definition~\ref{evengood}(ii).

We next show that $\{A, B\}$ is also odd-good. For this, arbitrarily choose distinct vertices $a_1, a_2, \dots, a_9, a'_1, \dots, a'_9 \in A$ and $b_1, \dots, b_9 \in B$. 
For any $1 \leq i, j \leq 9$ we have $|N(a_i, a'_i, b_j) \cap B| \geq \delta(H)-(|A| - 2) \geq 1$, so there must be $b_j^i \in B$ such that $\{a_i, a'_i, b_j, b^i_j\}$ is an (even) edge of $H$. 
If for each $1 \leq j \leq 9$ the vertices $b^i_j$ for $1 \leq i \leq 9$ are all distinct, then $\Heven$ contains a set of nine disjoint edges, so there is no set $X \subseteq V(H)$ with $|X| \leq 8$ which intersects every even edge of $H$. 
However, by Remark~\ref{deletevs} such a set $X$ must exist if $\{A, B\}$ is not odd-good. 
So we may assume that $b^{i'}_j = b^i_j$ for some $1 \leq i,i', j \leq 9$ with $i \neq i'$. It follows that $\{a_i, a_i', b_j, b^i_j\}$ is an even edge of $H$ with exactly two vertices in $A$, whilst $(a_i, a_i', b_j, b^i_j, a_{i'}, a'_{i'})$ is a $2$-path of length $2$ in $\Heven$. 
So $\{A, B\}$ is odd-good by Definition~\ref{evengood}(v),~(vi),~(vii) or~(viii), according to the value of $n$.
\end{proof}

\section{Proof of Theorem~\ref{character}} \label{sec:proof}
In this section we prove Theorem~\ref{character}, our characterisation of $4$-graphs with minimum codegree close to the Dirac threshold which contain a Hamilton $2$-cycle, although the proofs of several lemmas are deferred to subsequent sections. We begin with the following proposition, which establishes the forward implication of Theorem~\ref{character}; note that the minimum codegree assumption is not needed for this direction.

\begin{prop} \label{forwards}
Let $H$ be a $4$-graph. If $H$ contains a Hamilton $2$-cycle, then every bipartition of $V(H)$ is both even-good and odd-good. 
\end{prop}

\begin{proof}
Let $n$ be the order of $H$, let $C = (v_1, v_2, \dots, v_n)$ be a Hamilton 2-cycle in $H$ and let $\{A, B\}$ be a bipartition of $V(H)$. Write $P_i = \{v_{2i-1}, v_{2i}\}$ for each $1 \leq i  \leq \tfrac{n}{2}$, so the edges of $C$ are $e_i := P_i \cup P_{i+1}$ for $1 \leq i  \leq \tfrac{n}{2}$ (with addition taken modulo $\tfrac{n}{2}$). The key observation is that $e_i$ is even if $P_i$ and $P_{i+1}$ are both split pairs or both connate pairs, and odd otherwise. 

We first show that $\{A, B\}$ is even-good. This holds by Definition~\ref{evengood}(ii) if $H$ contains two disjoint odd edges, so we may assume without loss of generality that all edges of $H$ other than $e_1$ and $e_{n/2}$ are even. It follows that the pairs $P_2, P_3, \dots, P_{n/2}$ are either all split pairs or all connate pairs. In the former case, if $P_1$ is a split pair then $|A| = |B|$, so Definition~\ref{evengood}(i) holds, whilst if $P_1 \subseteq A$ then Definition~\ref{evengood}(iii) holds, and if $P_1 \subseteq B$ then Definition~\ref{evengood}(iv) holds. In the latter case, if $P_1$ is a connate pair then $|A|$ is even, so Definition~\ref{evengood}(i) holds, whilst if $P_1$ is a split pair then Definition~\ref{evengood}(ii) holds. So in all cases we find that $\{A, B\}$ is even-good. 

To show that $\{A, B\}$ is odd-good, suppose first that 4 does not divide $n$, and note that by our key observation the number of even edges in $C$ must then be odd. 
If $C$ contains three or more even edges or an edge with precisely two vertices in $A$, then $\{A, B\}$ is odd-good by Definition~\ref{evengood}(vi) and~(viii), so we may assume without loss of generality that $e_{n/2}$ is the unique even edge in $C$ and that $e_{n/2} \subseteq A$ or $e_{n/2} \subseteq B$. It follows that $P_1, P_3, \dots, P_{n/2}$ are connate pairs (of which there are $\lceil{\tfrac{n}{4}}\rceil$ in total) and the remaining pairs are split, so $|A| - |B| \equiv 2\lceil{\tfrac{n}{4}}\rceil \mod 4$. We must therefore have $(m, d) \in \{(2, 2), (6, 0)\}$, and $\{A, B\}$ is odd-good by Definition~\ref{evengood}(vi). On the other hand, if $4$ divides $n$, then by our key observation the number of even edges in $C$ is even. If this number is at least two then $\{A, B\}$ is odd-good by Definition~\ref{evengood}(v) and~(vii). If instead every edge of $C$ is odd, then exactly $\tfrac{n}{4}$ of the pairs $P_i$ are connate pairs, so $|A| - |B| \equiv 2 \cdot \tfrac{n}{4} \mod 4$, and $C$ is odd-good by Definition~\ref{evengood}~(v).
\end{proof}

The main difficulty in proving Theorem~\ref{character} is therefore to establish the backwards implication. For this note that if $H$ is a $4$-graph and $\{A, B\}$ is a bipartition of $V(H)$ which is not odd-good, then $H$ has very few even edges, as it does not contain three disjoint even edges. Similarly, if $\{A, B\}$ is not even-good, then $H$ has very few odd edges. This motivates the following definition.

\begin{definition}
Fix $c > 0$, a $4$-graph $H$ on $n$ vertices, and a bipartition $\{A, B\}$ of $V(H)$.
\begin{enumerate}[noitemsep, label=(\alph*)]
\item We say that $\{A, B\}$ is \emph{$c$-even-extremal} if $\frac{n}{2}-cn \leq |A| \leq \frac{n}{2}+cn$ and $H$ contains at most $c\binom{n}{4}$ odd edges.
\item We say that $\{A, B\}$ is \emph{$c$-odd-extremal} if $\frac{n}{2}-cn \leq |A| \leq \frac{n}{2}+cn$ and $H$ contains at most $c\binom{n}{4}$ even edges.
\item We say that $H$ is \emph{$c$-even-extremal} if $V(H)$ admits a $c$-even-extremal bipartition, and likewise that $H$ is \emph{$c$-odd-extremal} if $V(H)$ admits a $c$-odd-extremal bipartition.
\end{enumerate}
\end{definition}

In our proof of Theorem~\ref{character} we distinguish between the non-extremal case, in which $H$ is neither even-extremal nor odd-extremal, and the two extremal cases.

\subsection{Non-extremal 4-graphs} Suppose first that $H$ is neither even-extremal nor odd-extremal. In this case we proceed by the so-called `absorbing' method, introduced by R\"odl, Ruci\'nski and Szemer\'edi~\cite{RoRuSz09}, to construct a Hamilton $2$-cycle in $H$. More specifically, we adapt the approach used by Karpi{\'n}ski, Ruci{\'n}ski and Szyma{\'n}ska~\cite{KaRuSz10}, proving an `absorbing lemma' and a `long cycle lemma'. Loosely speaking, our `absorbing lemma' allows us to find a short $2$-path in $H$ which can `absorb' most small collections of pairs of $H$.

\begin{lemma}[Absorbing lemma] \label{lem:absorbing}
Suppose that $1/n\ll\eps\ll \gamma\ll\lambda\ll c,\mu$. If $H$ is a $4$-graph of order $n$ with $\delta(H) \geq n/2 - \eps n$ which is neither $c$-even-extremal nor $c$-odd-extremal, then there is a $2$-path $P$ in $H$ and a $2$-graph $G$ on $V(H)$ with the following properties.
\begin{enumerate}[label=(\roman*)]
\item $P$ has at most $\mu n$ vertices.
\item Every vertex of $V(H) \sm V(P)$ is contained in at least $(1-\lambda) n$ edges of $G$. 
\item For any $s \leq \gamma n$ and any $s$ disjoint edges $e_1, \dots, e_s$ of $G$ which do not intersect $P$.
there is a $2$-path $P^*$ in $H$ with the same ends as $P$ such that $V(P^*) = V(P) \cup \bigcup_{j=1}^s e_j$.
\end{enumerate}
\end{lemma}

Next, our `long cycle lemma' states that, having found an absorbing-path $P_0$ in $H$, we can cover almost all vertices of $H$ by a long $2$-cycle $C$ of which $P_0$ is a segment such that the vertices not covered by $C$ form a collection of pairs which can be absorbed by $P_0$.

\begin{lemma}[Long cycle lemma]\label{lem:longcycle}
Suppose that $1/n \ll \eps \ll \gamma \ll \lambda \leq \mu \ll c$ and that $n$ is even. 
Let $H=(V,E)$ be a $4$-graph of order $n$ with $\delta(H)\geq n/2-\eps n$ which is not $c$-even-extremal.
Also let $P_0$ be a $2$-path in $H$ on at most $\mu n$ vertices, and let $G$ be a $2$-graph on~$V$ such that each vertex $v \in V \sm V(P_0)$ has $d_G(v) \geq (1-\lambda) n$. Then $H$ contains a $2$-cycle~$C$ on at least $(1-\gamma) n$ vertices such that $P_0$ is a segment of $C$ and $G[V\sm V(C)]$ contains a perfect matching.
\end{lemma}
 
By combining these two lemmas we obtain the following lemma, which shows that if $H$ is a non-extremal $4$-graph whose minimum codegree is close to the Dirac threshold, then $H$ contains a Hamilton $2$-cycle, thereby establishing the backwards implication of Theorem~\ref{character} for non-extremal $4$-graphs. 

\begin{lemma} \label{thm:gencase}
Suppose that $1/n \ll\eps \ll c$ and that $n$ is even, and let $H$ be a $4$-graph of order $n$ with $\delta(H)\geq n/2-\eps n$. If $H$ is neither $c$-odd-extremal nor $c$-even-extremal, then $H$ contains a Hamilton 2-cycle.
\end{lemma}

\begin{proof}
We introduce constants $\gamma,\lambda,\mu>0$ such that
$$\tfrac{1}{n}\ll\eps\ll\gamma\ll\lambda\ll\mu\ll c\;.$$
By Lemma~\ref{lem:absorbing} there is a $2$-path $P_0$ in $H$ and a graph $G$ on $V(H)$ such that the properties~(i)-(iii) of Lemma~\ref{lem:absorbing} hold. Then by Lemma~\ref{lem:longcycle} there is a $2$-cycle $C$ on at least $(1-\gamma)n$ vertices which contains $P_0$ as a segment and such that $G[V\sm V(C)]$ contains a perfect matching $e_1,\cdots,e_s$. Denote the $2$-path $C[V(C)\sm V(P_0)]$ by $P'$. Since $s \leq \gamma n$, there is a $2$-path $P^*$ with the same ends as $P_0$ such that $V(P^*)=V(P_0)\cup\bigcup_{j=1}^{s}e_j$, and  $P^*P'$ is then a Hamilton $2$-cycle in~$H$.
\end{proof}

\subsection{Extremal 4-graphs} It remains to prove the backwards implication of Theorem~\ref{character} in the case where $H$ is either even-extremal or odd-extremal. In each case we have significant structural information about $H$. Indeed, if $V(H)$ admits an even-extremal bipartition $\{A, B\}$ (respectively odd-extremal), then we know that almost all edges of $H$ are even (respectively odd). However, the assumption that the bipartition of $H$ is even-good (respectively odd-good) yields certain small structures in $\Hodd$ (respectively $\Heven$). By a careful analysis we are able to use this information to find a Hamilton $2$-cycle in $H$ in each case. We consider the even-extremal case in Section~\ref{sec:evenextr} and the odd-extremal case in Section~\ref{sec:oddextr}, culminating in the following two lemmas, which establish the backwards implication of Theorem~\ref{character} in each extremal case.

\begin{lemma} \label{thm:evenextr}
Suppose that $1/n \ll \eps, c \ll 1$ and that $n$ is even, and let $H$ be a $4$-graph of order $n$ with $\delta(H)\geq n/2-\eps n$. If $H$ is $c$-even-extremal and every bipartition $\{A, B\}$ of $V(H)$ is even-good, then $H$ contains a Hamilton $2$-cycle. 
\end{lemma}

\begin{lemma}\label{thm:oddextremal}
Suppose that $1/n \ll \eps, c \ll 1$ and that $n$ is even, and let $H$ be a $4$-graph of order $n$ with $\delta(H)\geq n/2-\eps n$. If $H$ is $c$-odd-extremal and every bipartition $\{A, B\}$ of $V(H)$ is odd-good, then $H$ contains a Hamilton $2$-cycle. 
\end{lemma}

Combining Lemmas~\ref{thm:gencase},~\ref{thm:evenextr} and~\ref{thm:oddextremal} completes the proof of Theorem~\ref{character} by establishing the backwards implication.

\begin{proof}[Proof of Theorem~\ref{character}]
Fix a constant $c$ small enough for Lemmas~\ref{thm:evenextr} and~\ref{thm:oddextremal}. Having done so, choose $\eps$ sufficiently small for us to apply Lemma~\ref{thm:gencase} with this choice of $c$, and $n_0$ sufficiently large that we may apply Lemmas~\ref{thm:gencase},~\ref{thm:evenextr} and~\ref{thm:oddextremal} with these choices of $c$ and~$\eps$ and any even $n \geq n_0$. Let $n \geq n_0$ be even, and let $H$ be a $4$-graph on $n$ vertices with $\delta(H) \geq (\tfrac{1}{2} - \eps) n$, and suppose that every bipartition $\{A, B\}$ of $V(H)$ is both even-good and odd-good. If $H$ is either $c$-even-extremal or $c$-odd-extremal then $H$ contains a Hamilton $2$-cycle by Lemma~\ref{thm:evenextr} or~\ref{thm:oddextremal} respectively. On the other hand, if $H$ is neither $c$-odd-extremal nor $c$-even-extremal then $H$ contains a Hamilton $2$-cycle by Lemma~\ref{thm:gencase}. This completes the proof of the backwards implication of Theorem~\ref{character}; the proof of the forwards implication was Proposition~\ref{forwards}.
\end{proof}

\section{Hamilton 2-cycles in 4-graphs: non-extremal case}\label{sec:gen}

In this section we give the proofs of Lemma~\ref{lem:absorbing} (the absorbing lemma) and Lemma~\ref{lem:longcycle} (the long cycle lemma). Throughout this section we only consider $2$-paths and $2$-cycles, so we suppress the $2$ and speak simply of paths, cycles and Hamilton cycles.

\subsection{A connecting lemma}

We begin with a `connecting lemma'. This states that if $H$ is a $4$-graph whose minimum codegree is close to the Dirac threshold, then either $H$ is even-extremal, or $H$ is well-connected in the sense that any two disjoint pairs of vertices of $H$ are connected by many short paths. This property is encapsulated in the following definition.

\begin{definition}
For $\kappa > 0$, we say that a $4$-graph $H$ of order $n$ is \emph{$\kappa$-connecting} if for every two disjoint pairs $p_1,p_2\in\binom{V(H)}{2}$ there are either at least $\kappa n^2$ paths of length $2$ or at least $\kappa n^4$ paths of length $3$ whose ends are $p_1$ and $p_2$.
\end{definition}

In the proof of our connecting lemma we will make use of the following result of Goodman~\cite{Go59}, that any two-colouring of a complete graph has many monochromatic triangles.

\begin{thm}[\cite{Go59}]\label{thm:Goodman}
Suppose that $1/n\ll\eps$. If $G$ is a graph on $n$ vertices, then
$$|\{S \in \tbinom{V(G)}{3} : G[S]\cong K_3\textrm{ or }G[S]\cong\overline{K_3}\;\}|\geq \tfrac{1-\eps}{24}n^3\;.$$
\end{thm}

\begin{lemma}[Connecting lemma]\label{lem:con}
Suppose that $1/n\ll\eps\ll\kappa\ll c$ and that $H = (V, E)$ is a $4$-graph on $n$ vertices with $\delta(H) \geq n/2-\eps n$. If $H$ is not $\kappa$-connecting, then $H$ is $c$-even-extremal. Moreover, there exists an algorithm Procedure~\ref{proc:evenpartition}$(H)$ which returns a $c$-even-extremal bipartition $\{A, B\}$ of $V$ in time~$O(n^8)$.
\end{lemma}

\begin{proof}
Introduce further constants $\beta,\eta,\delta>0$ such that
\[\tfrac{1}{n}\ll\eps\ll\kappa\ll\beta\ll\eta\ll\delta\ll c\;.\]
Suppose that $H$ is not $\kappa$-connecting. Then we may fix two disjoint pairs $\{a_1,a_2\}$ and $\{b_1,b_2\}$ of vertices of $H$ such that there are fewer than $\kappa n^2$ paths of length $2$ whose ends are $\{a_1,a_2\}$ and $\{b_1,b_2\}$, and fewer than $\kappa n^4$ paths of length $3$ whose ends are $\{a_1,a_2\}$ and $\{b_1,b_2\}$. Observe that there are then at most $\beta n$ vertices $v \in V$ with  $|N(a_1,a_2,v)\cap N(b_1,b_2,v)|\geq \beta n$, as otherwise there are at least $\frac{1}{2} \beta^2 n^2>\kappa n^2$ paths of length $2$ connecting $\{a_1,a_2\}$ and $\{b_1,b_2\}$. Now we colour the edges of the complete graph $G$ on $V$ as follows. For distinct vertices $v, w \in V$ we say that
$$
\{v,w\}\textrm{ is }
\begin{cases}
\textrm{red, if }\{a_1,a_2,v,w\}\in E\textrm{ and }\{b_1,b_2,v,w\} \notin E,\\ 
\textrm{blue, if }\{a_1,a_2,v,w\}\notin E\textrm{ and }\{b_1,b_2,v,w\} \in E,\\ 
\textrm{uncoloured, otherwise.} 
\end{cases} 
$$
Because of our first observation we have for at least $(1-\beta)n-4\geq (1-2\beta)n$ vertices $v\in V$ that 
\begin{equation}
d_\red(v),d_\blue(v)\geq \delta(H)-\beta n\geq (\tfrac{1}{2}-2\beta)n\;.\label{eq:coldeg}
\end{equation}
Furthermore, for at least $\frac{1}{2}(n-2)\delta(H) \geq(\frac{1}{2}-2\eps)\binom{n}{2}$ pairs $\{v,w\}\subseteq V$ we have $\{a_1,a_2,v,w\}\in E$, and likewise for at least $(\frac{1}{2}-2\eps)\binom{n}{2}$ pairs $\{v,w\}\subseteq V$ we have $\{b_1,b_2,v,w\}\in E$. Since by assumption at most $\kappa n^2$ pairs $\{x,y\}\subseteq V$ fulfil both $\{a_1,a_2,x,y\}\in E$ and $\{b_1,b_2,x,y\}\in E$, it follows that there are at most 
\begin{equation}
2\kappa n^2 +4\eps\binom{n}{2}<5 \kappa\binom{n}{2} \label{eq:uncol}
\end{equation}
uncoloured edges. Therefore the number of red edges and the number of blue edges are each at least
\begin{equation}
(\tfrac{1}{2}-2\eps)\binom{n}{2}-5\kappa\binom{n}{2}\geq(\tfrac{1}{2}-6\kappa)\binom{n}{2}\;.\label{eq:redblue}
\end{equation}

Now we show that $G$ either contains at least $n^3/50$ red triangles or at least $n^3/50$ blue triangles. To do this, suppose that there are fewer than $n^3/50$ red triangles. Then by Theorem~\ref{thm:Goodman} there are at least $(1-\eps)\frac{n^3}{24}-\frac{n^3}{50} \geq \frac{n^3}{50} + 3 \kappa n^3$ triangles which consist only of blue and uncoloured edges. Since by~\eqref{eq:uncol} there are at most $5 \kappa\binom{n}{2}\cdot n \leq 3\kappa n^3$ triangles in $G$ which contain an uncoloured edge, there are then at least $n^3/50$ blue triangles in $G$. So we may assume without loss of generality that $G$ contains at least $n^3/50$ red triangles. At least $n^3/50 -2\beta n^3 \geq n^3/100$ of these red triangles contain only vertices which fulfil~\eqref{eq:coldeg}; we denote the set of such triangles by $\T$. Furthermore for a triangle $T\in \T$ with vertex set $\{v_1, v_2, v_3\}$ we denote by $N_\red(T)$ (respectively $N_\blue(T)$) the set of vertices $x \in V$ such that each of the edges $\{v_1, x\}$, $\{v_2, x\}$ and $\{v_3, x\}$ is red (respectively blue). We call an edge $\{x,y,z,w\}$ of~$H$ {\it colourful} if it contains disjoint red and blue edges of $G$, say $\{x, y\}$ and $\{z, w\}$ respectively. Observe that each colourful edge of $H$ disjoint from $\{a_1,a_2,b_1,b_2\}$ creates at least one path $(a_1,a_2,x,y,z,w,b_1,b_2)$ of length $3$ in $H$ with ends $\{a_1, a_2\}$ and $\{b_1, b_2\}$. Since there are at most $\kappa n^4$ such paths, but no two distinct colourful edges of $H$ can create the same path, there must be at most 
\begin{equation}
\kappa n^4+4n^3<2\kappa n^4\label{eq:coledges}
\end{equation}
colourful edges of $H$. This observation will prove the following claim.

\begin{claim}\label{clm:triangle}
There exist vertices $v_1,v_2,v_3\in V$ such that $T^*=\{v_1,v_2,v_3\}$ is a triangle with $T^*\in \T$ and such that the following properties hold.
\begin{enumerate}[label=(\roman*)]
\item $|N_\red(T^*)|\geq (\frac{1}{2}-\eta)n$, \label{clm:redneigh}
\item there are at most $\eta n$ vertices $x\in N_\red(T^*)$ with $|N_H(v_1,v_2,x)\cap N_\blue(T^*)|\geq\eta n$, \label{clm:blueinH}
\item there are at most $\eta n$ vertices $x\in N_\blue(T^*)$ with $|N_H(v_1,v_2,x)\cap N_\red(T^*)|\geq\eta n$, \label{clm:redinH}
\item there are at most $\eta n$ vertices $x\in V$ with $|N_H(v_1,v_2,x)\cap N_\blue(x)|\geq\eta n$. \label{clm:blueedgesinH}
\end{enumerate}
\end{claim}

\begin{proof}[Proof of Claim~\ref{clm:triangle}]
Arbitrarily fix, for each $T \in \T$, a labelling of the vertices of $T$ as $v_1, v_2$ and $v_3$. For this labelling we refer to the pair $\{v_1, v_2\}$ as the \emph{specified pair} of $T$. Now let $\T_1$ consist of those triangles $T \in \T$ which do not satisfy~(i). Likewise, let $\T_2, \T_3$ and $\T_4$ consist of those triangles $T \in \T$ which do not satisfy~(ii),~(iii) and~(iv) respectively (for our choice of specified pair). Let $T \in \T_1$, so by definition we have $|N_\red(T)| < (\frac{1}{2}-\eta)n$. Since $T \in \T$, each vertex of $T$ satisfies~\eqref{eq:coldeg} and so is incident to at most $4\beta n$ uncoloured edges of $G$, so there are at least $|N_H(V(T))|-|N_\red(T)|-3\cdot 4\beta n\geq(\eta-\eps-12\beta)n > \frac{1}{2}\eta n$ colourful edges of $H$ containing $V(T)$. So in total $H$ has at least $\frac{1}{4} \cdot |\T_1| \cdot \frac{1}{2}\eta n$ colourful edges; by \eqref{eq:coledges} we find that $|\T_1| < \frac{16\kappa}{\eta} n^3 < \beta n^3$. Next observe that at least $\tfrac{1}{n}|\T_2|$ pairs $\{v_1,v_2\} \in \binom{V}{2}$ are the specified pair for some triangle $T \in \T_2$. For each such pair there are at least $\frac{1}{2}\eta^2 n^2$ pairs $\{x, y\} \in \binom{V}{2}$ with $x \in N_\red(T)$ and $y \in N_H(v_1,v_2,x)\cap N_\blue(T)$, and each $4$-tuple $\{v_1, v_2, x, y\}$ formed in this way is a colourful edge of $H$. Overall this gives at least $\frac{1}{6} \cdot \tfrac{1}{n}|\T_2| \cdot\frac{1}{2}\eta^2 n^2$ colourful edges in total, so $|\T_2| < \beta n^3$. Essentially the same argument shows that $|\T_3| < \beta n^3$. Finally, at least $\tfrac{1}{n}|\T_4|$ pairs $\{v_1,v_2\} \in \binom{V}{2}$ are the specified pair for some triangle $T \in \T_4$. For each such pair there are at least $\frac{1}{2}\eta^2 n^2$ pairs $\{x, y\} \in \binom{V}{2}$ with $y \in N_H(v_1,v_2,x)\cap N_\blue(x)$, and each $4$-tuple $\{v_1, v_2, x, y\}$ formed in this way is a colourful edge of $H$. In total this gives at least $\frac{1}{6} \cdot \tfrac{1}{n}|\T_4| \cdot \frac{1}{2}\eta^2 n^2$ colourful edges, so $|T_4| < \beta n^3$. So at least $|\T| - \sum_{i=1}^4 |\T_i| > |\T| - 4\beta n^3 > 0$ triangles $T\in \T$ satisfy~(i)-(iv).
\end{proof}

Fix a triangle $T^*\in\T$ with vertex set $\{v_1, v_2, v_3\}$ as in Claim~\ref{clm:triangle}, and define $d_{\mathrm{min}}:=(\frac{1}{2}-2\beta)n$, $N^*_{\mathrm{r}}:=N_\red(T^*)$ and $N^*_{\mathrm{b}}:=N_\blue(T^*)$. Since $T^* \in \T$ each $v_i$ fulfills~\eqref{eq:coldeg} and so is incident to at least $d_{\mathrm{min}}$ blue edges, none of which has an endvertex in $N^*_{\mathrm{r}}$, so we have
\begin{equation}
|N^*_{\mathrm{b}}|\geq |V \sm N^*_{\mathrm{r}}|-3(|V \sm N^*_{\mathrm{r}}|-d_{\mathrm{min}})=3d_{\mathrm{min}}-2|V\setminus N^*_{\mathrm{r}}|>(\tfrac{1}{2}-3\eta)n\;, \label{eq:blueneigh}
\end{equation}
where the final inequality is by Claim~\ref{clm:triangle}\ref{clm:redneigh}. Also by \eqref{eq:coldeg} we derive the upper bounds
\begin{equation}
|N^*_{\mathrm{r}}|\leq (\tfrac{1}{2}+2\beta)n\;\textrm{ and } |N^*_{\mathrm{b}}|\leq (\tfrac{1}{2}+2\beta)n\;.\label{eq:upboundneigh}
\end{equation}
Together,~\eqref{eq:blueneigh} and Claim~\ref{clm:triangle}\ref{clm:redneigh} tell us that at most $4\eta n$ vertices of $V$ are neither in $N^*_{\mathrm{r}}$ or $N^*_{\mathrm{b}}$, so by Claim~\ref{clm:triangle}\ref{clm:blueinH}, all but at most $\eta n$ vertices $x\in N^*_{\mathrm{r}}$ satisfy
\begin{equation}
|N_H(v_1,v_2,x)\cap N^*_{\mathrm{r}}| \geq \delta(H) - |N_H(v_1,v_2,x)\cap N^*_{\mathrm{b}}| - 4 \eta n \geq(\tfrac{1}{2}-6\eta)n\;.\label{eq:redNintHN} 
\end{equation}
By the same argument using Claim~\ref{clm:triangle}\ref{clm:redinH}, all but at most $\eta n$ vertices $x\in N^*_{\mathrm{b}}$ have
\begin{equation}
|N_H(v_1,v_2,x) \cap N^*_{\mathrm{b}}|\geq (\tfrac{1}{2}-6\eta)n\;.\label{eq:blueNintHN} 
\end{equation}
Define $A := N^*_{\mathrm{r}}$ and $B := V\sm A$; we conclude the proof by showing that $\{A, B\}$ is a $c$-even-extremal bipartition of $V$. For this, observe that by~\eqref{eq:upboundneigh},~\eqref{eq:redNintHN} and Claim~\ref{clm:triangle}\ref{clm:blueedgesinH} at most $2 \eta n$ vertices $x \in N^*_{\mathrm{r}}$ have $|N_{\textrm{blue}}(x) \cap N^*_{\mathrm{r}}| > 8 \eta n$. Similarly, by~\eqref{eq:upboundneigh},~\eqref{eq:blueNintHN} and Claim~\ref{clm:triangle}\ref{clm:blueedgesinH} at most $2 \eta n$ vertices $x \in N^*_{\mathrm{b}}$ have $|N_{\textrm{blue}}(x) \cap N^*_{\mathrm{b}}| > 8 \eta n$.
Since at most $5 \kappa \binom{n}{2}$ edges of $G$ are uncoloured (see~\eqref{eq:uncol}), it follows that at least 
$$\tfrac{1}{2}\cdot(|N^*_{\mathrm{r}}|-2\eta n)\cdot(|N^*_{\mathrm{r}}|-8\eta n)+\tfrac{1}{2}\cdot(|N^*_{\mathrm{b}}|-2\eta n)\cdot(|N^*_{\mathrm{b}}|-8\eta n)-5\kappa\binom{n}{2}\geq(\tfrac{1}{2}-\delta)\binom{n}{2}$$
edges of $G[A] \cup G[B]$ are coloured red. Since by Claim~\ref{clm:triangle}\ref{clm:redneigh} and~\eqref{eq:upboundneigh} we have
$$\binom{|A|}{2}+\binom{|B|}{2}\leq \binom{(\tfrac{1}{2}+\eta)n}{2}+\binom{(\tfrac{1}{2}-\eta)n}{2}\leq (\tfrac{1}{2}+\delta)\binom{n}{2},$$
we conclude that at most $2\delta\binom{n}{2}$ edges of $G[A] \cup G[B]$ are not red. Together with~\eqref{eq:redblue} this shows that at least $(\tfrac{1}{2}-6\kappa-2\delta)\binom{n}{2}\geq(\tfrac{1}{2}-3\delta)\binom{n}{2}$ of the at most $\frac{n^2}{4}\leq(\tfrac{1}{2}+\delta)\binom{n}{2}$ edges of~$G$ with one vertex in $A$ and the other in $B$ are blue, and so at most $4\delta \binom{n}{2} $ such edges are not blue. Combining this and~\eqref{eq:coledges}, we find that the number of edges of $H$ with an odd number of vertices in $A$ is at most
\[2\delta\binom{n}{2}\cdot(\tfrac{1}{2}+\delta)\binom{n}{2} + 4\delta\binom{n}{2}\cdot(\tfrac{1}{2}+\delta)\binom{n}{2}+ 2\kappa n^4 < c\binom{n}{4}\;.\]
Since $(\tfrac{1}{2}-c)n<|A|<(\tfrac{1}{2}+c)n$ by Claim~\ref{clm:triangle}\ref{clm:redneigh} and~\eqref{eq:upboundneigh}, it follows that the bipartition $\{A, B\}$ is $c$-even-extremal. To complete the proof we note that each step of the proof directly translates to an algorithm which returns the desired bipartition in the claimed running-time.
\end{proof}

\subsection{A reservoir lemma}

Another lemma which we use to prove our `long cycle lemma' is a `reservoir lemma'. This states that given a $4$-graph $H$ and a graph $G$ on the same vertex set $V$, we can choose a `reservoir set', that is, a small subset $R \subseteq V$ such that $H[R]$ and $G[R]$ are representative of $H$ and $G$. A standard and straightforward probabilistic argument shows that such a subset must exist, but derandomising this argument to give an algorithm which finds such a subset is somewhat more technical. For this we adapt the approach of Karpi\'nski, Ruci\'nski and Szyma\'nska~\cite{KaRuSz10} to our setting. We first describe an algorithm, Procedure~\ref{proc:SelSet}, which is similar to Procedure~SelectSubset from~\cite{KaRuSz10}, but chooses a subset which is representative of two graphs (instead of just one graph) simultaneously. We will use this procedure in this subsection to find a `reservoir set', and also in the next subsection to find an absorbing path. The conditions on the two graphs from which we want to choose the subset are described by the following setup.

\begin{setup} \label{setup:derand}
Fix constants $\beta, \lambda, \tau >0$ and integers $m, M, N$ and $r$ with $1 \leq r \leq N$. Let $U$ and $W$ be disjoint sets of sizes $|U| = M$ and $|W| = N$. Let $G_1$ be a graph with vertex set $U \cup W$ such that $G_1[U]$ is empty, $G_1[W]$ has precisely $m$ edges, and $|N_{G_1}(u)|\geq \beta N$ for every $u\in U$. Also let $G_2$ be a graph with vertex set $W$ such that $|N_{G_2}(w)|\geq (1-\lambda)N$ for every $w\in W$. Finally, define $\nu := 2mr/N^2$.
\end{setup} 

Note that specifying the sextuple $(G_1,G_2,r,\beta,\lambda,\tau)$ determines all of the information given in Setup~\ref{setup:derand}.
\begin{procedure}
\SetAlgoVlined
\SetAlgoNoEnd
\caption{SelectSet($G_1,G_2,r,\beta,\lambda,\tau$)}
\KwData{ A sextuple $(G_1,G_2,r,\beta,\lambda,\tau)$ as in Setup~\ref{setup:derand}.}
\KwResult{A `reservoir set' $R \subseteq W$.}
\BlankLine
Set $U:=\{u_1,\cdots,u_M\}$, $W:=\{w_1,\cdots,w_N\}$ and $R':=\emptyset$.\\
\For{$k=1$ to $r$}
   {
   \For{$j=1$ to $N$}
      {
      Set $R'_j := R' \cup \{w_j\}$. \\
      Set $e_{j} :=e(G_1[R'_j])$, $e_{j}' :=e(G_1[R'_j, W\sm R'_j])$ and $e_{j}'' := e(G_1[W\sm R'_j])$. \\
      \For{$i=1$ to $M$}
         {
         Set $d_{i,j}:= |N_{G_1}(u_i) \cap R'_j|$ and $d_{i,j}' := |N_{G_1}(u_i) \sm R'_j|$.
         }
      \For{$i=1$ to $N$}
         {
         Set $f_{i,j} :=|N_{G_2}(w_i) \cap R'_j|$ and $f'_{i, j}:=|N_{G_2}(w_i) \sm R'_j|$.\\
         }
      }
   Find $j_k \in [N] \sm \{j_1, \dots, j_{k-1}\}$ such that $A+B+C < 1$, where  
   \begin{align*}
   A:=& \sum_{i=1}^M\sum_{s \leq(\beta-\tau)r-d_{i,j_k}}\frac{\binom{d_{i,j_k}'}{s}\binom{N-k-d_{i,j_k}'}{r-k-s}}{\binom{N-k}{r-k}}, \\
   B:=& \frac{1}{\nu r}\left(e_{j_k} +e_{j_k}' \frac{r-k}{N-k}+e_{j_k}'' \frac{(r-k)(r-k-1)}{(N-k)(N-k-1)}\right),\\ 
   C:=&\sum_{i = 1}^N\sum_{s\leq (1-2\lambda)r-f_{i,j_k}}\frac{\binom{f'_{i, j_k}}{s}\binom{N-k-f'_{i,j_k}}{r-k-s}}{\binom{N-k}{r-k}}. 
   \end{align*}
   Add $w_{j_k}$ to $R'$.
   }
Remove one vertex from each edge of $G_1[R']$ and call the resulting set $R$.\\
\Return{$R$.}
\label{proc:SelSet}
\end{procedure}
Whilst Procedure~\ref{proc:SelSet} is entirely deterministic, to analyse it we consider a set of~$r$ vertices chosen uniformly at random and apply the following Chernoff-type bounds for binomial and hypergeometric distributions.

\begin{thm} [\cite{JaLuRu11RG}, Corollary 2.3 and Theorem 2.10]\label{chernoff}
Suppose $X$ has binomial or hypergeometric distribution and $0<a<3/2$. Then
$\Prob(|X - \Expect(X)| \ge a \Expect(X)) \le 2
\exp({-\frac{a^2}{3}\Expect(X)})$.
\end{thm}

\begin{prop}\label{prop:selset}
Adopt Setup~\ref{setup:derand}. Assume additionally that $\beta>\tau$, that $r$ and $N$ are sufficiently large, and that $M \leq \tfrac{1}{8} \cdot \exp(\tfrac{\tau^2r}{3\beta})$ as well as $N \leq \tfrac{1}{8} \cdot \exp(\tfrac{\lambda r}{3})$. Then in time $O(N^4+MN^3)$ Procedure~\ref{proc:SelSet}$(G_1,G_2,r,\beta,\lambda,\tau)$ returns a set $R\subseteq W$ such that
\begin{enumerate}[label=(\alph*)]
\item $(1-\nu)r\leq |R|\leq r$,
\item $R$ is an independent set in $G_1$,
\item $|N_{G_1}(u)\cap R|\geq(\beta-\tau-\nu)r$ for all $u\in U$ and
\item $|N_{G_2}(w)\cap R|\geq (1-2\lambda-\nu)r$ for all $w\in W$.
\end{enumerate}
\end{prop}

\begin{proof}
For a set $S \in \binom{W}{r}$, let $X(S)$ denote the number of vertices $u\in U$ with $|N_{G_1}(u)\cap S| \leq (\beta-\tau)r$, let $Y(S)$ denote the number of edges in $G_1[S]$ and let $Z(S)$ denote the number of vertices $w\in W$ with $|N_{G_2}(w)\cap S| \leq(1-2\lambda)r$. Now choose uniformly at random a subset $S \in \binom{W}{r}$, and let $X, Y$ and $Z$ denote the random variables $X(S), Y(S)$ and $Z(S)$ respectively.
For each $u \in U$ the random variable $|N_{G_1}(u)\cap S|$ has a hypergeometric distribution with expectation at least $\beta r$; by Theorem~\ref{chernoff} and our assumption that $M \leq \tfrac{1}{8} \exp(\tfrac{\tau^2 r}{3\beta})$ it follows that the probability of the event $|N_{G_1}(u)\cap S| \leq (\beta-\tau)r$ is at most $2 \exp({-\frac{(\tau/\beta)^2}{3}\beta r}) \leq \tfrac{1}{4M}$. So by linearity of expectation we have $\Expect(X) < \tfrac{1}{4}$. Likewise, for each $w \in W$ the random variable $|S \sm N_{G_2}(w)|$ has a hypergeometric distribution with expectation at most $\lambda r$, so a similar calculation using our assumption that $N \leq \tfrac{1}{8} \exp(\tfrac{\lambda r}{3})$ shows that $\Expect(Z) < \tfrac{1}{4}$. By linearity of expectation we also have $\Expect(Y) = m\cdot \frac{r(r-1)}{N(N-1)} \leq m(\frac{r}{N})^2$, so in particular we have $\tfrac{1}{\nu r}\Expect(Y) \leq \tfrac{1}{2}$. Thus we have in total that 
\begin{equation}
\Expect(X)+\tfrac{1}{\nu r}\Expect(Y)+\Expect(Z)<1\;.\label{eq:sumofexp}
\end{equation}
Now suppose that for some $1 \leq k \leq r$ we have already chosen vertices $w_{j_1},\cdots,w_{j_{k-1}}\in W$ such that 
$$\Expect(X|j_1,\cdots,j_{k-1})+
\tfrac{1}{\nu r}\Expect(Y|j_1,\cdots,j_{k-1})+
\Expect(Z|j_1,\cdots,j_{k-1})<1\;,$$
where we identify $j_i$ with the event that $w_{j_i} \in S$. Note that the base case $k=1$ is guaranteed by \eqref{eq:sumofexp}. Then, by the law of total probability, it is possible to choose $w_{j_k} \in W \sm \{w_{j_1}, \dots, w_{j_{k-1}}\}$ such that
$$\Expect(X|j_1,\cdots,j_{k})+\tfrac{1}{\nu r}\Expect(Y|j_1,\cdots,j_{k})+\Expect(Z|j_1,\cdots,j_{k})<1\;.$$
Having chosen $w_{j_1}, \dots, w_{j_r}$ in this way, define $R' := \{w_{j_1}, \dots, w_{j_r}\}$. Then
$$X(R')+\frac{1}{\nu r}Y(R')+Z(R')=
\Expect(X|j_1, \dots, j_r)+\frac{1}{\nu r}\Expect(Y|j_1, \dots, j_r)+\Expect(Z|j_1, \dots, j_r) < 1.$$ 
So $X(R') = Z(R') = 0$, as $X(R')$ and $Z(R')$ must have non-negative integer values. Also $Y(R') < \nu r$, meaning that $G_1[R']$ has at most $\nu r$ edges. So if we form $R$ from $R'$ by removing one vertex from each of these edges, then $R$ has the properties~(a)-(d). 

It therefore suffices to show that the choices of $j_1, \dots, j_k$ in Procedure~\ref{proc:SelSet} are identical to the choices of $j_1, \dots, j_k$ in the above argument, so the resulting sets $R'$ are identical. That is, we must show that for each $1 \leq k \leq r$ we have $A = \Expect(X|j_1,\cdots,j_{k})$, $B = \tfrac{1}{\nu r}\Expect(Y|j_1,\cdots,j_{k})$ and $C = \Expect(Z|j_1,\cdots,j_{k})$, where $A$, $B$ and $C$ are the quantities given in Procedure~\ref{proc:SelSet}. The first two of these equalities were established in~\cite{KaRuSz10} (and are straightforward to verify). For the third define $f_{i, j}$ and $f'_{i, j}$ as in Procedure~\ref{proc:SelSet}. Then we have
\begin{align*}
\Expect(Z|j_1,\cdots,j_{k})&=\sum_{w \in W} \Prob\left(|N_{G_2}(w) \cap S|\leq (1-2\lambda)r \mid j_1,\cdots,j_k\right)\\
&=\sum_{i = 1}^N\sum_{s\leq (1-2\lambda)r-f_{i,j_k}}\Prob\left(|N_{G_2}(w_i) \cap (S \sm \{j_1, \dots, j_k\})| = s \mid j_1,\cdots,j_k\right) \\
&=\sum_{i = 1}^N\sum_{s\leq (1-2\lambda)r-f_{i,j_k}}\frac{\binom{f'_{i, j_k}}{s}\binom{N-k-f'_{i,j_k}}{r-k-s}}{\binom{N-k}{r-k}} = C, 
\end{align*}
as required.
\end{proof}

A simple application of Procedure~\ref{proc:SelSet} gives an algorithm to find a reservoir set.

\begin{lemma}[Reservoir lemma]\label{lem:reservoir}
Suppose that $1/n \ll \rho \ll \lambda, \kappa$, that $H=(V,E)$ is a $4$-graph of order $n$ which is $\kappa$-connecting, and that $G$ is a $2$-graph on the same vertex set $V$ with $\delta(G)\geq n-\lambda n$. Then there exists a subset $R\subseteq V$ such that
\begin{enumerate} [label=(\alph*)]
\item $(1-4\rho)\rho n\leq|R|\leq\rho n$,
\item for every $x\in V$ we have $|N_G(x)\cap R|\geq (1-35\lambda)|R|$ and
\item for every disjoint $p_1,p_2\in\binom{V}{2}$ there are at least $\frac{\kappa}{5}|R|$ internally disjoint paths of length at most three in $H[R\cup p_1\cup p_2]$ with ends $p_1$ and $p_2$.
\end{enumerate}
Moreover, there exists an algorithm, Procedure~\ref{proc:selres}($H,G,\rho$), which returns such a subset $R \subseteq V$ in time $O(n^{16})$.
\end{lemma}

\begin{proof}
We first define the graphs on which we will use Procedure~\ref{proc:SelSet}. For this set $U :=\{\{p,p'\} : p, p' \in \binom{V}{2} \mbox{ and } p\cap p'=\emptyset\}$ and $W :=\binom{V}{4}$. We also set 
\begin{align*}
&E_1:=\{\{\{p,p'\},S\}:  \{p,p'\}\in U, S\in W\textrm{ and }H[S\cup p\cup p']\textrm{ contains a path with ends }p,\;p'\},\\
&E_1':=\{\{S,S'\}: S, S' \in W \mbox{ and } S\cap S'\neq\emptyset\},\mbox{ and }\\
&E_2:=\{\{S,S'\}: S, S' \in W \mbox{ and } \{u,v\}\in E(G)\textrm{ for all }u\in S,v\in S'\}, 
\end{align*}
and define graphs $G_1:=(U\cup W,E_1\cup E_1')$ and $G_2:=(W,E_2)$. We now define $M := |U| \leq 3\binom{n}{4}$ and $N := |W| = \binom{n}{4}$, $r := \frac{\rho}{4}n$, $\beta := \kappa$, $\lambda' := 17\lambda$ and $\tau := \rho$. We also define $m := |E_1'|$ and note that we then have $m \leq \frac{1}{2} \cdot\binom{n}{4} \cdot 4 \cdot\binom{n-1}{3}$, from which it follows that $$\nu := \frac{2mr}{N^2} = \frac{2 \cdot 2 \binom{n}{4} \binom{n-1}{3} \cdot \tfrac{\rho}{4}n}{\binom{n}{4}^2} = 4\rho.$$ 

Observe that $G_1[U]$ is empty and $G_1[W]$ has precisely $m$ edges. Furthermore, since $H$ is $\kappa$-connecting we have $d_{G_1}(\{p,p'\}) \geq \min \{\kappa n^4, \frac{1}{6} \cdot \kappa n^2 \cdot \binom{n-2}{2}\} \geq \kappa \binom{n}{4}$ for every $\{p,p'\}\in U$, or in other words, $|N_{G_1}(u)| \geq \beta N$ for all $u \in U$. Also, since $\delta(G)\geq(1-\lambda)n$ we have $d_{G_2}(S)\geq \binom{n}{4} - 4\binom{n-1}{3} - 4\cdot\lambda n\cdot\binom{n-5}{3} \geq(1-17\lambda)\binom{n}{4}$ for every $S\in W$, or in other words, $|N_{G_2}(w)| \geq (1-\lambda') N$ for all $w \in W$. So our chosen graphs and constants satisfy Setup~\ref{setup:derand} with $\lambda'$ in place of $\lambda$, and the conditions of Proposition~\ref{prop:selset} are satisfied (in particular our assumption that $\tfrac{1}{n} \ll \lambda, \kappa, \rho$ gives $M \leq \tfrac{1}{8} \cdot \exp(\tfrac{\tau^2r}{3\beta})$ and $N \leq \tfrac{1}{8} \cdot \exp(\tfrac{\lambda r}{3})$. We may therefore apply Procedure~\ref{proc:SelSet} to obtain $R':=$\ref{proc:SelSet}($G_1,G_2,r,\beta,\lambda',\tau) \subseteq W$ such that
\begin{enumerate}[label=(\alph*)]
\item $(1-4\rho)\frac{\rho}{4} n\leq |R'|\leq\frac{\rho}{4}n$, 
\item $R'$ is an independent set in $G_1$,
\item $|N_{G_1}(u)\cap R'|\geq(\kappa-\rho-4\rho)\frac{\rho}{4} n>\frac{\kappa \rho}{5} n$ for all $u\in U$ and 
\item $|N_{G_2}(w)\cap R'|\geq (1-34\lambda-4\rho)\frac{\rho}{4} n>(1-35\lambda)\frac{\rho}{4} n$ for all $w\in W$. 
\end{enumerate}
We now define $R:=\bigcup_{S\in R'}S$; it remains to show that $R$ has the desired properties. Indeed,~(b) implies that the members of $R'$ are pairwise-disjoint, so $|R| = 4|R'|$, and so~(a) immediately implies~(i). Now consider some $x \in V$, and arbitrarily choose a set $S \in W$ with $x \in S$. Then by~(d) there are at least $(1-35\lambda)\frac{\rho}{4} n$ sets $S' \in R'$ such that every $y \in S'$ is a neighbour of $x$; since the members of $R'$ are pairwise-disjoint it follows that $|N_G(x)\cap R| \geq (1-35\lambda)\rho n \geq (1-35\lambda)|R|$, establishing~(ii). Finally,~(c) implies that for every disjoint $p_1,p_2\in\binom{V}{2}$ there are at least $\frac{\kappa \rho}{5} n \geq\frac{\kappa}{5}|R|$ sets $S \in R'$ such that $H[S \cup p_1 \cup p_2]$ contains a path of length at most three with ends $p_1$ and $p_2$. Together with the fact that the members of $R'$ are pairwise-disjoint this gives~(iii). Furthermore, the dominant term of the running time is that of Procedure~\ref{proc:SelSet}, which runs in time $O(N^4 + MN^3) = O(n^{16})$.
\end{proof}

\subsection{Proof of the absorbing lemma (Lemma~\ref{lem:absorbing})}

We now turn to the proof of the absorbing lemma (Lemma~\ref{lem:absorbing}), for which we make the following general definition of an absorbing structure.

\begin{definition}\label{def:conabs}
Let $\alpha,\beta>0$ and let $H=(V,E)$ be a $4$-graph of order $n$. 
\begin{enumerate}[label=(\alph*)]
\item We say that an ordered octuple $O = (a_1,a_2,c_1,c_2,c_3,c_4,b_1,b_2)$ of distinct vertices of~$H$ is an \emph{absorbing structure} for a pair $p \in\binom{V}{2}$ if there are paths $P$ and $P'$ in $H$, both with ends $\{a_1, a_2\}$ and $\{b_1, b_2\}$, such that $V(P) = O$ and $V(P') = O \cup p$.
\item We say that a pair $p \in \binom{V}{2}$ is \emph{$\beta$-absorbable} if there are at least $\beta n^8$ absorbing structures for $p$ in $H$.
\item We say that $H$ is \emph{$(\alpha,\beta)$-absorbing} if at most $\alpha n^2$ pairs $p\in\binom{V}{2}$ are not \emph{$\beta$-absorbable}.
\end{enumerate}
\end{definition}

More specifically, we will work with the two types of absorbing structures given by the next definition. 

\begin{definition}
Let $H=(V,E)$ be a $4$-graph, and let $x$ and $y$ be distinct vertices of~$H$. We say that an ordered octuple $O = (a_1,a_2,c_1,c_2,c_3,c_4,b_1,b_2)$ of vertices of $H$, is an \emph{absorbing structure of type $1$ for the pair $\{x, y\}$} if
$$\{a_1,a_2,c_1,c_2\}, \{c_1,c_2,c_3,c_4\}, \{c_3,c_4,b_1,b_2\}, \{c_1,c_2,x,y\}, \{x,y,c_3,c_4\}\in E\;.$$
Similarly, we say that $O$ is an \emph{absorbing structure of type $2$ for $\{x,y\}$} if 
$$\{a_1,a_2,c_1,c_2\}, \{c_1,c_2,c_3,c_4\}, \{c_3,c_4,b_1,b_2\}, \{a_1,a_2,x ,y\}, \{x,y,c_1,c_4\}, \{c_2,c_3,b_1,b_2\}\in E\;.$$
\end{definition}

\begin{figure}
\begin{center}
\includegraphics[width=\textwidth]{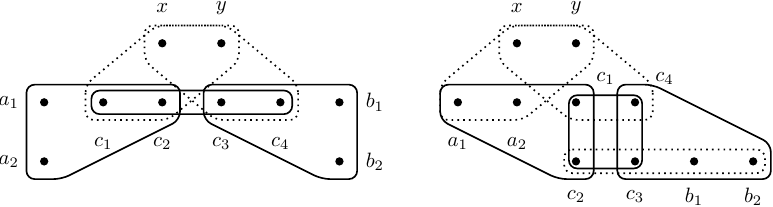}
\end{center}
\caption{Absorbing structures of type 1 (left) and type 2 (right) for the pair $p = \{x, y\}$. In each case the edges of the path not containing $p$ are marked by solid lines, whilst the paths including $p$ use the edges marked by dashed lines together with some of the edges marked by solid lines.}
\label{fig:absstr}
\end{figure}

The absorbing structures are depicted in Figure~\ref{fig:absstr}. Observe that if $O$ is an absorbing structure for $\{x, y\}$ of type $1$ or type $2$, then the sequence $(a_1,a_2,c_1,c_2,c_3,c_4,b_1,b_2)$ forms a path in $H$ with vertex set $O$ and ends $\{a_1, a_2\}$ and $\{b_1, b_2\}$. Moreover, if $O$ is an absorbing structure of type $1$ then $(a_1, a_2, c_1, c_2, x, y, c_3, c_4, b_1, b_2)$ is a path in $H$ with vertex set $O \cup \{x, y\}$ and the same ends, whilst if $O$ is an absorbing structure of type $2$ then $(a_1, a_2, x, y, c_1, c_4, c_2, c_3, b_1, b_2)$ is a path in $H$ with vertex set $O \cup \{x, y\}$ and the same ends. Therefore absorbing structures of type $1$ and $2$ are indeed absorbing structures according to Definition~\ref{def:conabs}. We now show that most pairs in non-odd-extremal $4$-graphs are contained in many absorbing structures.

\begin{lemma}\label{lem:absstruc}
Suppose that $1/n\ll\eps\ll \beta \ll\alpha\ll c$ and that $H=(V,E)$ is a $4$-graph of order $n$ with $\delta(H) \geq n/2 - \eps n$. If $H$ is not $(\alpha,\beta)$-absorbing, then $H$ is $c$-odd-extremal. Moreover, there exists an algorithm Procedure~\ref{proc:oddpartition}($H$) which returns a $c$-odd-extremal bipartition $\{A, B\}$ of $V$ in time $O(n^6)$.
\end{lemma}

\begin{proof}
We introduce constants $\omega,\omega',\vphi,\eps_2,\psi,\eps_1>0$ such that 
\[\tfrac{1}{n}\ll\eps\ll \beta\ll\omega\ll\omega'\ll\vphi \ll \eps_2\ll\psi\ll\eps_1 \ll\alpha\ll c\;.\]
Furthermore, for a pair $\{x,y\} \in \binom{V}{2}$ we define
\[T_{x,y}=\{\{e_1,e_2,e_3\}\subseteq E\;:\;e_1\cap e_2=\{x,y\} \mbox{ and } (e_1\cup e_2)\sm\{x,y\}=e_3\}\;.\]
We then can make the following observation.
\begin{claim}\label{clm:abs1}
If $|T_{x,y}| \geq \beta n^4$, then $\{x, y\}$ is $\beta$-absorbable. Moreover, there exists an edge $\{x,y,x',y'\}$ of $H$ such that $|T_{x,y}|<\beta n^4$ and $|N_H(x,y)\cap N_H(x',y')|<\frac{\omega}{2}\binom{n}{2}$.
\end{claim}

\begin{proof}
Suppose that $|T_{x,y}| \geq \beta n^4$ for some $\{x,y\} \in \binom{V}{2}$. Then there are at least~$8 \beta n^4$ ordered sextuples $S = (c_1,c_2,x,y,c_3,c_4)$ such that each of $\{c_1,c_2,x,y\}$, $\{x,y,c_3,c_4\}$ and $\{c_1,c_2,c_3,c_4\}$ is an edge of $H$. By the fact that $\delta(H)\geq(\tfrac{1}{2}-\eps)n$, for each such sextuple $S$ there are at least $(n-6)((\tfrac{1}{2}-\eps)n-4) \geq \tfrac{3n^2}{8}$ ways to choose an ordered pair $(a_1, a_2)$ of vertices of $V \sm S$ such that $\{a_1, a_2, c_1, c_2\} \in E(H)$, and then at least $(n-8)((\tfrac{1}{2}-\eps)n-6) \geq \tfrac{3n^2}{8}$ ways to choose an ordered pair $(b_1, b_2)$ of vertices of $V \sm (S \cup \{a_1, a_2\})$ such that $\{b_1, b_2, c_3, c_4\} \in E(H)$. Overall this gives at least $8\beta n^4 \cdot (\tfrac{3n^2}{8})^2 > \beta n^8$-many $10$-tuples $(a_1,a_2,c_1,c_2,x,y,c_3,c_4,b_1,b_2)$, each of which is an absorbing structure of type~$1$ for $\{x, y\}$. Hence $\{x, y\}$ is $\beta$-absorbable, proving the first statement of the claim.

Since by assumption $H$ is not $(\alpha,\beta)$-absorbing, we may choose a pair $\{x,y\} \in \binom{V}{2}$ which is not $\beta$-absorbable, so $|T_{x,y}| < \beta n^4$. Since $|N_H(x, y)| \geq (\frac{1}{2}-2\eps)\binom{n}{2}$, there must then exist a pair $\{x',y'\} \in N_H(x,y)$ such that $|N_H(x,y)\cap N_H(x',y')|<\frac{\omega}{2}\binom{n}{2}$, as otherwise we would have $|T_{x,y}|\geq \frac{1}{2} \cdot(\frac{1}{2}-2\eps)\binom{n}{2}\cdot\frac{\omega}{2}\binom{n}{2} > \beta n^4$. This gives the desired edge $\{x,y,x',y'\}$ of $H$.
\end{proof}

Fix an edge $e := \{x,y,x',y'\}$ of $H$ as in Claim~\ref{clm:abs1}. We now colour the edges of the complete $2$-graph $K$ on $V$ in the following way: for $\{a,b\}\in\binom{V}{2}$ we say that
$$\{a,b\} \mbox{ is }
\begin{cases} 
\red & \mathrm{if~} \{x,y,a,b\}\in E(H),\{x',y',a,b\}\notin E(H) \textrm{ and }a,b \notin e,\\ 
\blue & \mathrm{if~} \{x,y,a,b\}\notin E(H), \{x',y',a,b\}\in E(H)\textrm{ and }a,b \notin e,\\ 
\mathrm{uncoloured} & \mathrm{otherwise}.
\end{cases}$$
Note that if $\{a,b\}$ and $\{a',b'\}$ are disjoint red edges of $K$ such that $\{a,b,a',b'\}\in E$, then 
$\{\{x,y,a,b\},\{x,y,a',b'\},\{a,b,a',b'\}\}$ is an element of $T_{x,y}$. In such a case we call the edge $\{a,b,a',b'\}$ of $H$ a \emph{red hyperedge}. Also, we say that a vertex of $V$ is \emph{normal} if we have both $d_{\textrm{red}}(v) \geq (\tfrac{1}{2}-\omega')n$ and $d_{\textrm{blue}}(v) \geq (\tfrac{1}{2}-\omega')n$ in $K$. We observe that our colouring of $K$ has the following properties. 
\begin{enumerate}[label=(\Alph*)]
\item There are at most $|T_{x, y}| < \beta n^4$ red hyperedges of $H$,
\item at least $(\tfrac{1}{2}-\omega)\binom{n}{2}$ edges of $K$ are coloured red,
\item at least $(\tfrac{1}{2}-\omega)\binom{n}{2}$ edges of $K$ are coloured blue,
\item at most $2\omega\binom{n}{2}$ edges of $K$ are uncoloured, and
\item all but at most $\omega'n$ vertices are normal.
\end{enumerate}
Indeed,~(A) follows immediately from our definition of a red hyperedge and our choice of~$e$. For~(B) note that there are at least $\tfrac{1}{2}(n-2)\delta(H)$ pairs $\{a, b\}$ for which $\{x, y, a, b\}$ is an edge of $H$. By choice of $e$ at most $\frac{\omega}{2}\binom{n}{2}$ of these pairs have $\{x', y', a, b\} \in E(H)$, whilst at most $4n$ such pairs have $a \in e$ or $b \in e$, so at least  $\tfrac{1}{2}(n-2)\delta(H)-\frac{\omega}{2}\binom{n}{2}-4n\geq(\tfrac{1}{2}-\omega)\binom{n}{2}$ edges of $K$ are coloured red. Essentially the same argument proves~(C), and~(D) follows immediately from~(B) and~(C). Finally, observe that for any $v \in V\sm e$ we have $d_H(\{x,y,v\}),d_H(\{x',y',v\})\geq (\tfrac{1}{2}-\eps)n$, so for any $v \in V$ we have $d_{\red}(v),d_{\blue}(v)\leq (\tfrac{1}{2}+\eps)n$. So if there are more than $\frac{\omega'}{2}n$ vertices with $d_{\textrm{red}}(v) < (\tfrac{1}{2}-\omega')n$, then the total number of red edges is at most $\tfrac{1}{2}(\frac{\omega'}{2}n\cdot(\tfrac{1}{2}-\omega')n+(1-\frac{\omega'}{2})n\cdot(\tfrac{1}{2}+\eps)n)<(\tfrac{1}{2}-\omega)\binom{n}{2}$, contradicting~(B). A similar argument shows that there cannot be more than $\frac{\omega'}{2}n$ vertices with $d_{\textrm{blue}}(v) < (\tfrac{1}{2}-\omega')n$, establishing~(E). 

For any triangle $T$ in $K$ with vertex set $\{u, v, w\}$, we write $N_\red(T)$ (respectively $N_\blue(T)$) for the set of vertices $x \in V$ for which each of $ux, vx$ and $wx$ is a red (respectively blue) edge of $K$. We then have the following claim.

\begin{claim}\label{clm:abscase1}
If the number of red triangles in $K$ is at least $\vphi n^3$, then there exists a red triangle $T$ in $K$ with vertex set $\{u,v,w\}$ such that
\begin{enumerate}[label=(\alph*)]
\item $|N_{\blue}(T)|>(\tfrac{1}{2}-\vphi) n$ and $|N_{\red}(T)|>(\tfrac{1}{2}-\vphi) n$,
\item at most $\vphi n$ vertices $x\in N_{\red}(T)$ have $|N_H(v,w,x)\cap N_{\red}(T)| > \vphi n$, and
\item there are at most $9\vphi \binom{n}{2}$ red edges of $K$ between $N_{\red}(T)$ and $V\sm N_{\red}(T)$.
\end{enumerate}
\end{claim}

\begin{proof}
By~(E) there are at most $\omega'n\cdot n^2<\frac{\vphi}{3} n^3$ triangles in $K$ which contain a vertex which is not normal, and by~(A) there are at most $\frac{\vphi}{3} n^3$ triangles in $K$ which are contained in more than $\omega' n$ red hyperedges, as otherwise there would be at least $\frac{1}{4}\cdot\frac{\vphi}{3} n^3\cdot \omega' n> \beta n^4$ red hyperedges in total. Similarly, by~(A) all but at most $\omega' n^2$ edges of $K$ are contained in at most~$\omega' n^2$ red hyperedges, as otherwise there would be at least $\frac{1}{6}\cdot \omega' n^2\cdot \omega' n^2>\beta n^4$ red hyperedges in total. We call these edges of $K$ {\it normal}; observe that at most $\omega' n^2 \cdot n <\frac{\vphi}{3} n^3$ triangles in $K$ contain an edge which is not normal. So we can choose a red triangle $T$ in $K$ with vertex set $\{u,v,w\}$ which contains only normal vertices and normal edges and which is contained in at most~$\omega' n$ red hyperedges. Then $|N_H(T) \cap N_{\red}(x)| \leq \omega' n$ for each $x \in \{u,v,w\}$, as otherwise $\{u,v,w\}$ would be contained in more than $\omega' n$ red hyperedges. Also, since $u$, $v$ and $w$ are normal, for each $x \in \{u, v, w\}$ at most $2\omega' n$ vertices are in neither $N_{\red}(x)$ nor $N_{\blue}(x)$. It follows that for each $x \in \{u, v, w\}$ at most $3\omega' n$ of the vertices inside $N_H(T)$ are not in $N_{\blue}(x)$, so $|N_{\blue}(T)| \geq |N_H(T)| - 3 \cdot 3 \omega ' n \geq \delta(H) - 9 \omega' n > (\tfrac{1}{2}-\vphi) n$. Similarly, for each $x\in \{u,v,w\}$ we have 
$|(V \sm N_H(T)) \cap N_{\red}(x)| \geq d_{\red}(x) - \omega' n \geq (\tfrac{1}{2} - 2 \omega') n$; since $|V \sm N_H(T)| \leq n-\delta(H) \leq (\tfrac{1}{2}+\eps) n$ it follows that 
$|N_{\red}(T)| \geq (\tfrac{1}{2}+\eps) n - 3 \cdot (2\omega' + \eps) n \geq (\tfrac{1}{2}-\vphi) n,$ proving~(a).

Now observe that if~(b) does not hold, then there are at least $\frac{1}{2} (\vphi n)^2 >\omega' n^2$ pairs $\{x, y\}$ such that $x \in N_{\red}(T)$ and $y \in N_H(v,w,x) \cap N_{\red}(T)$. Each such pair yields a red hyperedge  $\{v, w, x, y\}$ containing $\{v, w\}$, contradicting the fact that $\{v,w\}$ is normal. So~(b) holds, and together with~(a) and inclusion-exclusion we find that all but at most $\vphi n$ vertices $x \in N_{\red}(T)$ have $|N_{\red}(T) \cup N_H(v, w, x)| \geq (\tfrac{1}{2}- \vphi) n + \delta(H) - \vphi n \geq n - 3\vphi n$. So there are at most $4\vphi n^2$ pairs $\{x, y\}$ with $x \in N_\red(T)$ and $y \in V \sm (N_{\red}(T) \cup N_H(v, w, x))$. On the other hand, any red edge $\{x, y\}$ of $K$ with $x\in N_{\red}(T)$ and $y \in N_H(v,w,x)$ yields a red hyperedge $\{v, w, x, y\}$ containing $\{v, w\}$, so there are at most $\omega' n^2$ such edges. It follows that $K$ has at most $4 \vphi n^2 + \omega' n^2 \leq 9 \vphi \binom{n}{2}$ red edges between $N_{\red}(T)$ and $V\sm N_{\red}(T)$, proving~(c).
\end{proof}

Suppose first that there exists a red triangle $\{u, v, w\}$ in $K$ with the properties given in Claim~\ref{clm:abscase1}, and define $A := N_{\red}(T)$ and $B := V \sm A$. Then by Claim~\ref{clm:abscase1}(a) we have $(\tfrac{1}{2} - \vphi)n \leq |A|, |B| \leq (\tfrac{1}{2} +\vphi)n$, so certainly we have $(\tfrac{1}{2} - c) n \leq |A|, |B| \leq (\tfrac{1}{2} + c)n$. 
Furthermore, by~(B) and Claim~\ref{clm:abscase1}(c) there are at least $(\tfrac{1}{2}-\omega)\binom{n}{2} - 9\vphi \binom{n}{2}>(\tfrac{1}{2}-10\vphi)\binom{n}{2}$ red edges in $K[A] \cup K[B]$. Together with the fact that there are 
$\binom{|A|}{2} + \binom{|B|}{2} \leq \tbinom{(1/2+\vphi)n}{2}+\tbinom{(1/2-\vphi)n}{2} \leq (\tfrac{1}{2} + \vphi) \binom{n}{2}$ edges in $K[A] \cup K[B]$, this implies that there are at most 
$11\vphi\binom{n}{2}$
edges in $K[A] \cup K[B]$ which are not red. So if there are more than $c\binom{n}{4}$ edges $e\in E$ such that $|e\cap A|$ is even, then there are at least $c\binom{n}{4}-11 \vphi \binom{n}{2} \cdot (\tfrac{1}{2} + \vphi)\binom{n}{2} > \beta n^4$ red hyperedges, contradicting~(A). We may therefore conclude that there are at most $c\binom{n}{4}$ edges $e\in E$ such that $|e\cap A|$ is even, whereupon $\{A, B\}$ is the desired $c$-odd-extremal bipartition of $V$.

Now assume that such a red triangle does not exist. We may then by Claim~\ref{clm:abscase1} assume that there are at most $\vphi n^3$ red triangles in $K$. So we may choose a normal vertex $v\in V$ which is contained in at most $4\vphi n^2$ red triangles, as otherwise $K$ would have at least $\tfrac{1}{3}(1-\omega')n\cdot 4\vphi n^2>\vphi n^3$ red triangles in total. There are then at most $4 \vphi n^2$ red edges inside $N_{\textrm{red}}(v)$, so by~(E) and the fact that $v$ is normal there are at least
\[\sum_{u\in N_{\textrm{red}}(v)}d_{\textrm{red}}(u)- 2\cdot 4\vphi n^2\geq (\tfrac{1}{2}-2\omega')n\cdot (\tfrac{1}{2}-\omega')n- 8\vphi n^2>(\tfrac{1}{2}- 17\vphi)\binom{n}{2}\]
red edges between $N_{\textrm{red}}(v)$ and $V\sm N_{\textrm{red}}(v)$, so at most $\frac{n^2}{4}-(\tfrac{1}{2}- 17\vphi)\binom{n}{2} \leq 18 \vphi \binom{n}{2}$ edges between $N_{\textrm{red}}(v)$ and $V\sm N_{\textrm{red}}(v)$ are not red. It follows that at most $\eps_2\binom{n}{4}$ edges $f \in E(H)$ have $|N_{\textrm{red}}(v)\cap f|=2$, as otherwise there would be at least $\eps_2\binom{n}{4} - 18\vphi \binom{n}{2}\cdot \frac{n^2}{4}>\beta n^4$ red hyperedges in total, contradicting~(A). So if $H[N_{\textrm{red}}(v)]$ and $H[V\sm N_{\textrm{red}}(v)]$ each contain at most $\eps_1 \binom{n}{4}$ edges of $H$, then the total number of edges $f \in E(H)$ for which $|f \cap N_{\textrm{red}}(v)|$ is even is at most $(2\eps_1+\eps_2)\binom{n}{4} \leq c \binom{n}{4}$, and then taking $A := N_{\textrm{red}}(v)$ and $B := V \sm A$ gives the desired $c$-odd-extremal bipartition $\{A, B\}$ of $V$, as $(\tfrac{1}{2}-c)n\leq|N_{\textrm{red}}(v)|\leq(\tfrac{1}{2}+c)n$ since $v$ is normal. 

This leaves only the cases in which either $H[N_{\textrm{red}}(v)]$ or $H[V\sm N_{\textrm{red}}(v)]$ contains more than $\eps_1\binom{n}{4}$ edges of $H$. If the former holds then we set $A := N_{\textrm{red}}(v)$ and $B := V\sm N_{\textrm{red}}(v)$, and otherwise we set $A := V\sm N_{\textrm{red}}(v)$ and $B := N_{\textrm{red}}(v)$; either way this results in a bipartition $\{A, B\}$ of $V$ such that
\begin{enumerate}[label=(\Alph*)]
\setcounter{enumi}{5}
\item $(\tfrac{1}{2}-\omega')n \leq |A|, |B| \leq (\tfrac{1}{2}+\omega')n$ (since $v$ is normal),
\item at least $\eps_1 \binom{n}{4}$ edges $f\in E(H)$ have $|f \cap A|=4$, and 
\item at most $\eps_2 \binom{n}{4}$ edges $f\in E(H)$ have $|f\cap A| = 2$.
\end{enumerate}
In the remaining part of the proof we show that these conditions imply that at most $\alpha\binom{n}{2}$ pairs of vertices are not $\beta$-absorbable. This contradicts our assumption that $H$ is not $(\alpha,\beta)$-absorbing, and so completes the proof. Recall that a pair is \emph{split} if it has one vertex in $A$ and one in $B$, and \emph{connate} otherwise. Additionally, we say that a pair $p \in \binom{V}{2}$ is \emph{good} if there are at least $(\frac{1}{2}-\psi) \binom{n}{2}$ pairs $p' \in \binom{V}{2}$ such that $p \cup p'$ is an odd edge of $H$ (so a split pair is good if it forms an edge with most connate pairs, and a connate pair is good if it forms an edge with most split pairs).
\begin{claim} \label{absclaim} 
At most $\psi \binom{n}{2}$ pairs in $\binom{V}{2}$ are not good.
\end{claim}

\begin{proof}
First consider a split pair $p \in \binom{V}{2}$ which is not good. Then there are at least $\frac{1}{2}(n-2)\delta(H) \geq (\frac{1}{2} - 2\eps) \binom{n}{2}$ pairs $p'$ for which $p \cup p'$ is an edge of $H$. Since $p$ is not good it follows that $p$ is contained in at least $(\psi - 2\eps)\binom{n}{2} \geq \frac{\psi}{2}\binom{n}{2}$ even edges of $H$, each of which must have precisely two vertices in $A$ (because $p$ is a split pair). So at most $\frac{\psi}{3}\binom{n}{2}$ split pairs in $\binom{V}{2}$ are not good, as otherwise in total there would be at least $\frac{1}{6} \cdot \frac{\psi}{3}\binom{n}{2} \cdot \frac{\psi}{2}\binom{n}{2} > \eps_2 \binom{n}{4}$ edges of~$H$ with precisely two vertices in $A$, contradicting~(H). 

Next observe that for all but at most $\frac{\psi}{3} \binom{n}{2}$ pairs $p \in\binom{B}{2}$ there are at most $\frac{\psi}{3}n$ vertices $w \in A$ such that $|N_H(p \cup \{w\}) \cap A| > \frac{\psi}{3}n$, as otherwise there would be at least $\tfrac{1}{12} \cdot \frac{\psi}{3} \binom{n}{2} \cdot \frac{\psi}{3}n \cdot \frac{\psi}{3}n > \eps_2 \binom{n}{4}$ edges of $H$ with precisely two vertices in $A$, contradicting~(H). For each such $p$ there are at least 
$(|A|-\frac{\psi}{3}n) (\delta(H) - \frac{\psi}{3}n) \geq (\frac{1}{2} - \omega' - \frac{\psi}{3})(\frac{1}{2} - \eps - \frac{\psi}{3})n^2 > (\tfrac{1}{2} - \psi) \binom{n}{2}$ 
split pairs $p'$ such that $p \cup p'$ is an edge of $H$; in other words, each such $p$ is good. The same argument with the roles of $A$ and $B$ reversed shows that all but at most $\frac{\psi}{3} \binom{n}{2}$ pairs $p \in\binom{A}{2}$ are good.
\end{proof}

We now show that any good pair is $\beta$-absorbable. First consider any good pair $\{x, y\}$ with $x \in A$ and $y \in B$. Since $\{x, y\}$ is good, at most $\binom{|A|}{2} + \binom{|B|}{2} - (\frac{1}{2} - \psi) \binom{n}{2} \leq 2 \psi \binom{n}{2}$ pairs $p' \in \binom{A}{2}$ are not in $N_H(\{x, y\})$, and so at most $2 \psi \binom{n}{2} \cdot \binom{n}{2}$ sets $S \in \binom{A}{4}$ contain such a pair $p'$. By~(G) it follows that at least $\eps_1\binom{n}{4} - \psi \binom{n}{2}^2 \geq \beta n^4$ edges $\{a, b, c, d\} \in H[A]$ are such that $\{x, y, a, b\}$ and $\{x, y, c, d\}$ are both edges of $H$. Each such edge yields an element of~$T_{x, y}$, so $|T_{x, y}| \geq \beta n^4$, and so by Claim~\ref{clm:abs1} $\{x,y\}$ is $\beta$-absorbable. 

Now consider any good pair $\{x, y\}$ with $x, y \in A$ or $x, y \in B$. Choose (not necessarily distinct) vertices $a_1, c_1, c_2, c_3, b_1, b_2 \in A$ and $a_2, c_4 \in B$ uniformly at random, and observe that with probability at least $1 - \frac{100}{n}$ these eight vertices are distinct from each other and from $x$ and $y$. Furthermore, since $\{x, y\}$ is good and $|A||B| \leq \frac{n^2}{4}$ we have that
$$\Prob(\{x, y, a_1, a_2\} \in E(H)) \geq \frac{1}{|A||B|} \cdot \left(\frac{1}{2} - \psi\right)\binom{n}{2} \geq 1 - 3 \psi,$$ 
and likewise the probability that $\{x, y, c_1, c_4\}$ is an edge of $H$ is at least $1-3\psi$. Also, 
\begin{align*} 
&\Prob(\{a_1, a_2, c_1, c_2\} \in E(H)) \\
\geq &\Prob(\mbox{$\{c_1, c_2\}$ is good})\cdot \Prob(\{a_1, a_2, c_1, c_2\} \in E(H) \mid \mbox{$\{c_1, c_2\}$ is good}) \\ 
\geq &\frac{2}{|A|^2} \left(\binom{|A|}{2} - \psi\binom{n}{2}\right) \cdot  \frac{1}{|A||B|} \left(\frac{1}{2} - \psi\right)\binom{n}{2} \geq 1 - 8\psi,
\end{align*}
where we use~(F) for the final inequality. Exactly the same calculation shows that the probabilities that $\{c_1, c_2, c_3, c_4\}$ and $\{c_3, c_4, b_1, b_2\}$ are edges of $H$ are each at least $1 - 8 \psi$. Finally, by~(G) the probability that $\{c_2, c_3, b_1, b_2\}$ is an edge of $H$ is at least $\frac{4!}{|A|^4} \cdot e(H[A]) \geq \eps_1$. Taking a union bound we find that with probability at least $\eps_1 - 30\psi - \frac{100}{n} \geq \frac{\eps_1}{2}$ all of these events occur, in which case $(a_1, a_2, c_1, c_2, c_3, c_4, b_1, b_2)$ is an absorbing structure of type 2 for $\{x, y\}$. So in total there are at least $\frac{\eps_1}{2} |A|^6|B|^2 \geq \beta n^8$ such absorbing structures for $\{x, y\}$, so $\{x, y\}$ is $\beta$-absorbable.

We conclude by Claim~\ref{absclaim} that at most $\psi \binom{n}{2} < \alpha n^2$ pairs in $\binom{V}{2}$ are not $\beta$-absorbable. This contradicts our assumption that $H$ is not $(\alpha, \beta)$-absorbing and so completes the proof.

Finally, note that each step of the proof directly translates to an algorithm which returns the desired bipartition in the claimed running-time.
\end{proof}

We are now ready to prove our absorbing lemma (Lemma~\ref{lem:absorbing}) which guarantees the existence of an absorbing path in the non-extremal case. In fact, we actually prove the following stronger statement, in which the assumption that $H$ is non-extremal is replaced by the assumption that $H$ is absorbing and connecting, and which concludes that we can find an absorbing path $P$ in polynomial time (furthermore, the modified condition~(iii) allows the absorption of pairs into $P$ to be done greedily). Since Lemmas~\ref{lem:con} and~\ref{lem:absstruc} imply that a non-extremal graph $H$ must be absorbing and connecting, this is indeed a stronger statement than Lemma~\ref{lem:absorbing}. We could instead prove Lemma~\ref{lem:absorbing} directly by a standard random selection argument; we avoid this approach since we will also use the polynomial-time algorithm given by Lemma~\ref{lem:abspath2} in the proof of Theorem~\ref{2cycle}.

\begin{lemma}\label{lem:abspath2}
Suppose that $1/n\ll\eps\ll \gamma\ll \beta\ll\alpha\ll\lambda\ll \kappa,\mu$. Let $H$ be a $4$-graph on $n$ vertices with $\delta(H) \geq n/2 - \eps n$ which is $\kappa$-connecting and $(\alpha,\beta)$-absorbing. Then there is a path~$P$ in $H$ and a graph $G$ on $V(H)$ with the following properties.
\begin{enumerate}[label=(\roman*)]
\item $P$ has at most $\mu n$ vertices.
\item Every vertex of $V(H) \sm V(P)$ is contained in at least $n-\lambda n$ edges of $G$. 
\item For any edge $e$ of $G$ which does not intersect $V(P)$ there are at least $2\gamma n$ vertex-disjoint segments of $P$ which are absorbing structures for $e$.
\end{enumerate}
Furthermore, there exists an algorithm, Procedure~\ref{proc:abspath}($H$), which returns such a path~$P$ and graph $G$ in time $O(n^{32})$.
\end{lemma}

\begin{proof}
Let $W:=V(H)^8$, let $U$ be the set of all $\beta$-absorbable pairs of vertices of $H$, and define the graph $G:=(V(H), U)$. Furthermore set $V_1:=U\cup W$ and let $E_1$ be the set of all pairs $\{p,T\}$ with $p \in U$, $T \in W$ for which $T$ is an absorbing structure for $p$. We then set $E_1':=\{\{T,T'\} : T, T' \in W \mbox{ and } T\cap T'\neq\emptyset\}$ and define the graph $G_1:=(V_1,E_1\cup E_1')$. Set $M := |U|$, so $M \geq \binom{n}{2} - \alpha n^2 > (1-3\alpha)\binom{n}{2}$ since $H$ is $(\alpha, \beta)$-absorbing, and set $N := |W| = n^8$, $m := |E_1'| < 64n^{15}$, $r=\beta^2 n$ and $\nu = \frac{2mr}{N^2} < 128\beta^2$. Then, taking $G_2$ to be the empty graph on vertex set $W$, the conditions of Setup~\ref{setup:derand} and Proposition~\ref{prop:selset} are satisfied (with $\beta$ playing the same role here as there, and with $1$ and $\beta^2$ in place of $\lambda$ and $\tau$ respectively), since each pair $p\in U$ is $\beta$-absorbable and so has $d_{G_1}(p)\geq \beta n^8$. The call of Procedure~\ref{proc:SelSet}$(G_1,\emptyset,\beta^2 n,\beta,1,\beta^2)$ then returns a family $\T' \subseteq W$ of ordered octuples of vertices of $H$ which is an independent set in $G_1$ such that $(1-128\beta^2)\beta^2 n \leq |\T'|\leq \beta^2 n$ and $|\T' \cap N_{G_1}(p)| \geq(\beta-129\beta^2) |\T'| > \frac{1}{2}\beta^3 n$ for each $p\in U$. If we now delete from $\T'$ every $T \in \T'$ which is not an absorbing structure for some $\beta$-absorbable pair $\{x, y\} \in \binom{V(H)}{2}$, then we obtain a pairwise-disjoint family $\T$ satisfying the following properties:
\begin{enumerate}[label=(\Alph*)]
\item $|\T| \leq \beta^2 n$,
\item for any $\beta$-absorbable pair $\{x, y\} \in \binom{V(H)}{2}$ the family $\T$ contains at least $\tfrac{1}{2}\beta^3 n$ absorbing structures for $\{x, y\}$, and
\item every $T \in \T$ is an absorbing structure for some $\beta$-absorbable pair $\{x, y\} \in \binom{V(H)}{2}$.
\end{enumerate} 

Enumerate the members of $\T$ as $T_1,\cdots,T_q$, so $q\leq \beta^2n$ by~(A), and for each $1 \leq i \leq q$ let $a_1^i$,~$a_2^i$,~$b_1^i$ and $b_2^i$ be the first, second, seventh and eighth elements of $T_i$ respectively. Then by~(C) and the definition of an absorbing structure we may choose, for each $i$, a path $P_i$ in $H$ with vertex set $T_i$ and with ends $\{a_1, a_2\}$ and $\{b_1, b_2\}$. Let $Q = \bigcup_{i=1}^q T_i$, so $|Q| = 8q \leq 8 \beta^2 n$, and let $X \subseteq V \sm Q$ be the set of vertices  not in $Q$ which lie in fewer than $(1-\lambda) n$-many $\beta$-absorbable pairs. We must have $|X|\leq\lambda n$, as otherwise there would be at least $\tfrac{1}{2} \cdot \lambda n \cdot (\lambda n-1) > \alpha \binom{n}{2}$ pairs in $H$ which are not $\beta$-absorbable, contradicting the fact that~$H$ is $(\alpha, \beta)$-absorbing. 

We now greedily construct a path $P_0$ containing every vertex of $X$. For this write $X = \{x_1, \dots, x_t\}$, so $t = |X| \leq \lambda n$, and greedily choose distinct vertices $y_1, \dots, y_t \in V \sm (Q \cup X)$ such that $\{x_{i-1}, y_{i-1}, x_i, y_i\}$ is an edge of $H$ for each $2 \leq i \leq t$. This is possible since $y_1$ can be any vertex of $V \sm (Q \cup X)$, and when choosing $y_i$ for $2 \leq i \leq t$ there are at least $d_H(x_{i-1}, y_{i-1}, x_i) - |Q| - |X| - (i-1) \geq \delta(H) - 8\beta^2 n - \lambda n - t > 0$ suitable choices available. Having done so, we let $P_0$ be the path $(x_1, y_1, \dots, x_t, y_t)$ in $H$, and set $b_0^1 := x_t$ and $b_0^2 := y_t$. Note that $P_0$ has $2t \leq 2 \lambda n$ vertices and that the paths $P_0, P_1, \dots, P_q$ are pairwise vertex-disjoint.

To complete the proof we use the fact that $H$ is $\kappa$-connecting to greedily choose paths $Q_i$ of length at most three which join the paths $P_i$ together into a single path. Suppose for this that we have already chosen paths $Q_1, \dots, Q_{i-1}$ for some $1 \leq i \leq q$, and set $A_i := (Q \cup V(P_0) \cup \bigcup_{j=1}^{i-1}Q_j) \sm \{b^{i-1}_1, b^{i-1}_2, a^i_1, a^i_2\}$. Since $H$ is $\kappa$-connecting, there are either at least $\kappa n^2$ paths of length two or at least $\kappa n^4$ paths of length three in $H$ with ends $\{b^{i-1}_1, b^{i-1}_2\}$ and $\{a^i_1, a^i_2\}$. In either case, since $|A_i|\leq 8\beta^2 n + 2\lambda n + 4 q < \kappa n$, at least one of these paths does not intersect $A_i$. Arbitrarily choose such a path and call it $Q_i$. Having proceeded in this manner to find paths $Q_1, \dots, Q_q$ we define $P := P_0 Q_1P_1 \cdots Q_q P_q$ and observe that $P$ is a path in $H$. It remains only to show that $P$ has the desired properties. Indeed, as just shown $P$ has at most $\kappa n \leq \mu n$ vertices, so~(i) holds. For~(ii), recall that any edge $e$ of $G$ is a $\beta$-absorbable pair in $H$ and that by construction of $P_0$ the paths $P$ included every vertex which was in fewer than $(1-\lambda) n$-many $\beta$-absorbable pairs. Finally, by~(B) there are at least $\tfrac{1}{2}\beta^3n \geq 2\gamma n$ paths $P_i$ for which $T_i = V(P_i)$ is an absorbing structure for $e$, and these paths $P_i$ are vertex-disjoint segments of $P$. 
For the running time note that the dominant term is the call of the Procedure~\ref{proc:SelSet} with $O(N^4+MN^3)=O(n^{32})$.
\end{proof}

\subsection{Proof of the long cycle lemma (Lemma~\ref{lem:longcycle})}

Now we can turn to the proof of Lemma~\ref{lem:longcycle} for which we need the following result of Erd\H{o}s~\cite{Er64}. We say that a $k$-graph $H$ is \emph{$k$-partite} if its vertex set may be partitioned into \emph{vertex classes} $V_1, \dots, V_k$ such that $|e \cap V_i| = 1$ for every $e \in E(H)$ and every $i \in [k]$. We say that $H$ is \emph{complete $k$-partite} if additionally every set $e \subseteq V(H)$ such that $|e \cap V_i| = 1$ for every $i \in [k]$ is an edge of $H$.

\begin{thm}[\cite{Er64}]\label{thm:Erdoes}
Suppose that $1/n\ll d, 1/f, 1/k$. Let $F$ be a $k$-partite $k$-graph on $f$ vertices. If $H$ is a $k$-graph on $n$ vertices with $e(H) \geq d\binom{n}{k}$, then $H$ contains a copy of $F$. Moreover, such a copy can be found in time $O(n^k)$.
\end{thm}

Actually the original version of this theorem did not consider the running time, but this can be derived by a straightforward argument. First we restrict to a constant size subgraph $H'$ of $H$ whose density is similar to that of $H$, and then we find a copy of $F$ in $H'$ by exhaustive search. The existence of such a subgraph can be established by a simple probabilistic argument, and this argument can be derandomised to give an algorithm which finds a subgraph $H'$ with density at least as large as that of $H$.

We will also make use of the following observation of R\"odl, Ruci\'nski and Szemer\'edi~\cite{RoRuSz09} (they did not mention the running time, but this follows immediately from their proof).

\begin{thm}[\cite{RoRuSz09}]\label{thm:pathindensegraph}
Given $a>0$ and $k\geq 2$, every $k$-graph $F$ on $m$ vertices and with at least $a\binom{m}{k}$ edges contains a tight path on at least $am/k$ vertices. Moreover, such a path can be found in time $O(n^k)$.
\end{thm}

Note that deleting every other edge of a tight path on $s$ vertices in a $4$-graph yields a $2$-path on at least $s-1$ vertices, so for $k=4$ we may replace `tight path' by `path' (\emph{i.e.} 2-path) and $am/k$ by $am/k-1$ in the statement of Theorem~\ref{thm:pathindensegraph}. 

We are now ready to prove our long cycle lemma, Lemma~\ref{lem:longcycle}. Again we actually prove a stronger statement, Lemma~\ref{lem:longcycle2}, which assumes instead that $H$ is connecting and states that we can find the cycle $C$ in polynomial time. Since by Lemma~\ref{lem:con} any $4$-graph which is not even-extremal is connecting, this is indeed a stronger statement. Our proof of Lemma~\ref{lem:longcycle2} is based on the proof given by Karpi\'nski, Ruci\'nski and Szyma\'nska~\cite[Fact~4]{KaRuSz10} for tight cycles, which in turn was based on a similar lemma for tight cycles in $3$-graphs given by R\"odl, Ruci\'nski and Szemer\'edi~\cite[Lemma~2.2]{RoRuSz09}. The principal differences are that our minimum codegree assumption is weaker, and that our absorbing lemma also requires us to consider the auxiliary graph $G$, in which we must find a perfect matching among the vertices not used in $C$.

\begin{lemma}\label{lem:longcycle2}
Suppose that $1/n\ll \eps\ll\gamma\ll\lambda \leq \mu \ll \kappa$ and that $n$ is even.
Let $H=(V,E)$ be a $4$-graph of order $n$ with $\delta(H)\geq n/2-\eps n$ which is $\kappa$-connecting. 
Also let $P_0$ be a $2$-path in $H$ on at most $\mu n$ vertices, and let $G$ be a $2$-graph on $V$ such that each vertex $v \in V \sm V(P_0)$ has $d_G(v) \geq (1-\lambda) n$.
Then there exists a $2$-cycle $C$ in $H$ on at least $(1-\gamma) n$ vertices such that $P_0$ is a segment of $C$ and $G[V\sm V(C)]$ contains a perfect matching. Moreover, there exists an algorithm, Procedure~\ref{proc:longcycle}($H,G,P_0$), which returns such a cycle $C$ in time $O(n^{16})$.
\end{lemma}

\begin{proof}
First, we introduce a constant $D> 0$ such that
$$\tfrac{1}{n} \ll \tfrac{1}{D} \ll \eps \ll \gamma \ll \lambda \leq \mu \ll \kappa \;.$$
Set $V' = V \sm V(P_0)$, $H'=H[V']$, $G'=G[V']$ and $n' = |V'|$, so $n' \geq (1-\mu)n$, $\delta(G') \geq (1-\lambda-\mu)n \geq (1-2\mu)n'$ and $\delta(H') \geq (\frac{1}{2}-\eps-\mu)n \geq (\frac{1}{2}-2\mu)n'$. Also it follows from the definition of $\kappa$-connecting that $H'$ is $\frac{\kappa}{2}$-connecting. So by Lemma~\ref{lem:reservoir} (with $2\gamma/3, n', \kappa/2$ and $2\mu$ in place of $\rho, n, \kappa$ and $\lambda$ respectively) we can choose a set $R \subseteq V'$ with $\frac{3\gamma}{5}n \leq (1-\frac{8}{3}\gamma)\frac{2\gamma}{3} n' \leq |R| \leq \frac{2\gamma}{3} n' \leq \frac{2\gamma}{3} n$ such that for any $x\in V'$ we have$|N_{G'}(x)\cap R| \geq (1-70\mu)|R| \geq \frac{4\gamma}{7}n$ and for every disjoint $p_1,p_2\in\binom{V'}{2}$ there are at least $\frac{\kappa}{10}|R| \geq \frac{\kappa\gamma}{20}n$ internally disjoint paths in $H'[R\cup p_1\cup p_2]$ of length at most three whose ends are $p_1$ and $p_2$. 

We first extend $P_0$ to a path $P_0'$ in $H$ by adding a single edge at each end. The purpose of this is that the ends of $P_0'$ will then be pairs in $H'$ to which we can apply the fact that $H'$ is $\kappa$-connecting. So let $\{u_1,u_2\}$ and $\{u_3,u_4\}$ be the ends of $P_0$. Then since $|V'\sm R|\geq (1-2\mu)n$ and $\delta(H)\geq (\frac{1}{2}-\eps)n$ we may choose distinct vertices $u_1', u_2', u_3', u_4' \in V' \sm R$ such that $u_2' \in N(u_1, u_2, u_1')$ and $u_4' \in N(u_3, u_4, u_3')$. This gives the desired path $P_0'=(u_1',u_2')P_0(u_3',u_4')$. Write $q := \{u_1', u_2'\}$, so $q$ is an end of $P_0'$.

We next proceed by an iterative process to extend $P_0'$ to a path on at least $(1-\gamma)n$ vertices in $H$. At any point in this process we write $P$ for the path we have built so far (so initially we take $P = P_0'$), and write $L := V\sm (V(P)\cup R)$ and $R' := R \sm V(P)$ (so at any point the sets $V(P)$, $R'$ and $L$ partition $V$). Moreover, throughout the process $P_0'$ will be a segment of~$P$ which shares an end in common with $P$, namely $q$. If at any point in the process we have $|L| \leq \frac{\gamma}{3} n$ then we terminate; observe that we then have $|V(P)| \geq n - |L| - |R| \geq (1-\gamma)n$, so $P$ is the desired path. We may therefore assume throughout the process that $|L| > \frac{\gamma}{3} n$.

The first stage of the process is while $P$ contains at most $(\frac{1}{2}-\mu)n$ vertices, in which case we have $|V(P) \cup R| \leq (\frac{1}{2}-\mu)n + \frac{2\gamma}{3} n \leq \delta(H) - 2$. So we may use the minimum degree condition as before to extend $P$ by one edge. That is, let $p$ be the end of $P$ other than $q$, choose any vertex $u \in L$ and any vertex $v \in N(p \cup \{u\}) \cap L$, and extend $P$ by the edge $p \cup \{u, v\}$. We continue to extend $P$ in this way until $P$ has more than $(\frac{1}{2}-\mu)n$ vertices. Note that no vertices of $R$ are added to $P$ during this stage of the process.

Once $P$ contains more than $(\frac{1}{2}-\mu)n$ vertices, we enter the second stage of the process. In each step of the process during this stage we will add at most $8$ additional vertices from the reservoir set $R$ to the path $P$, and this stage of the process will continue for at most $\frac{3n}{D}$ steps. In consequence we will always have $|R'| \geq |R| - 8 \cdot \frac{3n}{D}$. 
Since $8 \cdot \frac{3n}{D} \leq \frac{\kappa\gamma}{20}n$ it follows from our choice of $R$ that at any point in the process (or immediately after the process terminates), given disjoint pairs $p$ and $p'$ in $V'$, we can find a path of length at most three in $H[p \cup p' \cup R']$ with ends $p$ and $p'$.

Suppose first that, at some step of the second stage of the process, we have $e(H[L]) > \mu\binom{|L|}{4}$. Then we can use Theorem~\ref{thm:pathindensegraph} to find a path $P'$ in $H[L]$ on at least $\frac{\mu|L|}{4}-1$ vertices. We let $p\in\binom{V}{2}$ be the end of $P$ other than $q$ and let $p'$ be an end of $P'$, and 
choose a path $Q$ of length at most three in $H[R' \cup p \cup p']$ with ends $p$ and $p'$. We then replace $P$ by $PQP'$ and proceed to the next iteration. Note that in this step we added at most $4$ vertices from $R'$ to $P$, and that the total number of vertices added to $P$ is at least $\frac{\mu|L|}{4}-1 \geq \frac{\mu\gamma}{12}n-1 \geq \frac{D}{3}$. 

Now suppose instead that, at some step of the second stage of the process, we have $e(H[L]) \leq \mu\binom{|L|}{4}$.  Then we have the following claim.

\begin{claim}\label{clm:segment}
There exist sets $J \subseteq I\subseteq V(P)\sm V(P_0')$ and a $3$-graph $H_0$ with vertex set $L$ such that $|I|=D$, $P[I]$ is a segment of $P$, $|J| = \frac{D}{3}$, $e(H_0) \geq 2^{-D}\frac{1}{7}\binom{|L|}{3}$ and for every $e\in E(H_0)$ and every $v\in J$ we have $e\cup\{v\}\in E(H)$.
\end{claim}

\begin{proof}
Let $E_0$ be the set of edges of $H[L]$, $E_1$ be the set of edges of $H$ with three vertices in $L$ and one vertex in $V(P)\sm V(P_0')$ and $E_1'$ be the set of edges of $H$ with three vertices in~$L$ and one vertex in $R' \cup V(P_0')$. We then have that
$$(\tfrac{1}{2}-\eps)n\binom{|L|}{3}\leq \sum_{S\in\binom{L}{3}}d_H(S) = 4|E_0|+|E_1|+|E_1'|\;.$$
Since $4 |E_0|\leq 4\mu\binom{|L|}{4} < \mu |L| \binom{|L|}{3} \leq \mu n \binom{|L|}{3}$ and $|E_1'| \leq (\frac{2\gamma}{3} n + \mu n + 4)\binom{|L|}{3}$, this yields
$$|E_1|\geq (\tfrac{1}{2}-3\mu)n\binom{|L|}{3}\;.$$
Let $\Part$ be the family of segments of $P$ with precisely $D$ vertices which do not intersect~$P_0'$, 
and for each $Q \in \Part$ let $N_Q$ be the number of edges of $E_1$ which intersect~$V(Q)$. Since all but at most $2D$ vertices of $V(P) \sm V(P'_0)$ appear in precisely~$\frac{D}{2}$ of the sets $V(Q)$, we have  
$$\sum_{Q \in \Part} N_Q \geq \left(|E_1| - 2D\binom{|L|}{3}\right) \cdot \frac{D}{2} \geq (\tfrac{1}{2}-4\mu) \binom{|L|}{3} D \cdot \frac{n}{2}$$
Since $P$ has at most $n$ vertices we have $|\Part| \leq \frac{n}{2}$, so we may fix a segment $Q \in \Part$ with $N_Q \geq (\frac{1}{2}-4\mu)\binom{|L|}{3}D$. Write $I := V(Q)$ and let $H_1$ be the $3$-graph  on $L$ whose edges are all sets $S \in \binom{L}{3}$ with $|N_H(S) \cap I| \geq \frac{D}{3}$. Then we have $N_Q \leq e(H_1) D + \binom{|L|}{3}\frac{D}{3}$, and it follows that $e(H_1) \geq (\frac{1}{2}-4\mu-\frac{1}{3})\binom{|L|}{3}\geq \frac{1}{7}\binom{|L|}{3}$. Also, since $I$ has at most $2^{D}$ subsets, by averaging we may fix a set $J' \subseteq I$ with $|J'| \geq \frac{D}{3}$ such that at least $2^{-D}e(H_1)$ edges $S \in E(H_1)$ have $N_{H_i}(S)=J$. Let $H_0$ be the $3$-graph on $L$ with all such sets $S$ as edges, and choose any $J \subseteq J'$ with $|J| = \frac{D}{3}$.
\end{proof}

Fix such an $I$, $J$ and $H_0$. Then $H_0$ contains a complete $3$-partite $3$-graph $K$ with all vertex classes of size $\frac{D}{3}$ by Theorem~\ref{thm:Erdoes}. So let $K'$ be the complete $4$-partite subgraph of $H$ whose vertex classes are $J$ and the three vertex classes of $K$, and let $Q$ be a Hamilton path in $K'$ (so in particular $Q$ is a path in $H$ on $\frac{4D}{3}$ vertices). Since $P[I]$ is a segment of~$P$, removing the vertices in $I$ from $P$ leaves two vertex-disjoint subpaths of $P$; of these let $P_1$ be the path which has $q$ as an end (so in particular $P_0'$ is a segment of $P_1$) and let $P_2$ be the other path. Let $p_1$ be the end of $P_1$ other than $q$, let $p_2$ be an end of $P_2$, and let $q_1$ and $q_2$ be the ends of $Q$, and choose vertex-disjoint paths $Q_1$ and $Q_2$ in $H[p_1 \cup q_1 \cup R']$ and $H[p_2 \cup q_2 \cup R']$ respectively, each of length at most three, so that $Q_1$ has ends $p_1$ and $q_1$ and $Q_2$ has ends $p_2$ and $q_2$. We now replace $P$ with the path $P_1Q_1QQ_2P_2$ and proceed to the next iteration. Note that in this step we added at most eight vertices from $R'$ to $P$ (at most four in each of $Q_1$ and $Q_2$), and that the total number of vertices in $P$ increased by at least $|V(Q)| - |I| = \frac{D}{3}$. 

Since each step in the second stage of the process increases the number of vertices of $P$ by at least $\frac{D}{3}$, this stage of the process continues for at most $\frac{3n}{D}$ steps, as claimed. When the process terminates the final path $P$ has at least $(1-\gamma)n$ vertices and has $q$ as an end, and $P_0'$ is a segment of $P$. Let $p$ be the end of $P$ other than $q$; then we may choose a path $Q$ in $H[R' \cup p \cup q]$ of length at most three and with ends $p$ and $q$. This gives a cycle $C = PQ$ in $H$ on at least $(1-\gamma)n$ vertices such that $P_0$ is a segment of~$C$. It remains only to find a perfect matching in $G^*:=G[V\sm V(C)]$. For this note that $|V(G^*)|\leq \gamma n$ and $|R \sm V(C)| \geq |R| - \frac{24n}{D} - 4$. Therefore our choice of $R$ implies that $\delta(G^*) \geq \frac{4\gamma}{7}n - \frac{24n}{D} - 4 \geq \frac{\gamma}{2}n \geq \frac{|V(G^*)|}{2}$.
Since $C$ is a $2$-cycle (so has an even number of vertices) and $n$ is even, we have that $|V\sm V(C)|$ is even, so $G^*$ contains a perfect matching.

Following this proof gives a polynomial-time algorithm to find a long cycle as in the statement in a $\kappa$-connecting $k$-graph of high minimum codegree. Indeed, Lemma~\ref{lem:reservoir} gives a reservoir set $R$ as required in time $O(n^{16})$, and by Theorem~\ref{thm:Erdoes} we may find the complete $3$-partite $3$-graph $K$ in time $O(n^3)$, whilst Theorem~\ref{thm:pathindensegraph} allows the choice of the path $P'$ in time $O(n^4)$, and it is clear that the remaining steps of the proof can be carried out in polynomial time (e.g. by exhaustive search to find a path of length at most three).
\end{proof}

\section{Hamilton 2-cycles in 4-graphs: even-extremal case}\label{sec:evenextr}

In this section we give a detailed proof of Lemma~\ref{thm:evenextr} which states that Theorem~\ref{character} holds for even-extremal $4$-graphs. As in the previous section, all paths and cycles we consider in this section are $2$-paths and $2$-cycles, therefore we will again omit the $2$ and speak simply of paths and cycles. Also, for most of the section we work within the following setup.

\begin{setup} \label{setup:evenextr}
Fix constants satisfying $1/n\ll\eps, c \ll \gamma \ll \beta \ll \beta_2 \ll \beta_1 \ll \rho \ll \mu \ll 1.$ Let $H$ be a $4$-graph of order $n$, and let $V = V(H)$. 
\end{setup}

Recall that if $H$ is an even-extremal $4$-graph on $n$ vertices, with a corresponding even-extremal bipartition $\{A, B\}$ of $V(H)$, then $H$ has very few odd edges. If $\delta(H)$ is close to $n/2$ then it follows from this that $H[A]$ and $H[B]$ are both very dense, and also that $H$ has very high density of edges with precisely two vertices in $A$. Furthermore recall that we call a pair $p\in \binom{V}{2}$ a split pair if $|p\cap A|=1$ and a connate pair otherwise.

One strategy for finding a Hamilton cycle in such an $H$ is as follows. 
We first find short paths $P$ and $Q$ in $H$ each joining a connate pair in $A$ to a connate pair in $B$. 
That is, the ends $p$ and $p'$ of $P$ and the ends $q$ and $q'$ of $Q$ have $p, q \in \binom{A}{2}$ and $p', q' \in \binom{B}{2}$. 
Moreover $P$ and $Q$ are chosen so that $A' := A \sm (V(P) \cup V(Q))$ and $B' := B \sm (V(P) \cup V(Q))$ each have even size. 
We then use the high density of $H[A]$ and $H[B]$ to find a Hamilton path $P_A$ in $H[A' \cup p \cup q]$ with ends $p$ and $q$ and a Hamilton path $P_B$ in $H[B' \cup p' \cup q']$ with ends $p'$ and $q'$. 
Together $P, P_A, Q$ and $P_B$ then form a Hamilton cycle in $H$. 
Another strategy for finding a Hamilton cycle in such an $H$ is to first find a short path~$P$ in~$H$ whose ends $p$ and $q$ are both split pairs such that $V' := V(H) \sm V(P)$ satisfies $|V' \cap A| = |V' \cap B|$. We then use the high density of edges of $H$ with precisely two vertices in $A$ to find a Hamilton path $P'$ in $H[V' \cup p \cup q]$ with ends $p$ and $q$ which consists of a sequence of split pairs. Together~$P'$ and~$P$ then form a Hamilton cycle. 
We give the necessary preliminaries for implementing this strategy in Subsection~\ref{sec:evenextremal22}, culminating in Lemma~\ref{lem:evenextrhampathsplit} which gives sufficient conditions for us to find~$P'$ as desired. We will also use Lemma~\ref{lem:evenextrhampathsplit} to find the paths $P_A$ and $P_B$ in the very dense case.

Finally, in Subsection~\ref{sec:evenextremalpf} we complete the proof of Lemma~\ref{thm:evenextr} by distinguishing various cases; in each case we apply one of the two strategies described above to find a Hamilton cycle in $H$.

\subsection{Hamilton paths of split pairs} \label{sec:evenextremal22}

In this subsection we consider $4$-graphs $H$ admitting a bipartition $\{A, B\}$ of $V(H)$ such that $H$ has a very high density of edges with two vertices in $A$ and two vertices in $B$, motivating the following definitions.

\begin{definition} \label{def:good}
Under Setup~\ref{setup:evenextr}, for a fixed bipartition $\{A, B\}$ of $V$, we say that
\begin{enumerate}[label=(\roman*)]
\item a triple $S \in \binom{V}{3}$ is \emph{$\gamma$-good} if it is contained in at least $(\frac{1}{2}-\gamma)n$ even edges,
\item a pair $p \in \binom{V}{2}$ is \emph{$\gamma$-good} if it is contained in at least $(\frac{1}{2}-\gamma^3)\binom{n}{2}$ even edges,
\item a vertex $v\in V$ is \emph{$\gamma$-good} if it is contained in at least $(\frac{1}{2}-\gamma^5)\binom{n}{3}$ even edges,
\item a pair $p \in \binom{V}{2}$ is \emph{$\beta_2$-medium} if it is contained in at least $\beta_2\binom{n}{2}$ even edges,
\item a pair $p \in \binom{V}{2}$ is \emph{$\beta_2$-bad} if it is not $\beta_2$-medium,
\item a vertex $v\in V$ is \emph{$(\beta_1,\beta_2)$-medium} if it is contained in at least $\beta_1 n$-many $\beta_2$-medium pairs, and 
\item a vertex $v\in V$ is \emph{$(\beta_1,\beta_2)$-bad} if it is not $(\beta_1,\beta_2)$-medium.
\end{enumerate}
\end{definition}

The following elementary proposition shows that good vertices and pairs lie in many good pairs and triples.

\begin{prop}\label{clm:goodvpt}
Assume Setup~\ref{setup:evenextr}, and fix a bipartition $\{A, B\}$ of $V$ such that $n/2-cn\leq|A|\leq n/2+cn$. Then the following statements hold.
\begin{enumerate}[label=(\roman*)]
\item If $v\in V$ is $\gamma$-good, then $v$ is contained in at least $(1-\gamma)n$-many $\gamma$-good pairs.
\item If $p \in \binom{V}{2}$ is a $\gamma$-good pair, then $p$ is contained in at least $(1-\gamma)n$-many $\gamma$-good triples.
\end{enumerate}
\end{prop}

\begin{proof}
First, we make the following observation. Each set $S\subseteq V$ of size $i\in\{1,2,3\}$ is contained in at least $(\frac{1}{2}-2c)\binom{n}{4-i}$ and at most $(\frac{1}{2}+2c)\binom{n}{4-i}$ even $4$-sets. To prove~(i) note that since $v$ is $\gamma$-good, there are at most $(\frac{1}{2}+2c)\binom{n}{3}-(\frac{1}{2}-\gamma^5)\binom{n}{3}\leq(\gamma^5+2c)\binom{n}{3}$ even $4$-sets which contain $v$ and which do not form an edge in $H$. Now assume that there are more than $\gamma n$ vertices $w$ such that $\{v,w\}$ is not $\gamma$-good. Then there are at least $\frac{1}{3}\cdot\gamma n\cdot ((\frac{1}{2}-2c)\binom{n}{2}-(\frac{1}{2}-\gamma^3)\binom{n}{2})>(\gamma^5+2c)\binom{n}{3}$ even $4$-sets which contain $v$ and do not form an edge in $H$, a contradiction. A similar reasoning proves (ii).
\end{proof}

\begin{prop}\label{clm:evenconnect}
Assume Setup~\ref{setup:evenextr}, and fix a bipartition $\{A, B\}$ of $V$ such that $n/2-cn\leq|A|\leq n/2+cn$. Also let $R \subseteq V$ be such that $|A \sm R|\geq \mu n$ and $|B\sm R|\geq \mu n$. Then 
\begin{enumerate}[label=(\roman*)]
\item for any two disjoint $\gamma$-good split pairs $s_1$ and $s_2$ there exists a split pair $s_3 \in \binom{V\sm R}{2}$ such that $s_1 \cup s_3 \in E(H)$ and $s_3 \cup s_2 \in E(H)$, and
\item for any two disjoint $\gamma$-good connate pairs $p_1$ and $p_2$ there exists a connate pair $p_3 \in \binom{A\sm R}{2}$ such that $p_1 \cup p_3 \in E(H)$ and $p_3 \cup p_2 \in E(H)$.
\end{enumerate}
Now suppose additionally that at most $\rho n$ vertices of $V$ are not $\gamma$-good. Then
\begin{enumerate}[label=(\roman*)]
\setcounter{enumi}{2}
\item for any $\gamma$-good split pair $s_1$ there exists a $\gamma$-good split pair $s_2\in \binom{V\sm R}{2}$ such that $s_1\cup s_2\in E(H)$, and
\item for any $\gamma$-good connate pair $p_1$ there exists a $\gamma$-good connate pair $p_2 \in \binom{A\sm R}{2}$ such that $p_1\cup p_2\in E(H)$.
\end{enumerate}
Note that by symmetry~(ii) and~(iv) remain valid with $B$ in place of $A$.
\end{prop}

\begin{proof}
For~(i), since $s_1$ and $s_2$ are $\gamma$-good, there are at most $2(\frac{n^2}{4}-(\frac{1}{2}-\gamma^3)\binom{n}{2})<3\gamma^3\binom{n}{2}$ split pairs which do not form an edge with both $s_1$ and $s_2$. Since there are at least $\mu^2n^2 >3\gamma^3\binom{n}{2}$ split pairs which do not contain a vertex of $R$, we may choose $s_3$ as required. Similarly, for~(iii), there are at most $\frac{n^2}{4}-(\frac{1}{2}-\gamma^3)\binom{n}{2} < \frac{1}{2}\mu^2 n^2$ split pairs which do not form an edge with $s_1$, and since at most $\rho n$ vertices are not $\gamma$-good, by Proposition~\ref{clm:goodvpt}(i) the total number of pairs which are not $\gamma$-good is at most $\rho n \cdot n + n \cdot \gamma n < \frac{1}{2}\mu^2 n^2$, and so we may choose a split pair $s_2 \in \binom{V\sm R}{2}$ as required. The arguments for~(ii) and~(iv) are very similar, so we omit them.
\end{proof}

Note that we can list all $\gamma$-good pairs in $H$ in time $O(n^4)$, and having done so we can find pairs as in Proposition~\ref{clm:evenconnect}(i)-(iv) in time $O(n^2)$ by exhaustive search. We now state and prove our Hamilton path connecting lemma for this setting.

\begin{lemma}\label{lem:evenextrhampathsplit}
Assume Setup~\ref{setup:evenextr}, and fix a bipartition $\{A, B\}$ of $V$ with $|A| = |B|$. Suppose also that every vertex of $V$ is $\gamma$-good. If $s_1$ and $s_2$ are disjoint $\gamma$-good split pairs, then there exists a Hamilton path in $H$ with ends $s_1$ and $s_2$. Moreover, such a path can be found in time $O(n^4)$.
\end{lemma}

\begin{proof}
Set $m := \lceil \frac{(1-\rho)n}{8} \rceil$. 
Throughout this argument we only use edges $e\in E$ such that $|e\cap A|=|e\cap B|=2$. Consequently, every path we construct contains the same number of vertices from $A$ and~$B$. The proof consists of the following three steps.
\begin{enumerate}[label=(\Roman*)]
\item We define a notion of a palatable vertex, and build a grid $L \subseteq V \sm (s_1 \cup s_2)$ which consists of $3m+2$ vertices from $A$ and $3m+2$ vertices from $B$, contains two $\gamma$-good split pairs $q_1, q_2 \in \binom{L}{2}$ and which can swallow any set $S \subseteq V \sm L$ of $2m$ palatable vertices with $|S\cap A|=|S \cap B| = m$, meaning that for any such~$S$ there is a path in~$H$ with vertex set $L\cup S$ and with ends $q_1$ and $q_2$. 
\item Next, we construct disjoint paths $Q_1$ and $Q_2$ in $H[(V\sm L)\cup q_1\cup q_2]$ such that $Q_1$ has ends $s_1$ and $q_1$, $Q_2$ has ends $q_2$ and $s_2$, the set $R := V \sm (L \cup Q_1 \cup Q_2)$ satisfies $|R \cap A| = |R \cap B| = m$ and every vertex in $R$ is palatable.
\item Finally, since $L$ can swallow $R$, there is a path $P$ in $H$ with vertex set $R \cup L$ and with ends $q_1$ and $q_2$, and $Q_1PQ_2$ is then a Hamilton path in $H$ with ends $s_1$ and $s_2$.
\end{enumerate}
To construct the grid we first greedily choose a set $L_1'=\{x_1,y_1,x_2,y_2, \cdots,y_m,x_{m+1}\}$ of distinct vertices of $A \sm (s_1 \cup s_2)$ such that for each $i\in [m]$ both $\{x_{i},y_{i}\}$ and $\{y_{i},x_{i+1}\}$ are $\gamma$-good. 
We then greedily select distinct vertices $z_1,\cdots,z_{m+1}\in B \sm (s_1 \cup s_2)$ such that $\{x_1,z_1\}$ and $\{x_{m+1},z_{m+1}\}$ are $\gamma$-good pairs and such that for any $i\in[m]$ both $\{z_i,x_i,y_i\}$ and $\{y_i,x_{i+1},z_{i+1}\}$ are $\gamma$-good triples. 
Having done this, we set $L_1=L_1'\cup\{z_1,\cdots,z_{m+1}\}$ and continue to form a `mirror image' $L_2$ as follows. 
We greedily choose a set $L_2'=\{x_1',y_1',x_2',y_2', \cdots,y_m',x_{m+1}'\}$ of distinct vertices of $B\sm (L_1 \cup s_1 \cup s_2)$ such that $\{x_{m+1},z_{m+1},x_1'\}$ is $\gamma$-good and for each $i\in [m]$ both $\{x_{i}',y_{i}'\}$ and $\{y_{i}',x_{i+1}'\}$ are $\gamma$-good. 
Finally, we greedily select distinct vertices $z_1',\cdots,z_{m+1}'\in A \sm (L_1 \cup s_1 \cup s_2)$ such that $\{x_{m+1},z_{m+1},x_1',z_1'\} \in E(H)$, the pair $\{x_{m+1}',z_{m+1}'\}$ is $\gamma$-good and for any $i\in[m]$ both $\{z_i',x_i',y_i'\}$ and $\{y_i',x_{i+1}',z_{i+1}'\}$ are $\gamma$-good triples. 
Let $L_2 := L_2'\cup \{z_1',\cdots,z_{m+1}'\}$; our grid is then $L := L_1\cup L_2$, and we take $q_1 :=\{x_1,z_1\}$ and $q_2 :=\{x_{m+1}',z_{m+1}'\}$. To confirm that it is possible to make these greedy selections, observe that by the definition of a $\gamma$-good triple and Proposition~\ref{clm:goodvpt}(i) and~(ii) at least $|B| - 2\gamma n$ vertices are suitable for each choice of a vertex from $B$ and at least $|A| - 2 \gamma n$ vertices are suitable for each choice of a vertex from $A$. Since in total we choose $3m+2$ vertices from each of $A$ and $B$, and $|s_1 \cup s_2|=4$, there are always at least $n/2 - (3m+2) - 4 - 2\gamma n \geq 1$ suitable vertices which have not previously been chosen.

We now set $A'=A\sm L$ and $B'=B\sm L$, and define two bipartite auxiliary graphs $G_A$ and $G_B$. Indeed, we take $G_A$ to be the bipartite graph with vertex classes $A'$ and $Y':=\{y_1',\cdots, y_m'\}$ whose edges are those pairs $\{y_{i}',w\}$ with $i\in[m]$ and $w\in A'$ such that $\{x_{i}',z_i',y_i',w\}, \{y_i',w,x_{i+1}',z_{i+1}'\}\in E(H)$. By construction both $\{x_{i}',z_i',y_i'\}$ and $\{y_i',x_{i+1}',z_{i+1}'\}$ are $\gamma$-good for each $i \in [m]$, and it follows that $d_{G_A}(y'_i) \geq |A'|-2 \gamma n > (1-16\gamma)|A'|$. In particular the graph $G_A$ has at least $m(1-16\gamma)|A'|$ edges. Likewise we take $G_B$ to be the bipartite graph with vertex classes $B'$ and $Y:=\{y_1,\cdots, y_m\}$ whose edges are those pairs $\{y_{i},w\}$ with $i\in[m]$ and $w\in B'$ such that $\{x_{i},z_i,y_i,w\},\{y_i,w,x_{i+1},z_{i+1}\}\in E(H)$; by the same argument we have $d_{G_B}(y'_i) > (1-16\gamma)|B'|$ for each $i \in [m]$, and so in particular the graph $G_B$ has more than $m(1-16\gamma)|B'|$ edges.
We call a vertex $a\in A'$ palatable if $d_{G_A}(a)\geq 0.9m$ and a vertex $b\in B'$ palatable if $d_{G_B}(b)\geq 0.9m$. Let $M_A\subseteq A'$ be the set of non-palatable vertices in $A'$ and let $M_B\subseteq B'$ be the set of non-palatable vertices in $B'$. Then $|M_A|<\frac{\rho}{100} n$, as otherwise the number of edges in $G_A$ would be at most 
$$(|A'|-\frac{\rho}{100}n)m+\frac{\rho}{100}n\cdot 0.9m<m(1-16\gamma)|A'|\;,$$
contradicting our previous lower bound. For the same reason we have $|M_B|<\frac{\rho}{100} n$.
Observe that given a set $S\subseteq A'\cup B'$ of $2m$ palatable vertices with $|S\cap A'| = |S \cap B'| =m$ the subgraphs $G'_A := G_A[(A' \cap S) \cup Y']$ and $G'_B := G_B[(B' \cap S) \cup Y]$ contain perfect matchings $\{\{y_i', a_i\}\mid i\in[m]\}$ and $\{\{y_i, b_i\}\mid i\in[m]\}$ respectively, as $\delta(G'_A),\delta(G'_B)\geq 0.9 m$. It follows that 
$$(x_1,z_1,y_1,b_1,\cdots,y_m,b_m,x_{m+1},z_{m+1},x_1',z_1',y_1',a_1,\cdots,y_m',a_m,x_{m+1}',z_{m+1}')$$ 
is a path in $H$ with ends $q_1$ and $q_2$. This demonstrates that $L$ can swallow any set $S$ of $2m$ palatable vertices with $|S\cap A|= |S \cap B| = m$, and so completes Step~(I) of the proof.

We now construct $Q_1$ and $Q_2$. We first use Proposition~\ref{clm:evenconnect}(i) to find $s_1' \in \binom{V \sm (L \cup s_2)}{2}$ such that $Q_1 :=s_1s_1'q_1$ is a path. It then remains to construct a path $Q_2$ with ends $q_2$ and $s_2$, of length $\ell := \frac{n}{2}-4m-4$, with $V(Q_2) \subseteq V \sm (L\cup Q_1)$, so that $Q_2$ contains all non-palatable vertices not in $Q_1$ and such that $|V(Q_2)\cap A|=|V(Q_2)\cap B|$. We do this in the following way. Let $M$ be the set of all non-palatable vertices not in $Q_1$ or $s_2$, so $|M| \leq |M_A| + |M_B| \leq \frac{\rho}{50}n$, and write $M = \{g_1, \dots, g_k\}$ and $V' := V \sm (L \cup V(Q_1)\cup s_2)$. Now greedily choose distinct vertices $h_1,\cdots,h_k \in V' \sm M$ such that $\{g_i, h_i\}$ is a $\gamma$-good split pair for each $1\leq i\leq k$. This is possible as each $g_i$ is $\gamma$-good and $k \leq \frac{\rho}{50} n$. Write $p_0 := s_2$, and for each $i \in [k]$ let $p_{2i} := \{g_i, h_i\}$. By repeated application of Proposition~\ref{clm:evenconnect}(i) we may then choose split pairs $p_i \in \binom{V'}{2}$ for each odd $i \in [2k]$ such that the pairs $p_i$ are all disjoint and $p_{i-1}p_i \in E(H)$ for each $i \in [k]$. We then use Proposition~\ref{clm:evenconnect}(iii) repeatedly to obtain $\gamma$-good split pairs $p_{2k+1}, p_{k+2}, \dots, p_{\ell-2} \in \binom{V'}{2}$ which are pairwise-disjoint and disjoint from $p_0, \dots, p_{2k}$ and such that $p_{i-1}p_i \in E(H)$ for each $k+1 \leq i \leq \ell-2$. Finally, we set $p_\ell := q_2$ and apply Proposition~\ref{clm:evenconnect}(i) to choose a split pair $p_{\ell-1} \in \binom{V'}{2}$ disjoint from $\bigcup_{i=0}^{\ell-2} p_i$ so that $p_{\ell-2}p_{\ell-1} \in E(H)$ and $p_{\ell-1}p_\ell \in E(H)$, and $Q_2 := p_0p_1\dots p_\ell$ is then a path with the desired properties, completing Step~(II). In particular, the set $L \cup Q_1 \cup Q_2$ contains precisely $\frac{1}{2}(|L| + |V(Q_1) \sm L| + |V(Q_2) \sm L|) = (3m+2) + 2 + \ell = \frac{n}{2}- m$ vertices from each of $A$ and $B$, so $R := V \sm (L \cup Q_1 \cup Q_2)$ satisfies $|R \cap A| = |R \cap B| = m$, and therefore can be swallowed by $L$ for Step~(III). This calculation also justifies that it was possible to choose vertices $h_i$ and to use Proposition~\ref{clm:evenconnect} as claimed above.

Finally, for the algorithmic statement observe that for Step~(I) we can list the $\gamma$-good vertices, pairs and triples, greedily construct the grid $L$, form the auxiliary bipartite graphs $G_A$ and $G_B$ and list all palatable vertices in time $O(n^4)$. Then, for Step~(II), we can form the paths $Q_1$ and $Q_2$ in time $O(n^3)$ by repeated exhaustive search of which pair to add next. Finally, for Step~(III) we need to find a perfect matching in $G'_A$ and $G'_B$, and we can do so in time $O(n^4)$ by using Edmonds's blossom algorithm~\cite{Ed65}.
\end{proof}

\subsection{Proof of Lemma~\ref{thm:evenextr}} \label{sec:evenextremalpf}

We now turn to the proof of Lemma~\ref{thm:evenextr}, for which we first give two preliminary results. The first of these states that if $\{A, B\}$ is an even-extremal bipartition of $V$, then almost all vertices and pairs are good, and that pairs which are not bad must form an even edge with a good pair, even if we forbid a small number of vertices from being used for this.

\begin{prop}\label{clm:types}
Assume Setup~\ref{setup:evenextr}, and suppose that $\delta(H)\geq n/2-\eps n$ and that $\{A, B\}$ is a $c$-even-extremal bipartition of $V$. Then
\begin{enumerate}[label=(\alph*)]
\item there are at most $\frac{c}{\gamma^3-2\eps}\binom{n}{2}$ pairs which are not $\gamma$-good,
\item there are at most $\frac{c}{\gamma^5-2\eps}n$ vertices which are not $\gamma$-good,
\item there are at most $\frac{c}{(1-\beta_1)(\frac{1}{2}-\beta_2-2\eps)}n$ vertices which are $(\beta_1,\beta_2)$-bad, and
\item if $R\subseteq V$ satisfies $|R| \leq \frac{1}{3} \beta n$, then for every $\beta$-medium pair $p_1$ there exists a $\gamma$-good pair $p_2 \in \binom{V \sm R}{2}$ such that $p_1\cup p_2$ is an even edge.
\end{enumerate}
\end{prop}

\begin{proof}
For~(a) note that by our minimum degree condition every pair forms an edge with at least $(\frac{1}{2}-2\eps)\binom{n}{2}$ other pairs, so if the assertion is not true, then there are more than $\frac{1}{6}\cdot\frac{c}{\gamma^3-2\eps}\binom{n}{2}\cdot (\gamma^3-2\eps)\binom{n}{2}\geq c\binom{n}{4}$ odd edges. Similarly, for~(b) note that every vertex forms an edge with at least $(\frac{1}{2}-2\eps)\binom{n}{3}$ triples, so if the assertion is not true, then there are more than $\frac{1}{4}\cdot\frac{c}{\gamma^5-2\eps}n\cdot (\gamma^5-2\eps)\binom{n}{3}\geq c\binom{n}{4}$ odd edges. For~(c), if we assume otherwise, then there are more than $\frac{1}{12}\cdot\frac{c}{(1-\beta_1)(\frac{1}{2}-\beta_2-2\eps)}n\cdot(1-\beta_1)n\cdot(\frac{1}{2}-\beta_2-2\eps)\binom{n}{2}\geq c\binom{n}{4}$ odd edges in $H$. In each case we have a contradiction to the fact that the bipartition $\{A, B\}$ is $c$-even-extremal.
Finally, for~(d) note that by~(a) there are at most $\frac{c}{\gamma^3-2\eps}\binom{n}{2} < \frac{\beta}{4}\binom{n}{2}$ pairs which are not $\gamma$-good, and at most $\frac{\beta}{3} n^2 < \frac{3\beta}{4}\binom{n}{2}$ pairs contain a vertex of $R$, but since $p_1$ is $\beta$-medium there are at least $\beta \binom{n}{2}$ pairs which form an even edge with $p_1$.
\end{proof}

Recall the proof strategies outlined at the start of this section, both of which begin by choosing short paths whose ends we can connect by the Hamilton path connecting lemma. We will use the following lemma to obtain suitable short paths.

\begin{lemma}\label{lem:bridge}
Assume Setup~\ref{setup:evenextr}, and suppose that $\delta(H)\geq n/2-\eps n$ and that $\{A, B\}$ is a $c$-even-extremal bipartition $\{A, B\}$ of $V$. Suppose also that every vertex of $H$ is $(\beta_1, \beta_2)$-medium. If $H$ contains an odd edge, then $H$ contains a path of length $3$ whose ends are a $\gamma$-good split pair and a $\gamma$-good connate pair.
\end{lemma}

\begin{proof}
It suffices to show that $H$ contains an odd edge $e$ which is the union of two $\beta_2$-medium pairs. Indeed, we can then write $e = p \cup s$ where $p$ is a $\beta_2$-medium connate pair and $s$ is a $\beta_2$-medium split pair, following which two applications of Proposition~\ref{clm:types}(d) yield a $\gamma$-good connate pair $p'$ and a $\gamma$-good split pair $s'$ such that $p'pss'$ is a path. 

By assumption we may choose an odd edge $e \in H$. If $e$ is the union of two $\beta_2$-medium pairs, then we are done by our first observation. So we may assume that there are vertices $x, y \in e$ such that $\{x, y\}$ is $\beta_2$-bad. This means that $\{x, y\}$ is contained in at most $\beta_2\binom{n}{2}$ even edges, and so (using the minimum degree condition) there are at least $(\frac{1}{2} - 2\beta_2) \binom{n}{2}$ pairs $\{z, w\}$ such that $\{x, y, z, w\}$ is an odd edge of $H$. 
By Proposition~\ref{clm:types}(b) at most $\frac{c}{\gamma^5-2\eps}n\cdot n < \gamma\binom{n}{2}$ pairs $\{z, w\}$ contain a vertex which is not $\gamma$-good.
Furthermore, since $y$ is $(\beta_1,\beta_2)$-medium, at least $\beta_1 n$ pairs containing $y$ are $\beta_2$-medium. Therefore there are at least $\beta_1 n\cdot((\frac{1}{2} - c)n-\beta_1 n-1)>\frac{\beta_1}{2}\binom{n}{2}$ pairs $\{z, w\}$ for which $\{y, z\}$ and $\{y, w\}$ are not both $\beta_2$-bad. 
Since there are at most $(\frac{1}{2} +c) \binom{n}{2}$ pairs $\{z, w\}$ such that $\{x, y, z, w\}$ is an odd $4$-tuple, and $(\frac{1}{2} + c) - (\frac{1}{2} - 2\beta_2) + \gamma < \frac{\beta_1}{2}$, it follows that there exist $\gamma$-good vertices $z$ and $w$ such that $e := \{x, y, z, w\}$ is an odd edge of $H$ and $\{y, z\}$ is a $\beta_2$-medium pair. If $\{x, w\}$ is $\beta_2$-medium, then $e$ is the union of two $\beta_2$-medium pairs and we are done by our first observation, whilst if $\{x, w\}$ is $\beta_2$-bad, then we have a $\beta_2$-bad pair containing the $\gamma$-good vertex $w$. We then proceed as follows.

The minimum degree condition of $H$, combined with the fact that $\{w, x\}$ is $\beta_2$-bad, implies that there are at least $(\frac{1}{2} - 2\beta_2) \binom{n}{2}$ pairs $\{z', w'\}$ such that $\{w, x, z', w'\}$ is an odd edge of $H$. Since $w$ is a $\gamma$-good vertex, by Proposition~\ref{clm:goodvpt}(i) at most $\gamma n$ pairs containing $w$ are not $\gamma$-good. So certainly there are at most $\gamma n^2 \leq 3 \gamma \binom{n}{2}$ pairs $\{z', w'\}$ for which $\{w, z'\}$ is $\beta_2$-bad or $\{w, w'\}$ is $\beta_2$-bad. Moreover, since $x$ is $(\beta_1, \beta_2)$-medium (since we assumed all vertices are), at least $\beta_1 n$ pairs containing $x$ are $\beta_2$-medium. Therefore there are at least $\beta_1 n\cdot((\frac{1}{2} - c)n-\beta_1 n-1)>\frac{\beta_1}{2}\binom{n}{2}$ pairs $\{z',w'\}$ such that $\{w,x,z',w'\}$ is an odd $4$-tuple and such that $\{x, z'\}$ and $\{x, w'\}$ are not both $\beta_2$-bad. Since there are at most $(\frac{1}{2} +c) \binom{n}{2}$ pairs $\{z', w'\}$ such that $\{w, x, z', w'\}$ is an odd $4$-tuple, and $(\frac{1}{2} + c)-(\frac{1}{2} -2\beta_2) + 3\gamma < \frac{\beta_1}{2}$, it follows that there exist vertices $z'$ and $w'$ such that $\{w, x, z', w'\}$ is an odd edge of $H$ containing $\beta_2$-medium pairs $\{w, z'\}$ and $\{x, w'\}$, and we saw already that this is sufficient.
\end{proof}

Now we have all the tools we need to prove Lemma~\ref{thm:evenextr}. In fact, we actually prove the following stronger statement, which additionally asserts that given an even-extremal bipartition as input, we can find a Hamilton $2$-cycle in polynomial time.

\begin{lemma} \label{thm:evenextralgo}
Suppose that $1/n \ll \eps, c' \ll 1$ and that $n$ is even, and let $H$ be a $4$-graph of order $n$ with $\delta(H)\geq n/2-\eps n$. If $H$ is $c'$-even-extremal and every bipartition $\{A, B\}$ of $V(H)$ is even-good, then $H$ contains a Hamilton $2$-cycle. Moreover, there exists an algorithm Procedure~HamCycleEven$(H,\{A,B\})$ which, given as input a $c'$-even-extremal bipartition $\{A,B\}$ of $V(H)$, returns a Hamilton $2$-cycle in $H$ in time $O(n^{10})$.
\end{lemma}

\begin{proof}
For this proof we introduce further constants such that 
\[\eps, c' \ll c \ll\gamma\ll\beta\ll\beta_2\ll\beta_2'\ll\beta_1\ll\beta_1'\ll 1\;.\]
Since $H$ is $c'$-even-extremal we may fix a bipartition $\{A', B'\}$ of $V := V(H)$ with $(\frac{1}{2}-c')n\leq|A'|\leq(\frac{1}{2}+c')n$ which induces at most $c'\binom{n}{4}$ odd edges. We begin by moving all vertices which are $(\beta_1',\beta_2')$-bad to the other side. 
More precisely, we define $A^{\bad}:=\{a\in A': a\textrm{ is }(\beta_1',\beta_2')\textrm{-bad}\}$ and $B^{\bad}:=\{b\in B': b\textrm{ is }(\beta_1',\beta_2')\textrm{-bad}\}$, 
and set $A:=(A'\sm A^{\bad})\cup B^{\bad}$ and $B:=(B'\sm B^{\bad})\cup A^{\bad}$; we say that the vertices of $A^{\bad}\cup B^{\bad}$ are \emph{moved}. 
By Proposition~\ref{clm:types}(c) at most $3c' n$ vertices are moved in total. 

\begin{claim}\label{clm:goodpartition}
$\{A, B\}$ is a $c$-even-extremal bipartition of $V$ with respect to which every vertex of $H$ is $(\beta_1,\beta_2)$-medium. Moreover, every vertex of $H$ is contained in at least $\beta n$ connate pairs which are $\beta$-medium with respect to $\{A, B\}$.
\end{claim}

\begin{proof}
Since at most $3c'n$ vertices were moved, we have $(\frac{1}{2}-c)n \leq (\frac{1}{2}-c' - 3c')n \leq |A| (\frac{1}{2}+c' + 3c')n \leq(\frac{1}{2}+c)n$, and the number of edges of $H$ which are odd with respect to $\{A, B\}$ is at most $c'\binom{n}{4} + 3c'n^4 \leq c \binom{n}{4}$, so $\{A, B\}$ is $c$-even-extremal. Now, since $\beta_2 \ll \beta_2' \ll \beta_1 \ll \beta_1'$, every unmoved vertex which was $(\beta_1',\beta_2')$-medium under $\{A', B'\}$ is $(\beta_1,\beta_2)$-medium under~$\{A, B\}$, whilst every moved vertex was in at least $(1-\beta_1')n$ pairs which were not $\beta_2'$-medium under $\{A', B'\}$ and therefore is in at least $(1-\beta_1' - 3c')n$-many $\beta_2$-medium pairs under $\{A, B\}$. Thus every vertex is $(\beta_1,\beta_2)$-medium with respect to $\{A, B\}$. 

Now suppose that some vertex $x \in V$ is contained in fewer than $\beta n$ connate pairs which are $\beta$-medium. Since $x$ is $(\beta_1,\beta_2)$-medium it follows that $x$ is contained in at least $(\beta_1-\beta)n \geq \frac{\beta_1}{2}n$ split pairs which are $\beta_2$-medium, so the number $N$ of even edges $e \in E(H)$ with $|e \cap A| = |e \cap B| =2$ which contain $x$ satisfies $N \geq \frac{1}{3}\cdot \frac{\beta_1}{2} n\cdot\beta_2\binom{n}{2} \geq 2\beta n^3$ even edges $e$ with $|e \cap A| = |e \cap B| =2$. But then there must be at least $\beta n$ connate pairs containing $x$ which are contained in at least $\beta \binom{n}{2}$ even edges of $H$, as otherwise we would have $N < \beta n \cdot \binom{n}{2} + n \cdot \beta \binom{n}{2} < 2 \beta n^3$, and each of these pairs is $\beta$-medium, giving a contradiction.
\end{proof}

We henceforth exclusively use the terms odd, even, $\gamma$-good, $(\beta_1, \beta_2)$-medium and so forth with respect to the bipartition $\{A, B\}$ of $V$. Observe that $H$, $V$, $n$ and the constants $\eps, c, \gamma, \beta, \beta_1$ and $\beta_2$ are as in Setup~\ref{setup:evenextr}.

We distinguish five cases which are related to the cases in the definition of an even-good bipartition. 
Case~A assumes only that $H$ contains two disjoint odd edges. All other cases assume that there are no two disjoint odd edges in $H$, and in addition Case~B assumes that $|A|$ and $|B|$ are even, Case~C assumes that $|A|$ and $|B|$ are odd and that there are two odd edges whose intersection is a split pair, Case~D assumes that $|A|=|B|$ is odd and that there are no two odd edges whose intersection is a split pair and Case~E assumes that $|A|=|B|+2$ and that there are two odd edges whose intersection is a connate pair in $A$. By symmetry Case~E also covers the case where $|B| = |A| + 2$ and there are two odd edges whose intersection is a connate pair in $B$. Since $\{A, B\}$ is even-good by assumption, at least one of these cases must hold, so to prove the lemma it suffices to construct a Hamilton cycle in $H$ in each case.

We begin with Cases A--C, for which we construct a Hamilton cycle as follows.
\begin{enumerate}[label=(\Roman*)]
\item We build a `bridge' $Q$, which is a path on at most $\gamma n$ vertices whose ends are a $\gamma$-good pair $q_A \in \binom{A}{2}$ and a $\gamma$-good pair $q_B \in \binom{B}{2}$ such that $|A\sm V(Q)|$ and $|B\sm V(Q)|$ are both even and so that $Q$ contains all vertices of $H$ which are not $\gamma$-good.
\item Next, we choose $\gamma$-good pairs $p_A \in \binom{A \sm V(Q)}{2}$ and $p_B \in\binom{B \sm V(Q)}{2}$ such that $p_A \cup p_B \in E(H)$, and take $P$ to be the path consisting of this single edge. 
\item Finally, we apply Lemma~\ref{lem:evenextrhampathsplit} twice to form a Hamilton path $P_{A}$ in $H[(A \sm V(Q)) \cup q_A]$ with ends $p_A$ and $q_A$ and a Hamilton path $P_{B}$ in $H[(B\sm V(Q))\cup q_B]$ with ends $p_B$ and $q_B$, and then $QP_{B}PP_{A}$ is a Hamilton cycle in $H$.
\end{enumerate}
It suffices to prove the existence of the bridge $Q$ in each case. Indeed, having constructed the bridge $Q$ we may choose $p_A$ and $p_B$ as in Step~(II) by choosing any $\gamma$-good pair $p_A \in \binom{A \sm V(Q)}{2}$ and then using Proposition~\ref{clm:evenconnect}(iv) to obtain a $\gamma$-good pair $p_B \in \binom{B \sm V(Q)}{2}$ with $p_A \cup p_B \in E(H)$. We then just need to explain how to apply Lemma~\ref{lem:evenextrhampathsplit} in Step~(III).

For this define $A^* := (A\sm V(Q))\cup q_A$, $H_A := H[A^*]$ and $n_A :=|A^*|$. Then by choice of $Q$ we have that $n_A$ is even and $(\frac{1}{2}-2\gamma)n \leq n_A \leq (\frac{1}{2}+c)n$, so we can choose an arbitrary bipartition $A^*=S_1\cup S_2$ of $H_A$ such that $|S_1|=|S_2|$ and such that $p_A$ and $q_A$ are split pairs in relation to $(S_1,S_2)$. Since each vertex $v\in A^*$ is $\gamma$-good in $H$, it follows that there are at most $(\frac{1}{2}+2c)\binom{n}{3}\ - (\frac{1}{2}-\gamma^5)\binom{n}{3}$ even $4$-tuples $S \in \binom{V}{4}$ which contain $v$ and do not form an edge of $H$. Then $v$ is contained in at least $n_A/2\cdot(n_A/2-1) - (2c+\gamma^5)\binom{n}{3} > (1/2-\beta^3)\binom{n_A}{3}$ edges in $H_A$, and so $v$ is $\beta$-good in $H_A$. Similar calculations show that, since each of $p_A$ and $q_A$ is a $\gamma$-good pair in $H$, each is a $\beta$-good pair in~$H_A$. Thus we can apply Lemma~\ref{lem:evenextrhampathsplit} to obtain a Hamilton path in $H_A$ with ends $p_A$ and~$q_A$ as required. The same argument shows that we can find a Hamilton path in $H_B$ with ends $p_B$ and $q_B$ also.

We now show how to construct the bridge $Q$ in each of Cases A--C.

\medskip\noindent {\bf Case~A:} Assume there are two disjoint odd edges $e_1,e_2\in E$. To construct $Q$ we first use Lemma~\ref{lem:bridge} to find a path $Q'_1$ of length three in $H[V \sm e_2]$ whose ends are a $\gamma$-good connate pair $p'_1$ and a $\gamma$-good split pair $s_1$. 
Next, we use Lemma~\ref{lem:bridge} again to find a path $Q_2$ of length three in $H[V\sm V(Q'_1)]$ whose ends are a $\gamma$-good split pair $s_2$ and a $\gamma$-good connate pair~$p_2$. 
By Proposition~\ref{clm:evenconnect}(iv) we may then choose a $\gamma$-good connate pair $p_1 \in \binom{A}{2}$ disjoint from $V(Q'_1) \cup V(Q'_2)$ such that $p_1 \cup p'_1 \in E(H)$, to give a path $Q_1 := p_1p'_1Q'_1s_1$ of length four which is vertex-disjoint from $Q_2$. Write $V' := V \sm (V(Q_1) \cup V(Q_2))$, and let $X = \{x_1,\cdots,x_k\}$ be the set of all vertices not in $Q_1$ or $Q_2$ which are not $\gamma$-good, so $k := |X| \leq \frac{c}{\gamma^5-2\eps}n$ by Proposition~\ref{clm:types}(b). Since every vertex is in at least $\beta n$-many $\beta$-medium connate pairs by Claim~\ref{clm:goodpartition}, and $\beta n > \frac{2c}{\gamma^5-2\eps}n + 18 \geq 2k + |V(Q_1)| + |V(Q_2)|$, we may then greedily choose distinct vertices $\{y_1,\cdots,y_k\} \in V' \sm X$ such that the pair $r_i := \{x_i,y_i\}$ is a $\beta$-medium connate pair for each $i \in [k]$. 

We next choose $\gamma$-good connate pairs $g_i, g'_i \in \binom{V'}{2}$ with $g_i \cup r_i \in E(H)$ and $g'_i \cup r_i \in E(H)$ for each $i \in [k]$, and then choose connate pairs $h_i \in \binom{V'}{2}$ with $h_i \cup g_i \in E(H)$ and $h_i \cup g'_{i+1} \in E(H)$ for each $i \in [k-1]$. We additionally require that the pairs $g_i, g'_i$ and $h_i$ are disjoint from each other and from the pairs $r_i$. By Proposition~\ref{clm:types}(d) it is possible to choose the pairs $g_i$ and $g'_i$ greedily with these properties, whilst Proposition~\ref{clm:evenconnect}(ii) ensures that we may choose the pairs $h_i$ greedily also. Similarly, we may also apply Proposition~\ref{clm:evenconnect}(iv) to choose a $\gamma$-good connate pair $h_k \in \binom{V' \cap B}{2}$ such that $h_k \cup g_k \in E(H)$ which is disjoint from all previously-chosen pairs. Observe that, having made these choices, 
$$Q_3 := g'_1r_1g_1h_1g'_2r_2g_2h_2g'_3r_3g_3h_3g'_4\dots h_{k-1}g'_{k}r_{k}g_{k}h_k$$
is a path in $H$. By Proposition~\ref{clm:evenconnect}(ii) we may then choose a connate pair $p^*$ such that $p^* \cup p_2 \in E(H)$ and $p^* \cup g'_1 \in E(H)$, so `connecting' the paths $Q_2$ and $Q_3$.  

Write $V'' := V \sm (V(Q_1) \cup V(Q_2) \cup V(Q_3) \cup p^*)$. If $|V'' \cap A|$ is odd then we use Proposition~\ref{clm:evenconnect}(i) to choose a split pair $s^* \in \binom{V''}{2}$ such that $s_1 \cup s^* \in E(H)$ and $s_2 \cup s^* \in E(H)$, and define the path $Q := p_1Q_1s_1s^*s_2Q_2p_2p^*g'_1Q_3h_k$. On the other hand, if $|V'' \cap A|$ is even then we first use Proposition~\ref{clm:evenconnect}(iii) to choose a $\gamma$-good split pair $s' \in \binom{V''}{2}$ such that $s_1 \cup s' \in E(H)$ and then use Proposition~\ref{clm:evenconnect}(i) to choose a split pair $s^{*} \in \binom{V''}{2}$ such that $s' \cup s^{*} \in E(H)$ and $s_2 \cup s^{*} \in E(H)$; we then define the path $Q := p_1Q_1s_1s's^{*}s_2Q_2p_2p^*g'_1Q_3h_k$. Either way $Q$ is a path in $H$ whose ends $q_A := p_1 \in \binom{A}{2}$ and $q_B := h_k \in \binom{B}{2}$ are $\gamma$-good pairs such that $|A \sm V(Q)|$ is even; since $n$ and $V(Q)$ are both even it follows that $|B \sm V(Q)|$ is even also. Moreover $Q$ contains all non-$\gamma$-good vertices by choice of $Q_3$, and the total number of vertices in $Q$ is at most $|V(Q_1)| + |V(Q_2)| + |V(Q_3)| + 6 \leq 22 + 8k \leq \gamma n$, so $Q$ has the required properties.

\medskip \noindent {\bf Case~B:} Assume that there are no two disjoint odd edges, but $|A|$ and $|B|$ are even. Then there must be at most one non-$\gamma$-good vertex, as otherwise we would have two disjoint odd edges. Let $x$ be such a vertex, if it exists; otherwise choose $x \in V$ arbitrarily. Then by Claim~\ref{clm:goodpartition} we may choose a $\beta$-medium connate pair $p$ which contains $x$. 
We now apply Proposition~\ref{clm:types}(d) twice to find disjoint $\gamma$-good connate pairs $p_1$ and $p_2$ such that $p_1 \cup p \in E(H)$ and $p_2 \cup p \in E(H)$. By symmetry we may assume that $p_2 \in \binom{B}{2}$, and by Proposition~\ref{clm:evenconnect}(iv) we may choose a $\gamma$-good connate pair $p_0 \in \binom{A}{2}$ disjoint from $p_1, p$ and $p_2$ such that $p_0 \cup p_1 \in E(H)$. We may then take $Q := p_0p_1pp_2$ (in particular $Q$ has an even number of vertices in each of $A$ and $B$, so $|A \sm V(Q)|$ and $|B \sm V(Q)|$ are both even, as required).

\medskip \noindent{\bf Case~C:} Assume that $|A|$ and $|B|$ are odd and there are no two disjoint odd edges, but there are two odd edges $e_1$ and $e_2$ whose intersection is a split pair. That is, we may write $e_1= p_1 \cup s$ and $e_2 = s \cup p_2$ where $s$ is a split pair and $p_1$ and $p_2$ are connate pairs. Then $p_1$ and $p_2$ must be $\gamma$-good pairs, as otherwise we would have two disjoint odd edges. For the same reason all vertices in $V \sm (s \cup p_1 \cup p_2)$ must be $\gamma$-good. By Proposition~\ref{clm:evenconnect}(iv) we may choose disjoint $\gamma$-good connate pairs $q_1 \in \binom{A}{2}$ and $q_2 \in \binom{B}{2}$ which do not intersect $s, p_1$ or $p_2$ so that $p_1 \cup q_1 \in E(H)$ and $p_2 \cup q_2 \in E(H)$, and we may then take $Q := q_1p_1sp_2q_2$ (in particular $Q$ has an odd number of vertices in each of $A$ and $B$, so $|A\sm Q|$ and $|B\sm Q|$ are even, as required).

\medskip We now turn to cases $D$ and $E$, for which we use the following, similar strategy to construct a Hamilton cycle.
\begin{enumerate} [label=(\Roman*)]
\item We construct a path $P_0$ on at most six vertices whose ends $s_1$ and $s_2$ are both $\gamma$-good split pairs such that $P_0$ contains all non-$\gamma$-good vertices of $H$ and $|A\sm V(P_0)| = |B\sm V(P_0)|$.
\item Write $V' := (V\sm V(P_0)) \cup s_1 \cup s_2$. Then by our choice of $P_0$ we have $|A \cap V'| = |B \cap V'|$, and $H[V']$ contains only $\gamma$-good vertices. So by Lemma~\ref{lem:evenextrhampathsplit} there is a Hamilton path~$P_1$ in $H[V']$ with ends $s_1$ and $s_2$, and then $P_0P_1$ is a Hamilton cycle in $H$.
\end{enumerate}
Hence it suffices to construct $P_0$ in each case.

\medskip \noindent {\bf Case~D:} Assume that $|A|=|B|$ is odd, there are no two disjoint odd edges and there are no two distinct odd edges whose intersection is a split pair. Then there is at most one non-$\gamma$-good vertex, as otherwise we would have two disjoint odd edges. If there is such a vertex, we denote it by $x$ and assume without loss of generality that $x \in A$; if not then we choose a vertex $x \in A$ arbitrarily. For every $b \in B$ the pair $\{x, b\}$ must be $\gamma$-good, as otherwise there would be two distinct odd edges whose intersection is the split pair $\{x, b\}$. So we may choose $b\in B$ such that $s_1 :=\{x,b\}$ is $\gamma$-good and then use Proposition~\ref{clm:evenconnect}(iii) to obtain a $\gamma$-good split pair $s_2$ such that $s_1\cup s_2 \in E(H)$. We may then take $P_0$ to be the single edge $s_1 \cup s_2$.

\medskip \noindent {\bf Case~E:} Assume that $|A|=|B|+2$ and there are no two disjoint odd edges, but there are odd edges $e_1$ and $e_2$ whose intersection is a connate pair in $A$. That is, we may write $e_1= s_1 \cup p$ and $e_2= p \cup s_2$ where $s_1$ and $s_2$ are split pairs and $p$ is a connate pair in $A$. Then $s_1$ and $s_2$ must both be $\gamma$-good pairs, as otherwise we would have two disjoint odd edges. For the same reason all vertices in $V \sm (s_1 \cup p \cup s_2)$ must be $\gamma$-good, so we may take $P_0 := s_1ps_2$ (in particular, $P_0$ contains two more vertices from $A$ than $B$, so $|A \sm V(P_0)| = |B \sm V(P_0)|$, as required). \medskip

Finally, for the `moreover' part of the lemma statement, note that if we are given a $4$-graph~$H$ as in the lemma, it is not clear that we can find an even-extremal partition $\{A', B'\}$ of $H$ in polynomial time. However, if such a partition is also given, then the remaining steps of the proof can be carried out efficiently. Indeed, we can identify the sets $A^{\bad}$ and $B^{\bad}$ and form the partition $\{A, B\}$ in time $O(n^4)$, and then we can identify in time $O(n^8)$ which of Cases A--E holds for $\{A, B\}$ (since $\{A, B\}$ is even-good at least one of the cases must hold). In Cases B--E the path $Q$ or $P_0$ (according to the case) has at most 10 vertices, so can be found by exhaustive search in time $O(n^{10})$, whilst the greedy argument given in case $A$ constructs $Q$ in time $O(n^4)$. This completes Step~(I) in each case. In Cases~A--C we can then find an edge $f$ as in Step~(II) in time $O(n^4)$ by exhaustive search, and Lemma~\ref{lem:evenextrhampathsplit} states that we can then complete Step~(III) in time $O(n^4)$ also. Similarly, in Cases~D and~E we can complete Step~(II) in time $O(n^4)$.
\end{proof}

\section{Hamilton 2-cycles in 4-graphs: odd-extremal case}\label{sec:oddextr}

In this section we give a detailed proof of Lemma~\ref{thm:oddextremal} which states that Theorem~\ref{character} holds for odd-extremal $4$-graphs. As in the previous two sections, all paths and cycles we consider in this section are $2$-paths and $2$-cycles, therefore we will again omit the $2$ and speak simply of paths and cycles. Throughout this section we will work within the following setup.

\begin{setup}\label{setup:oddextremal}
Fix constants satisfying $1/n \ll\eps, c\ll\gamma \ll \psi \ll \beta_2\ll\beta_1.$
Let~$H$ be a $4$-graph on $n$ vertices with $\delta(H)\geq n/2-\eps n$, write $V := V(H)$ and let $\{A, B\}$ be a $c$-odd-extremal bipartition of $V$. 
\end{setup}

Our strategy is broadly similar to the one outlined in the previous section. Indeed, the minimum degree condition on $H$, combined with the fact that $H$ has very few even edges (since the bipartition $\{A, B\}$ is odd-extremal), implies that almost all possible odd edges are present in $H$. To find a Hamilton cycle in $H$, we will find a short path $P$ in $H$ whose ends $s_1$ and $s_2$ are split pairs so that, using the high density of odd edges of $H$, we can then find a Hamilton path $Q$ in $H[(V \sm V(P)) \cup s_1 \cup s_2]$ with ends $s_1$ and $s_2$. Together $P$ and $Q$ then form a Hamilton cycle in $H$. To implement this strategy, we begin by establishing some necessary preliminaries, then prove the Hamilton path connecting lemma (Lemma~\ref{lem:oddhampath}) which we use to find $Q$, before proceeding to give the proof of Lemma~\ref{thm:oddextremal}.

\begin{definition}
Under Setup~\ref{setup:oddextremal}, we say that
\begin{enumerate}[label=(\roman*)]
\item a triple $\{x,y,z\} \in \binom{V}{3}$ is \emph{$\gamma$-good} if it is contained in at least $(\frac{1}{2}-\gamma)n$ odd edges,
\item a pair $\{x,y\} \in \binom{V}{2}$ is \emph{$\gamma$-good} if it is contained in at least $(\frac{1}{2}-\gamma^3)\binom{n}{2}$ odd edges,
\item a vertex $v\in V$ is \emph{$\gamma$-good} if it is contained in at least $(\frac{1}{2}-\gamma^5)\binom{n}{3}$ odd edges,
\item a pair $\{x,y\}\in \binom{V}{2}$ is \emph{$\beta_2$-medium} if it is contained in at least $\beta_2\binom{n}{2}$ odd edges,
\item a pair $\{x,y\}\in \binom{V}{2}$ is \emph{$\beta_2$-bad} if it is not \emph{$\beta_2$-medium},
\item a vertex $v\in V$ is \emph{$(\beta_1,\beta_2)$-medium} if there are at least $\beta_1 n$ vertices in $A$ which form a $\beta_2$-medium pair with $v$ and also at least $\beta_1 n$ vertices in $B$ which form a $\beta_2$-medium pair with $v$, and
\item a vertex $v\in V$ is \emph{$(\beta_1,\beta_2)$-bad} if it is not $(\beta_1,\beta_2)$-medium.
\end{enumerate}
\end{definition}

Note that definitions~(i)-(v) above are identical to those of the previous section (Definition~\ref{def:good}) with `odd' in place of `even', but~(vi) and~(vii) (the definitions of $(\beta_1,\beta_2)$-medium and $(\beta_1,\beta_2)$-bad vertices) differ significantly.

\begin{prop}\label{prop:oddproperties}
Adopt Setup~\ref{setup:oddextremal}. Then the following statements hold.
\begin{enumerate}[label=(\alph*)]
\item If $ v\in V$ is $\gamma$-good, then $v$ is contained in at least $(1-\gamma)n$-many $\gamma$-good pairs.
\item If a pair $p \in \binom{V}{2}$ is $\gamma$-good, then $p$ is contained in at least $(1-\gamma)n$-many $\gamma$-good triples.
\item At most $\frac{c}{\gamma^3-2\eps}\binom{n}{2}$ pairs  are not $\gamma$-good.
\item At most $\frac{c}{\gamma^5-2\eps}n$ vertices are not $\gamma$-good.
\item At most $5cn$ vertices are $(\beta_1,\beta_2)$-bad.
\end{enumerate}
Now suppose also that $R \subseteq V$ is such that $|R \cap A|,|R \cap B| \leq n/2-\psi n$. Then the following statements hold.
\begin{enumerate}[label=(\alph*)]
\setcounter{enumi}{5}
\item For every two disjoint $\gamma$-good connate pairs $p_1$ and $p_2$ there exists a split pair $s\in \binom{V\sm R}{2}$ such that $p_1sp_2$ is a path of length $2$ in $H$.
\item For every two disjoint $\gamma$-good split pairs $s_1$ and $s_2$ there exists a connate pair $p\in \binom{A\sm R}{2}$ such that $s_1ps_2$ is a path of length $2$ in $H$.
\item For every $\gamma$-good connate pair $p$ there exists a $\gamma$-good split pair $s\in \binom{V\sm R}{2}$ such that $p\cup s\in E(H)$.
\item For every $\gamma$-good split pair $s$ there exists a $\gamma$-good connate pair $p\in \binom{A\sm R}{2}$ such that $s\cup p\in E(H)$.
\end{enumerate}
\end{prop}

\begin{proof}
The proofs of~(a) and~(b) are identical to those of Proposition~\ref{clm:goodvpt} with the words `odd' and `even' interchanged, the proofs of~(c) and~(d) are similarly identical to those of Proposition~\ref{clm:types}(a) and~(b). For~(e), observe that any vertex which is $(\beta_1, \beta_2)$-bad is contained in at least $(\frac{1}{2} - c - \beta_1)n$-many $\beta_2$-bad pairs, each of which is contained in fewer than $\beta_2 \binom{n}{2}$ odd edges of $H$. Since any pair is contained in at least $(\frac{1}{2}-2\eps)\binom{n}{2}$ edges of $H$, and $H$ has at most $c\binom{n}{4}$ even edges, as $\{A, B\}$ is $c$-odd-extremal, it follows that the number of $(\beta_1, \beta_2)$-bad vertices is at most $$ \frac{4 \cdot 3 \cdot c\binom{n}{4}}{(\frac{1}{2} - c - \beta_1)n \cdot (\frac{1}{2}-2\eps - \beta_2)\binom{n}{2}} < 5 cn\;.$$

For~(f), note that since $p_i$ is $\gamma$-good for $i\in\{1,2\}$, there are at least $(\frac{1}{2}-\gamma)\binom{n}{2}$ split pairs which form an edge with $p_i$. Since there are at most $\frac{n^2}{4}$ split pairs in total, it follows that there are at most $3 \gamma n^2$ split pairs $s$ for which  $p_1 \cup s$ and $p_2 \cup s$ are not both edges. Furthermore there are at least $(\psi-c)n$ vertices in each of $A\sm R$ and $B\sm R$, so there are at least $(\psi-c)^2n^2$ split pairs which do not contain a vertex of $R$. Since $3\gamma n^2<(\psi-c)^2n^2$, it follows that there exists a split pair $s \in \binom{V\sm R}{2}$ such that $p_1\cup s$ and $p_2\cup s$ are both edges, as required. Likewise~(h) follows, since by~(c) at most $\frac{c}{\gamma^3-2\eps}\binom{n}{2} < \gamma n^2$ split pairs are not $\gamma$-good. The proofs of~(g) and~(i) are very similar, so we omit them.
\end{proof}

The following lemma is our Hamilton path connecting lemma which, under the assumption that every vertex of $H$ is good, allows us to choose a Hamilton path in $H$ with specified ends. From a broad perspective the proof is similar to that of Lemma~\ref{lem:evenextrhampathsplit}, but the construction of the `grid' is quite different, reflecting the fact that $H$ is odd-extremal rather than even-extremal. Since many calculations are similar, we will be more concise and primarily emphasise the differences.
 
\begin{lemma}\label{lem:oddhampath}
Adopt Setup~\ref{setup:oddextremal}, and suppose also that every vertex of $H$ is $\gamma$-good, and that
\begin{enumerate}[label=(\roman*)]
\item $n\equiv 6$ mod $8$ and $|A|-|B|\equiv 2$ mod $4$, or
\item $n\equiv 2$ mod $8$ and $|A|-|B|\equiv 0$ mod $4$.
\end{enumerate} 
If $s_1$ and $s_2$ are disjoint $\gamma$-good split pairs, then there exists a Hamilton path in $H$ whose ends are $s_1$ and $s_2$. Moreover, such a path can be found in time $O(n^4)$.
\end{lemma}

\begin{proof}
Set $m := \lceil \frac{(1-\psi)n}{8}\rceil$. 
We begin by using Proposition~\ref{prop:oddproperties}(a) and our assumption that every vertex is $\gamma$-good to greedily choose sets $L'_1 = \{x_1,y_1, \cdots,x_{m+1},y_{m+1}\} \subseteq A \sm (s_1\cup s_2)$ and $L'_2 = \{x'_1,y'_1, \cdots,x'_{m+1},y'_{m+1}\} \subseteq B \sm (s_1\cup s_2)$ such that both $\{x_i, y_i\}$ and $\{x_i', y_i'\}$ are $\gamma$-good for every $i \in [m+1]$. Next we use Proposition~\ref{prop:oddproperties}(b) to greedily choose sets $Z= \{z_1, \dots, z_m\} \subseteq A \sm (L'_1 \cup s_1 \cup s_2)$ and $Z'= \{z'_1, \dots, z'_m\} \subseteq B \sm (L'_2 \cup s_1 \cup s_2)$ such that for each $i \in [m]$ the triples $\{x_{i},y_{i},z_i\}$, $\{z_i,x_{i+1},y_{i+1}\}$, $\{x'_{i},y'_{i},z'_i\}$ and $\{z'_i,x'_{i+1},y'_{i+1}\}$ are all $\gamma$-good. Finally by Proposition~\ref{prop:oddproperties}(f) we can then choose $v \in A \sm (L'_1 \cup Z \cup s_1 \cup s_2)$ and $v' \in B \sm (L'_2 \cup Z' \cup s_1\cup s_2)$ such that $\{x_{m+1},y_{m+1},v, v'\}$ and $\{v,v',x_1',y_1'\}$ are both edges of $H$. Our `grid' is then $L := L'_1 \cup Z \cup \{v,v'\} \cup L'_2 \cup Z'$, and we take $p_0 :=\{x_1,y_1\}$ and $q_0 :=\{x_{m+1}',y_{m+1}'\}$. Observe in particular that $|L \cap A| = |L \cap B| = 3m+3$.

Define $A'=A\sm (L\cup s_1\cup s_2)$ and $B'=B\sm (L\cup s_1\cup s_2)$. Let $G_A$ be the bipartite graph with vertex classes $Z'$ and $A'$ whose edges are all pairs $\{z_{i}',w\}$ with $i\in[m]$ and $w\in A'$ for which both $\{x_{i}',y_i',z_i',w\},\{w,z_i',x_{i+1}',y_{i+1}'\} \in E(H)$, and let $G_B$ be the bipartite graph with vertex classes $Z$ and $B'$ and whose edges are all pairs $\{z_{i},w\}$ with $i\in[m]$ and $w\in B'$ such that $\{x_{i},y_i,z_i,w\},\{w,z_i,x_{i+1},y_{i+1}\}\in E(H)$. We call a vertex $a\in A'$ \emph{palatable} if $d_{G_A}(a)\geq 0.9m$ and a vertex $b\in B'$ \emph{palatable} if $d_{G_B}(b)\geq 0.9m$. Essentially the same argument as in the proof of Lemma~\ref{lem:evenextrhampathsplit} then shows that $A'$ and $B'$ each contain at most $\frac{\psi}{100} n$ non-palatable vertices, and furthermore that $L$ can \emph{swallow} any set $S\subseteq V\sm L$ of $2m$ palatable vertices with $|S\cap A| = |S \cap B| = m$, meaning that for any such $S$ there is a path in $H$ with vertex set $L\cup S$ and ends $p_0$ and $q_0$.

Let $\ell := |A| - \tfrac{n}{2} = \tfrac{n}{2} - |B|= \frac{1}{2}(|A|-|B|)$. We then have $|\ell| \leq cn$ since $\{A, B\}$ is a $c$-odd-extremal bipartition of $V$, and our divisibility assumptions ensure that both $k_A := \frac{1}{8}(n-8m-10+4\ell)$ and $k_B := \frac{1}{8}(n-8m-10-4\ell)$ are integers. Observe also that $k_A, k_B \geq \frac{\psi}{10} n$. Since the number of non-palatable vertices in each of $A'$ and $B'$ is at most~$\frac{\psi}{100} n$, we may choose sets $U_A = \{a_1, \dots, a_{k_A-1}\} \subseteq A'$ and $U_B = \{b_1, \dots, b_{k_B-1}\} \subseteq B'$ of distinct vertices so that $U_A \cup U_B$ contains all non-palatable vertices. 
Now recall our assumption that every vertex of $H$ is $\gamma$-good. So by repeated application of Proposition~\ref{prop:oddproperties}(a) and~(i) we may greedily choose vertex-disjoint $\gamma$-good connate pairs $p_1, \dots, p_{k_A} \in \binom{A'}{2}$ and $q_1, \dots, q_{k_B} \in \binom{B'}{2}$ 
so that $a_i \in p_i$ for each $i \in [k_A-1]$ and $b_j \in q_j$ for each $j \in [k_B-1]$, and also so that each of $s_1 \cup p_{k_A}$ and $s_2 \cup q_{k_B}$ is an edge of $H$.
Write $W := \bigcup_{i \in [k_A]} p_i \cup \bigcup_{j \in [k_B]} q_j$. Then by repeated application of Proposition~\ref{prop:oddproperties}(f) we may 
greedily choose vertex-disjoint split pairs $p'_1, \dots, p'_{k_A}, q'_1, \dots, q'_{k_B} \in \binom{(A' \cup B') \sm W}{2}$ such that 
$p_{i-1} \cup p'_i$ and $p'_i \cup p_i$ are both edges of $H$ for each $i \in [k_A]$ and 
$q_{i-1} \cup q'_i$ and $q'_i \cup q_i$ are both edges of $H$ for each $i \in [k_B]$. 
We then define the paths $Q_1 := p_0p'_1p_1p'_2\dots p_{k_A-1}p'_{k_A}p_{k_A}s_1$ and $Q_2 := q_0q'_1q_1q'_2\dots q_{k_B-1}q'_{k_B}q_{k_B}s_2$. 
Observe that we then have $|(V(Q_1) \sm L) \cap A| = 3k_A+1$ and $|(V(Q_1) \sm L) \cap B| = k_A+1$, and likewise that $|(V(Q_2) \sm L) \cap A| = k_B+1$ and $|(V(Q_2) \sm L) \cap B| = 3k_B+1$. 
Since $|L \cap A| = |L \cap B| = 3m+3$, we conclude that $R:= V \sm (L \cup Q_1 \cup Q_2)$ satisfies 
\begin{align*}
|R \cap A| &= |A| - (3m+3) - (3k_A+1) - (k_B + 1) \\
&= \frac{n}{2} + \ell - 3m - 5 - \frac{3}{8} (n-8m-10+4\ell) - \frac{1}{8}(n-8m-10-4\ell) = m
\end{align*} 
and $|R \cap B| = m$ by a similar calculation. Also every vertex of $R$ is palatable, since by construction  $Q_1$ and $Q_2$ cover all non-palatable vertices. It follows that $L$ can swallow $R$, so there is a path $P$ in $H$ with vertex set $R \cup L$ and with ends $p_0$ and $q_0$. This gives a Hamilton path $Q_1PQ_2$ in $H$ with ends $s_1$ and $s_2$. The argument for the running time is essentially identical to that in the proof of Lemma~\ref{lem:evenextrhampathsplit}.
\end{proof}

Recalling the proof strategy at the start of the section, we will use the following simple proposition to find the short path $P$ whose ends are to be connected using the Hamilton path connecting lemma (Lemma~\ref{lem:oddhampath}).

\begin{prop}\label{clm:medconnect}
Adopt Setup~\ref{setup:oddextremal}, and suppose also that every vertex of $V$ is $(\beta_1,\beta_2)$-medium. Also let $R\subseteq V$ have size $|R|\leq \psi n$. Then the following statements hold.
\begin{enumerate}[label=(\roman*)]
\item For every $\beta_2$-medium connate pair $p$ there exists a $\gamma$-good split pair $s \in \binom{V\sm R}{2}$ such that $p\cup s\in E(H)$.
\item For every $\beta_2$-medium split pair $s$ there exists a $\gamma$-good connate pair $p \in \binom{V\sm R}{2}$ such that $s\cup p\in E(H)$.
\item If there exists a $\beta_2$-bad connate pair $p \in \binom{V\sm R}{2}$, then there exist disjoint $\beta_2$-medium connate pairs $p_1, p_2 \in \binom{V \sm R}{2}$ such that $p_1 \cup p_2 \in E(H)$.
\item If there exists a $\beta_2$-bad split pair $s \in \binom{V\sm R}{2}$, then there exist disjoint $\beta_2$-medium split pairs $s_1, s_2 \in \binom{V \sm R}{2}$ such that $s_1 \cup s_2 \in E(H)$.
\end{enumerate}
\end{prop}

\begin{proof}
For~(i) observe that, since $p$ is $\beta_2$-medium, there are $\beta_2\binom{n}{2}$ split pairs $s \in \binom{V}{2}$ such that $p \cup s$ is an edge of $H$. By Proposition~\ref{prop:oddproperties}(c) at most $\frac{c}{\gamma^3-2\eps}\binom{n}{2}$ such pairs are not $\gamma$-good and at most $\psi n^2$ such pairs contain a vertex from $R$. Since $\beta_2\binom{n}{2}>\frac{c}{\gamma^5-2\eps}\binom{n}{2}+\psi n^2$, there exists a $\gamma$-good split pair $s \in  \binom{V\sm R}{2}$ such that $p\cup s\in E(H)$. A similar argument proves~(ii). 

For~(iii), assume without loss of generality that $p \in \binom{A \sm R}{2}$, and 
observe that there are at least $(\frac{1}{2}-2\eps)\binom{n}{2}$ pairs $q$ for which $p \cup q \in E(H)$, but since $p$ is $\beta_2$-bad at most $\beta_2\binom{n}{2}$ such pairs $q$ are split pairs. It follows that $p \cup q \in E(H)$ for all but at most $2 \beta_2\binom{n}{2}$ connate pairs $q \in \binom{A}{2}$. 
Write $p = \{x_1, x_2\}$. Since both $x_1$ and $x_2$ are $(\beta_1,\beta_2)$-medium there are at least $\beta_1^2\binom{n}{2}$ pairs $\{y_1,y_2\} \in \binom{A}{2}$ such that $\{x_1,y_1\}$ and $\{x_2,y_2\}$ are both $\beta_2$-medium. 
Since at most $2\psi\binom{n}{2}$ pairs contain a vertex from $R$, and $2\beta_2+2\psi<\beta_1^2$, we can choose $\{y_1,y_2\}$ such that $e'=\{x_1,x_2,y_1,y_2\}$ is an even edge in $H[V\sm R]$ containing two disjoint $\beta_2$-medium connate pairs. A similar argument proves~(iv).
\end{proof}

The following definition is helpful in the proof of Lemma~\ref{thm:oddextremal}. 

\begin{definition}
We say that an edge $e\in E(H)$ is
\begin{enumerate}[label=(\roman*)]
\item an even connate edge, if $|e\cap A|\in\{0,4\}$,
\item an even split edge, if $|e\cap A|=2$.
\end{enumerate}
\end{definition}

Observe that for pairs $p$ and $q$, if $p \cup q$ is an even connate edge then both $p$ and $q$ must be connate pairs, but if $p \cup q$ is an even split edge then either $p$ and $q$ are both connate pairs or $p$ and $q$ are both split pairs. 

We are now ready to prove Lemma~\ref{thm:oddextremal}; in fact, we actually prove the following stronger algorithmic version of the lemma.

\begin{lemma}\label{thm:oddextremalalgo}
Suppose that $1/n \ll \eps, c' \ll 1$ and that $n$ is even, and let $H$ be a $4$-graph of order $n$ with $\delta(H)\geq n/2-\eps n$. If $H$ is $c'$-odd-extremal and every bipartition $\{A, B\}$ of $V(H)$ is odd-good, then $H$ contains a Hamilton $2$-cycle. Moreover, there exists an algorithm Procedure~HamCycleOdd$(H,\{A,B\})$ which, given as input a $c'$-odd-extremal bipartition $\{A,B\}$ of $V(H)$, returns a Hamilton $2$-cycle in $H$ in time $O(n^{12})$.
\end{lemma}

\begin{proof}
First we introduce further constants such that
\[\eps, c' \ll c  \ll\gamma \ll\beta_2\ll\beta_2'\ll\beta_1\ll\beta_1'\ll\mu\ll 1\;.\]
Since $H$ is $c'$-odd-extremal there exists a bipartition $\{A', B'\}$ with $n/2-c'n\leq|A'|\leq n/2+c'n$ for which there are at most $c'\binom{n}{4}$ even edges. We begin by moving all vertices which are $(\beta_1',\beta_2')$-bad to the other side. To be precise define $A^{\bad}=\{a\in A'\mid a\textrm{ is }(\beta_1',\beta_2')\mbox{-}\bad\}$ and $B^{\bad}=\{b\in B'\mid b\textrm{ is }(\beta_1',\beta_2')\mbox{-}\bad\}$ and we set $A :=(A'\sm A^{\bad})\cup B^{\bad}$ and $B:=(B'\sm B^{\bad})\cup A^{\bad}$; we say that the vertices of $A^{\bad}\cup B^{\bad}$ are \emph{moved}. 

\begin{claim} \label{clm:oddpartition}
$\{A, B\}$ is a $c$-odd-extremal bipartition of $V$ with respect to which every vertex of $H$ is $(\beta_1,\beta_2)$-medium.
\end{claim}

\begin{proof}
By Proposition~\ref{prop:oddproperties}(e) at most $5c' n$ vertices are moved in total, so $n/2-cn \leq n/2 - 6c'n \leq |A|, |B| \leq n/2 + 6c'n \leq n/2+cn$, and at most $c'\binom{n}{4} + 5c'n \binom{n}{3} \leq c\binom{n}{4}$ edges of $H$ are even with respect to $\{A, B\}$. This proves that $\{A, B\}$ is $c$-odd-extremal. Also, any vertex which was not moved was $(\beta_1', \beta_2')$-medium with respect to $\{A', B'\}$, and so is $(\beta_1, \beta_2)$-medium with respect to $\{A, B\}$. So it remains to show that each moved vertex $v$ is also $(\beta_1,\beta_2)$-medium with respect to $\{A, B\}$. Without loss of generality assume that $v \in A^{\bad}$. First consider the case that there are fewer than $\beta_1 n$ vertices $a \in A$ for which the pair $\{v,a\}$ was $\beta_2$-medium with respect to $\{A', B'\}$. Then $v$ was contained in at least $1/3\cdot(n/2-2\beta_1n)\cdot(1/2-2\beta_2)\binom{n}{2} > (1/4-\mu)\binom{n}{3}$ edges which were even with respect to $\{A', B'\}$. Since there can be at most $(1/8+\mu)\binom{n}{3}$ even connate edges in total, this implies that $v$ was contained in at least $(1/8-2\mu)\binom{n}{3}$ even split edges (with respect to $\{A', B'\}$). Now consider the other case that there are fewer than $\beta_1 n$ vertices $b \in B$ for which $\{v,b\}$ was $\beta_2$-medium with respect to $\{A', B'\}$. Then $v$ was contained in at least $1/3\cdot(1/2-2\beta_1)n\cdot(1/2-2\beta_2)\binom{n}{2} > (1/4-\mu)\binom{n}{3}$ even split edges (with respect to $\{A', B'\}$). 
In either case we conclude that $v$, now after the moving, is contained in at least $(1/8-3\mu)\binom{n}{3}$ edges which have precisely three vertices in $B$; since $v \in B$ it follows that $v$ is $(\beta_1,\beta_2)$-medium with respect to $\{A, B\}$. 
\end{proof}

For the rest of the proof we will not use the constants $c',\beta_1'$ and $\beta_2'$, and we use the terms even, odd, split, connate, $\gamma$-good, $\beta_2$-medium, $(\beta_1, \beta_2)$-medium and so forth exclusively with respect to the partition $\{A, B\}$.
Observe that $H$, $A$, $B$ and the remaining constants satisfy the conditions of Setup~\ref{setup:oddextremal}. 
Fix $m \in \{0,2,4,6\}$ and $d \in \{0,2\}$ with $m \equiv |V| \pmod 8$ and $d \equiv |A|-|B| \pmod 4$.
We consider separately four cases for the pair $(m, d)$ as in the definition of an odd-good bipartition (Definition~\ref{evengood}). In each case we proceed by the following steps.
\begin{enumerate}[label=(\Roman*)]
\item We use the fact that $\{A, B\}$ is odd-good to construct a path $P_0$ with at most $30$ vertices whose ends $s_1$ and $s_2$ are $\gamma$-good split pairs such that, writing $V^*=(V\sm V(P_0))\cup s_1\cup s_2$, $A^*=A\cap V^*$ and $B^*=B\cap V^*$, we have either $|V^*|\equiv 6 \pmod 8$ and $|A^*|-|B^*|\equiv 2 \pmod 4$ or $|V^*|\equiv 2 \pmod 8$ and $|A^*|-|B^*|\equiv 0 \pmod 4$.
\item We extend the path $P_0$ to a path $P$ which contains all non-$\gamma$-good vertices such that 
$|V(P)| \equiv |V(P_0)| \pmod 8$ and $|V(P)\sm (V(P_0)\cap A)| \equiv |V(P) \sm (V(P_0)\cap B)| \pmod 4$ and whose ends are  $\gamma$-good split pairs $s_1$ and $s_3$.
\item Finally, we apply Lemma~\ref{lem:oddhampath} to find a Hamilton path $Q$ in $H'=H[(V\sm V(P)) \cup s_1 \cup s_3]$ with ends $s_1$ and $s_3$. This gives a Hamilton cycle $PQ$ in $H$.
\end{enumerate}
It suffices to show that we can construct the path $P_1$ in each case. Indeed, having constructed the path $P_0$, let $X$ be the set of all non-$\gamma$-good vertices in $V^*$; by adding a single further vertex to $X$ if necessary, we may assume that $q := |X|$ is even. Every vertex of $X$ is $(\beta_1,\beta_2)$-medium by Claim~\ref{clm:oddpartition} and by Proposition~\ref{prop:oddproperties}(d) we have $q \leq \frac{c}{\gamma^5-2\eps}n + 1 < \gamma n/25$. Write $X = \{x_1,\cdots,x_q\}$, and greedily choose distinct vertices $y_1,\cdots,y_q \in V^*\sm (s_1\cup s_2\cup X)$ so that for each $i \in [q]$ the pair $\{x_i,y_i\}$ is $\beta_2$-medium. We now form a path $Q$ by the following iterative process. Write $g_0 := s_2$. Then, for each $i \in [q]$ in turn we proceed as follows to choose connate pairs $f_{3i-2}, f_{3i-1}, f_{3i}$ and split pairs $g_{3i-2}, g_{3i-1}, g_{3i}$. If $\{x_i, y_i\}$ is a connate pair, set $f_{3i-1}:= \{x_i, y_i\}$, and apply Proposition~\ref{clm:medconnect}(i) twice to obtain $\gamma$-good split pairs $g_{3i-2}$ and $g_{3i-1}$ such that $g_{3i-2} \cup f_{3i-1}$ and $f_{3i-1} \cup g_{3i-1}$ are both edges of $H$. Next choose a $\gamma$-good split pair $g_{3i}$, and apply Proposition~\ref{prop:oddproperties}(g) twice to choose connate pairs $f_{3i-2}$ and $f_{3i}$ such that each of $g_{3i-3}f_{3i-2}g_{3i-2}$ and $g_{3i-1}f_{3i}g_{3i}$ is a path of length two in $H$. 
On the other hand, if $\{x_i, y_i\}$ is a split pair, set $g_{3i-1}:= \{x_i, y_i\}$, and apply Proposition~\ref{clm:medconnect}(ii) twice to obtain $\gamma$-good connate pairs $f_{3i-1}$ and $f_{3i}$ such that $g_{3i-1} \cup f_{3i-1}$ and $f_{3i} \cup g_{3i-1}$ are both edges of $H$. Next apply Proposition~\ref{prop:oddproperties}(h) twice to obtain $\gamma$-good split pairs $g_{3i}$ and $g_{3i-2}$ such that $g_{3i-2} \cup f_{3i-1}$ and $f_{3i} \cup g_{3i}$ are both edges of $H$, and finally apply Proposition~\ref{prop:oddproperties}(g) to choose a connate pair $f_{3i-2}$ such that $g_{3i-3}f_{3i-2}g_{3i-2}$ is a path of length two in $H$. If we choose each pair to be disjoint from $V(P'_0)$ and from all previously-chosen pairs, having made these choices for every $i \in [q]$ we obtain the desired path $P_0 = P'_0g_0f_1g_1f_2g_2\dots f_{3q}g_{3q}$ with ends $s_1$ and $s_3 := g_{3q}$.
Observe that we then have $|V(P_0) \sm V(P_0')| = 12q$, and that each $f_i$ is a connate pair and each $g_i$ is a split pair. Since $q$ is even it follows that $|V(P)|\equiv |V(P_0)| \pmod 8$ and $|V(P) \cap A|-|V(P)\cap B| \equiv |V(P_0)\cap A|-|V(P_0)\cap B| \pmod 4$, as required. This completes Step~(II). Finally, since $|V(P)| \leq |V(P_0)| + 12q \leq \gamma n/2$, we may then apply Lemma~\ref{lem:oddhampath} to find a Hamilton path $Q$ as claimed in Step~(III).

We now show how to construct the path $P_0$ in each case.

\medskip \noindent
{\bf Case A: $(m, d) = (0,0)$ or $(m, d) = (4,2)$.} Using Proposition~\ref{prop:oddproperties}(c) we can choose a $\gamma$-good split pair $s_1$. Then by Proposition~\ref{prop:oddproperties}(i) we can find a $\gamma$-good connate pair $p \in \binom{A}{2}$ such that $p \cup s_1$ is an edge of $H$. Finally, using Proposition~\ref{prop:oddproperties}(h) there is a $\gamma$-good split pair $s_2 \in \binom{V \sm s_1}{2}$ such that $p \cup s_2$ is an edge of $H$, and then $P_0 = s_1ps_2$ is the desired path. Observe that we then have either $|V^*|\equiv 0-2\equiv 6 \pmod 8$ and $|A^*|-|B^*|\equiv 0-2 \equiv 2 \pmod 4$, or $|V^*|\equiv 4-2\equiv 2 \pmod 8$ and $|A^*|-|B^*|\equiv 2-2 \equiv 0 \pmod 4$.

\medskip \noindent
{\bf Case B: $(m, d) = (2,2)$ or $(m, d) = (6,0)$.} Since the bipartition $\{A, B\}$ of $V(H)$ is odd-good, in this case $H$ must contain an even edge $e$. If $e$ contains a $\beta_2$-bad connate pair, then we apply Proposition~\ref{clm:medconnect}(iii) with $R=\emptyset$ to obtain two disjoint $\beta_2$-medium connate pairs $p_1$ and $p_2$ such that $p_1 \cup p_2$ is an edge of $H$. On the other hand, if $e$ does not contain a $\beta_2$-bad connate pair, then we may write $e = p_1 \cup p_2$ where $p_1$ and $p_2$ are disjoint connate pairs. In either case we obtain $\beta_2$-medium connate pairs $p_1$ and $p_2$ such that $p_1 \cup p_2$ is an edge of $H$. So by Proposition~\ref{clm:medconnect}(i) there are $\gamma$-good split pairs $s_1,s_2$ such that $s_1\cup p_1$ and $s_2\cup p_2$ are edges in $H$, and then $P_0=s_1p_1p_2s_2$ is the desired path. Note that either $|V^*| \equiv 2-4\equiv 6 \pmod 8$ and $|A^*|-|B^*|\equiv 2-0\equiv2 \pmod 4$ or $|V^*|\equiv 6-4\equiv 2 \pmod 8$ and $|A^*|-|B^*|\equiv 0-0\equiv0 \pmod 4$.

\medskip \noindent
{\bf Case C: $(m, d) = (4,0)$ or $(m, d) = (0,2)$.} Since the bipartition $\{A, B\}$ of $V(H)$ is odd-good, in this case $\Heven$ must have total $2$-pathlength at least two. That is, $H$ contains even edges $e_1$ and $e_2$ such that either $e_1$ and $e_2$ are disjoint or $|e_1 \cap e_2| = 2$.

Suppose first that $e_1$ and $e_2$ are disjoint. Similarly as in Case B, if $e_1$ contains a $\beta_2$-bad connate pair, then we use Proposition~\ref{clm:medconnect}(iii) with $R = e_2$ to obtain $\beta_2$-medium connate pairs $p_1$ and $p_2$ such that $e'_1 := p_1 \cup p_2$ is an even edge of $H$. By replacing $e_1$ with $e_1'$ if necessary, we may assume that $e_1 = p_1 \cup p_2$ where $p_1$ and $p_2$ are $\beta_2$-medium connate pairs, and the same argument applied to $e_2$ shows that we may assume that $e_2 = p_3 \cup p_4$ where $p_3$ and $p_4$ are $\beta_2$-medium connate pairs. Then by Proposition~\ref{clm:medconnect}(i) there are $\gamma$-good split pairs $s_1,s_2,s_3$ and $s_4$ such that $s_1p_1p_2s_3$ and $s_4p_3p_4s_2$ are vertex-disjoint paths in $H$. Then, by applying Proposition~\ref{prop:oddproperties}(i) twice, followed by Proposition~\ref{prop:oddproperties}(f), we obtain $\gamma$-good connate pairs $q_1,q_2$ and a split pair $s_5$ such that $P_0 = s_1p_1p_2s_3q_1s_5q_2s_4p_3p_4s_2$ is the desired path. Note that then either $|V^*|\equiv 4-18\equiv 2 \pmod 8$ and $|A^*|-|B^*|\equiv 0-0\equiv0 \pmod 4$, or $|V^*|\equiv 0-18\equiv 6 \pmod 8$ and $|A^*|-|B^*|\equiv 2-0\equiv2 \pmod 4$.

Now suppose that $|e_1 \cap e_2| = 2$, and write $f_1 = e_1\sm e_2$, $f_2 = e_1 \cap e_2$ and $f_3 = e_2 \sm e_1$, so $f_1f_2f_3$ is a path in $H$. Since $e_1$ and $e_2$ are both even edges, $f_1, f_2$ and $f_3$ are either all connate pairs or all split pairs. If $f_1$ is not $\beta_2$-medium, then we may apply Proposition~\ref{clm:medconnect}(iii) or~(iv) with $R = e_2$ to obtain an even edge $e'_1$ which is disjoint from $e_2$; we may then proceed as in the previous case with $e_1'$ and $e_2$ in place of $e_1$ and $e_2$. So we may assume that $f_1$ is $\beta_2$-medium, and by the same argument applied to $f_3$ we may assume that $f_3$ is $\beta_2$-medium.
If each of $f_1, f_2$ and $f_3$ is a connate pair, then we may apply Proposition~\ref{clm:medconnect}(i) to obtain $\gamma$-good split pairs $s_1,s_2$ such that $s_1\cup f_1$ and $s_2\cup f_3$ are both edges of $H$. Then we may take $P_0=s_1f_1f_2f_3s_2$, since we have either $|V^*|\equiv 4-6\equiv 6 \pmod 8$ and $|A^*|-|B^*|\equiv 0-2=2 \pmod 4$, or $|V^*|\equiv 0-6\equiv 2 \pmod 8$ 
and $|A^*|-|B^*|\equiv 2-2=0 \pmod 4$. 
On the other hand, if each of $f_1, f_2$ and $f_3$ is a split pair, then we may apply Proposition~\ref{clm:medconnect}(ii) followed by Proposition~\ref{clm:medconnect}(i) to first find $\gamma$-good connate pairs $p_1$ and $p_2$, and then $\gamma$-good split pairs $s_1$ and $s_2$, such that $P_0=s_1p_1f_1f_2f_3p_2s_2$ is a path in $H$; we then have either $|V^*|\equiv 4-10\equiv 2 \pmod 8$ and $|A^*|-|B^*|\equiv 0-0\equiv 0 \pmod 4$, or $|V^*|\equiv 0-10\equiv 6 \pmod 8$ and $|A^*|-|B^*|\equiv 2-0\equiv 2 \pmod 4$.

\medskip \noindent
{\bf Case D: $(m, d) = (6,2)$ or $(m, d) = (2,0)$.} Since the bipartition $\{A, B\}$ of $V(H)$ is odd-good, in this case either $H$ contains an even split edge or $\Heven$ has total 2-pathlength at least three.

Suppose first that $e$ is an even split edge of $H$. If $e$ contains a $\beta_2$-bad split pair then we may apply Proposition~\ref{clm:medconnect}(iv) to obtain disjoint $\beta_2$-medium split pairs $s'_1$ and $s'_2$ such that $e' := s'_1 \cup s'_2$ is an even edge of $H$. By replacing $e$ with $e'$ if necessary, we may assume that $e = s'_1 \cup s'_2$ where $s'_1$ and $s'_2$ are $\beta_2$-medium split pairs. We next apply Proposition~\ref{clm:medconnect}(ii) twice to obtain $\gamma$-good connate pairs $p_1$ and $p_2$ such that $p_1s_1's_2'p_2$ is a path, and then Proposition~\ref{prop:oddproperties}(h) twice to obtain $\gamma$-good split pairs $s_1$ and $s_2$ such that $P_0=s_1p_1s_1's_2'p_2s_2$ is the desired path in $H$. Note that then either $|V^*|\equiv 6-8\equiv 6 \pmod 8$ and $|A^*|-|B^*|\equiv 2-0=2 \pmod 4$, or $|V^*|\equiv 2-8\equiv 2 \pmod 8$ and $|A^*|-|B^*|\equiv 0-0=0 \pmod 4$. 

From now on, for the rest of Case D, we may assume that $H$ does not contain an even split edge and therefore any even edge is connate. So suppose now that $H$ contains three disjoint even connate edges $e_1,e_2$ and $e_3$. By using Proposition~\ref{clm:medconnect}(iii) as in previous cases to replace $e_1$, $e_2$ or $e_3$ if necessary, we may assume that we can write $e_1 = q_1 \cup q'_1$, $e_2 = q_2 \cup q'_2$ and $e_3 = q_3 \cup q'_3$ where each of $q_1, q_2, q_3, q'_1, q'_2$ and $q'_3$ is a $\beta_2$-medium connate pair. Then by several applications of Proposition~\ref{clm:medconnect}(i) and Proposition~\ref{prop:oddproperties}(g) we may choose connate pairs $p_1, p_2$ and $\gamma$-good split pairs $s_1,\cdots,s_6$ such that $P_0=s_1q_1q'_1s_3p_1s_4q_2q'_2s_5p_2s_6q_3q'_3s_2$ is the desired path in $H$. Note that then either $|V^*| \equiv 6-24\equiv 6 \pmod 8$ and $|A^*|-|B^*|\equiv 2-0=2 \pmod 4$, or $|V^*|\equiv 2-24\equiv 2 \pmod 8$ and $|A^*|-|B^*|\equiv 0-0=0 \pmod 4$.

Next suppose that $H$ contains even connate edges $e_1$, $e_2$ and $e_3$ such that $|e_1 \cap e_2| = 2$ and $e_3$ is disjoint from $e_1 \cup e_2$. Then exactly as in Case C we may form a path $P_0' = s_1f_1f_2f_3s'_1$ where $s_1$ and $s'_1$ are $\gamma$-good split pairs and $f_1,f_2$ and $f_3$ are $\gamma$-good connate pairs, and we may do this so that $e_3$ is disjoint from $V(P_0')$. Using Proposition~\ref{clm:medconnect}(iii) as in previous cases to replace $e_3$ if necessary, we may then assume that $e_3 = q_1 \cup q_2$ where $q_1$ and $q_2$ are disjoint $\beta_2$-medium connate pairs. By two applications of Proposition~\ref{clm:medconnect}(i) we then choose $\gamma$-good split pairs $s_2$ and $s_2'$ such that $s_2' \cup q_1$ and $s_2 \cup q_2$ are both edges of $H$; finally, using Proposition~\ref{prop:oddproperties}(g) we obtain a connate pair $p$ such that $s_1'ps'_2$ is a path of length two in $H$. This gives the desired path $P_0=s_1f_1f_2f_3s_1'ps_2'q_1q_2s_2$. Note that then either $|V^*|\equiv 6-16\equiv 6 \pmod 8$ and $|A^*|-|B^*|\equiv 2-0=2 \pmod 4$, or $|V^*|\equiv 2-16\equiv 2 \pmod 8$ and $|A^*|-|B^*|\equiv 0-0=0 \pmod 4$.

Finally, suppose that $H$ contains a path of three even connate edges $e_1$, $e_2$ and $e_3$ (appearing in that order). Let $f_1 := e_1 \sm e_2$, $f_2 := e_1 \cap e_2$, $f_3 := e_2 \cap e_3$, and $f_4 := e_3 \sm e_2$, so each of $f_1, f_2, f_3$ and~$f_4$ is a connate pair. If the pair $f_1$ is $\beta$-bad, then by applying Proposition~\ref{clm:medconnect}(iii) to $f_1$ with $R = e_2 \cup e_3$ we obtain an even edge $e'_1$ of $H$ disjoint from the path $f_2f_3f_4$ of two even connate edges. Since we assumed that $H$ contains no even split edge this edge $e'_1$ is also an even connate edge, and we may then proceed as in the previous case. So we may assume that~$f_1$ is $\beta$-medium, and likewise that $f_4$ is $\beta$-medium. Using Proposition~\ref{clm:medconnect}(i) twice we find 
$\gamma$-good split pairs $s_1,s_2$ such that $P_0 = s_1f_1f_2f_3f_4s_2$ is the desired path. Note that then either $|V^*|\equiv 6-8\equiv 6 \pmod 8$ and $|A^*|-|B^*|\equiv 2-0=2 \pmod 4$, or $|V^*|\equiv 2-8\equiv 2 \pmod 8$ and $|A^*|-|B^*|\equiv 0-0=0 \pmod 4$. \medskip

For the `moreover' part of the statement, suppose that we are given a $4$-graph $H$ as in the lemma, and also an odd-extremal partition $\{A', B'\}$ of $V(H)$. We can then identify the sets $A^{\bad}$ and $B^{\bad}$ and form the partition $\{A, B\}$ in time $O(n^4)$, and in time $O(n)$ we can identify which of cases A--D holds for $\{A, B\}$. Furthermore we can find the at most three even edges which we use to begin the construction of $P_0$ in time at most $O(n^{12})$, and for each application of Proposition~\ref{prop:oddproperties} or Proposition~\ref{clm:medconnect} to choose a pair we can find such a pair by exhaustive search in time $O(n^2)$. In this way we can form $P_0$ as in Step~(I) in time $O(n^{12})$ and then extend $P_0$ to $P$ as in Step~(II) in time $O(n^3)$. Finally, the application of Lemma~\ref{lem:oddhampath} for Step~(III) gives the desired path $Q$ in time $O(n^4)$.
\end{proof}

The running-time for Lemmas~\ref{thm:evenextralgo} and~\ref{thm:oddextremalalgo} could be improved by more careful arguments, but we abstain from this here, since these procedures do not provide the dominant term for the running-time of our main algorithm in Section~\ref{sec:algo}.

\section{An algorithm to find a Hamilton 2-cycle in a dense 4-graph}\label{sec:algo}

In this section we complete the proof of Theorem~\ref{2cycle} by describing an algorithm which finds a Hamilton 2-cycle in a dense $4$-graph in polynomial time, or certifies that no such cycle exists. First we show that we can find a Hamilton 2-cycle in a dense, connecting and absorbing $4$-graph in polynomial time. For this we use Edmonds's well known algorithm~\cite{Ed65} which finds a maximum matching in a $2$-graph $G$ of order $n$ in time $O(n^4)$; we refer to this procedure here as $\mathrm{MaximumMatching}(G)$.

\begin{prop}\label{prop:gencase}
Suppose that $1/n\ll \eps\ll\beta\ll\alpha\ll\kappa$ and that $n$ is even. There exists an algorithm Procedure~\ref{proc:gencase}($H$) such that the following holds. If $H$ is a $4$-graph of order $n$ with $\delta(H)\geq n/2-\eps n$ which is $\kappa$-connecting and $(\alpha,\beta)$-absorbing, then Procedure~\ref{proc:gencase}($H$) returns a Hamilton 2-cycle in $H$ in time $O(n^{32})$.
\end{prop}

\begin{proof}
Introduce constants satisfying $1/n\ll \eps\ll\gamma\ll\beta\ll\alpha\ll\lambda\ll \mu \ll \kappa.$
By combining the Procedures~\ref{proc:abspath} and~\ref{proc:longcycle} from Lemmas~\ref{lem:abspath2} and~\ref{lem:longcycle2} we can find a long cycle $C$ in $H$ and a graph $G$ on $V(H)\setminus V(C)$ which contains a perfect matching $M$ with $|M| \leq \gamma n$, and for each edge $e \in M$ there will be at least $2 \gamma n$ vertex-disjoint segments of $C$ which are absorbing structures for $e$. We can find this perfect matching by using $\mathrm{MaximumMatching}(G)$. We can now for each $e \in M$ find (by exhaustive search) a segment $P_e$ which is an absorbing structure for $e$ and which is disjoint from each segment $P_{e'}$ chosen for each previously-considered $e'\in M$. There will always be a segment available, as each of the fewer than $\gamma n$ previously-chosen segments $P_{e'}$ intersects at most two of the segments which could be chosen for $e$. Replacing each $P_e$ in $C$ by a path with vertex set $V(P_e) \cup e$ and the same ends as $P_e$ yields a Hamilton 2-cycle in $H$. For the complexity note that we can find an absorbing structure for each of the at most~$n$ edges in $M$ in $O(n)$ by exhaustive search. Therefore the dominant term for the running time is determined by the Procedure~\ref{proc:abspath}.
\end{proof}

Using the above proposition together with the results of Sections~\ref{sec:evenextr} and~\ref{sec:oddextr} we can now describe a polynomial-time algorithm, Procedure~\ref{proc:hamcycle}$(H)$, which, given a $4$-graph $H$ with $\delta(H)\geq n/2-\eps n$ (where $\eps > 0$ is a fixed constant), either finds a Hamilton $2$-cycle in $H$ or certifies that there is no such cycle. So the existence of this algorithm proves Theorem~\ref{2cycle}. In the latter case the certificate is a bipartition $\{A, B\}$ of $V(H)$ which is not both even-good and odd-good, as by Theorem~\ref{character} the existence of such a bipartition demonstrates that $H$ has no Hamilton $2$-cycle. Since it is straightforward to verify in polynomial time whether a given partition is even-good or odd-good, our choice for the certificate is justified. 

\begin{proof}[Proof of Theorem~\ref{2cycle}]
Introduce constants satisfying $1/n\ll\eps\ll\beta\ll\alpha\ll\kappa\ll c\ll 1,$ and let~$H$ be a $4$-graph of order $n$ with $\delta(H)\geq n/2-\eps n$. By Corollary~\ref{algodecide} there is an algorithm with running time $O(n^{25})$ which tests whether $H$ contains a Hamilton $2$-cycle and which, if this condition fails, returns a bipartition $\{A, B\}$ of $V(H)$ which is not both even-good and odd-good. We first run this algorithm and, if $H$ does not contain a Hamilton $2$-cycle, then we return a certifying bipartition. We may therefore assume that $H$ does contain a Hamilton $2$-cycle and that every bipartition of $V(H)$ is both even-good and odd-good, and in particular that $H$ is of even order. Note that we can test in time $O(n^8)$ whether $H$ is $\kappa$-connecting by counting, for each of the $3\binom{n}{4}$ possible disjoint pairs $p_1, p_2 \in \binom{V}{2}$, the number of paths of length two or three with ends $p_1$ and $p_2$. Similarly, we can test in time $O(n^{10})$ whether $H$ is $(\alpha,\beta)$-absorbing by counting, for each of the $\binom{n}{2}$ possible pairs $p \in \binom{V}{2}$, the number of octuples of vertices from $V$ which form an absorbing structure for $p$. Suppose that $H$ is not $\kappa$-connecting. Then $H$ must be $c$-even-extremal by Lemma~\ref{lem:con}, and Procedure~\ref{proc:evenpartition} (defined in Lemma~\ref{lem:con}) returns a $c$-even-extremal partition of $\{A, B\}$ in time $O(n^8)$. Procedure~HamCycleEven (see Lemma~\ref{thm:evenextralgo}) then returns a Hamilton $2$-cycle in time $O(n^{10})$. Likewise, if $H$ is not $(\alpha,\beta)$-absorbing then $H$ must be $c$-odd-extremal by Lemma~\ref{lem:absstruc}, and Procedure~\ref{proc:oddpartition} (defined in Lemma~\ref{lem:absstruc}) returns a $c$-odd-extremal partition of $\{A, B\}$ in time $O(n^6)$. We can then use Procedure~HamCycleOdd (see Lemma~\ref{thm:oddextremalalgo}) to return a Hamilton $2$-cycle in time $O(n^{12})$. This leaves only the case when $H$ is both $\kappa$-connecting and $(\alpha,\beta)$-absorbing; in this case we can use Procedure~\ref{proc:gencase} (see Proposition~\ref{prop:gencase}) to return a Hamilton $2$-cycle in time $O(n^{32})$ . So in each case we can find a Hamilton $2$-cycle in $H$ in time at most $O(n^{32})$, as claimed.
\end{proof}

\section{Tight Hamilton cycles in $k$-graphs}\label{sec:tight}

In this section we describe the proof of Theorem~\ref{tightcycle}, for which we use the following notation. For a function $f(n)$, we write $\HC(k, f)$ (respectively $\HP(k, f)$) to denote the $k$-graph tight Hamilton cycle (respectively tight Hamilton path) decision problem restricted to $k$-graphs $H$ with minimum codegree $\delta(H) \geq f(|V(H)|)$. On the other hand, for an integer $D$, we write $\oHC(k, D)$ (respectively $\oHP(k, D)$) to denote the $k$-graph tight Hamilton cycle (respectively tight Hamilton path) decision problem restricted to $k$-graphs $H$ with maximum codegree $\Delta(H) \leq D$. Our starting point is the following theorem of Garey, Johnson and Stockmeyer~\cite{GaJoSt76} on subcubic graphs (we say that a graph $G$ is \emph{subcubic} if $G$ has maximum degree $\Delta(G) \leq 3$). 

\begin{thm}[\cite{GaJoSt76}] \label{subcubic}
The problem of determining whether a subcubic graph admits a Hamilton cycle (\emph{i.e.} $\oHC(2, 3)$) is NP-complete.
\end{thm}

Fix any integers $k \geq 2$ and $D$. We first show that there are polynomial-time reductions from $\oHC(k, D)$ to $\oHC(2k-1, 2D)$ and from $\oHC(k, D)$ to $\oHC(2k, D)$. For this, let $H$ be a $k$-graph with vertex set $V$, let $A$ and $B$ be disjoint copies of $V$, and let $\vphi_A: A \rightarrow V$ and $\vphi_B: B \rightarrow V$ be the corresponding bijections. We define $\overline{H}_{2k-1}$ to be the $(2k-1)$-graph with vertex set $A \cup B$ whose edges are all sets $e\in\binom{A\cup B}{2k-1}$ such that either $\vphi_A(e\cap A) \in E(H)$ and $\vphi_B(e\cap B) \subseteq \vphi_A(e\cap A)$, or $\vphi_B(e\cap B) \in E(H)$ and $\vphi_A(e\cap A) \subseteq \vphi_B(e\cap B)$. Likewise we define $\overline{H}_{2k}$ to be the $2k$-graph with vertex set $A \cup B$ whose edges are all sets $e\in\binom{A\cup B}{2k}$ such that $\vphi_A(e\cap A) \in E(H)$, $\vphi_B(e\cap B) \in E(H)$. Then it is easy to check that 
\begin{enumerate}[noitemsep, label=(\alph*)]
\item either $H$, $\overline{H}_{2k-1}$ and $\overline{H}_{2k}$ all contain tight Hamilton cycles, or none of them does, and
\item if $\Delta(H) \leq D$, then $\Delta(\overline{H}_{2k-1}) \leq 2D$ and $\Delta(\overline{H}_{2k}) \leq D$,
\end{enumerate}
so this construction gives the desired reductions.

We next show that there are polynomial-time reductions from $\oHP(k, D)$ to $\HC(2k-1, \lfloor\tfrac{n}{2}\rfloor - k(D+1))$ and from $\oHC(k, D)$ to $\HC(2k, \tfrac{n}{2} - k(D+1))$. For the first reduction, let $H$ be a $k$-graph on $n$ vertices; an elementary reduction shows that we may assume without loss of generality that $k$ divides $n$. Set $\ell:=n/k$ and $U := V(H)$, and let $X$ be a set of size $|X|=\tfrac{(k-1)n}{k} = \ell(k-1)$ which is disjoint from $U$. Next, set $A := U \cup X$, and let $B$ be a set of size $|B|=|A|+1 = \ell(2k-1)+1$ which is disjoint from $A$. Define $H_{2k-1}$ to be the $(2k-1)$-graph with vertex set $A \cup B$ whose edges are all sets $e\in\binom{A \cup B}{2k-1}$ with $|A\cap e| \notin \{k,k+1\}$, or with $|A\cap e|=k$ and $A\cap e \in E(H)$, or with $|A\cap e|=k+1$ and such that $e' \not\subseteq A \cap e$ for every $e' \in E(H)$. For the second reduction let $H$ again be a $k$-graph on $n$ vertices, let $S:= V(H)$ and let $T$ be a set of size $n$ which is disjoint from $S$, and define $H_{2k}$ to be the $2k$-graph with vertex set $S \cup T$ whose edges are all sets $e\in\binom{S \cup T}{2k}$ such that $|S\cap e| \notin \{k,k+1\}$, or such that $|S\cap e|=k$ and $S \cap e \in E(H)$, or such that $|S\cap e|=k+1$ and $e' \not\subseteq S \cap e$ for every $e' \in E(H)$. Observe that then
\begin{enumerate}[noitemsep, label=(\alph*)]
\setcounter{enumi}{2}
\item $H_{2k-1}$ contains a tight Hamilton cycle if and only if $H$ contains a tight Hamilton path,
\item $H_{2k}$ contains a tight Hamilton cycle if and only if $H$ contains a tight Hamilton cycle, and
\item if $\Delta(H) \leq D$, then $\delta(H_{2k-1}) \geq |A|-k(D+1)$ and $\delta(H_{2k}) \geq n-(D+1)k$.
\end{enumerate}
Since $H_{2k-1}$ has precisely  $|A| + |B| = 2|A| + 1$ vertices and $H_{2k}$ has precisely $2n$ vertices, this establishes the desired reductions. 

Finally, observe that there are elementary polynomial-time reductions from $\oHC(k, D)$ to $\oHP(k, D)$ and from $\oHP(k, D)$ to $\oHC(k, D)$; together with the above reductions and Theorem~\ref{subcubic} this observation completes the proof of Theorem~\ref{tightcycle}.

\section{Concluding remarks}\label{sec:disc}

Theorem~\ref{tightcycle} demonstrates an interesting contrast between the tight Hamilton cycle problem and the perfect matching problem in $k$-graphs. These two problems share many similarities: both are NP-hard for $k$-graphs in general (see~\cite{GJ74}), and the minimum codegree threshold which ensures the existence of a perfect matching in a $k$-graph $H$ on $n$ vertices (determined asymptotically by K\"uhn and Osthus~\cite{pp-kuhn06} and exactly for large $n$ by R\"odl, Ruci\'nski and Szemer\'edi~\cite{RoRuSz09b}) is close to $n/2$, that is, asymptotically equal to the minimum codegree threshold for a tight Hamilton cycle (see Theorem~\ref{codeg}). 
However, the two problems exhibit different complexity status between these two codegree thresholds. Indeed, 
Keevash, Knox and Mycroft~\cite{KeKnMy13} and Han~\cite{Han} recently showed that the perfect matching problem can be solved in polynomial time in $k$-graphs $H$ with $\delta(H) \geq n/k$. 
This complements a previous result of Szyma\'nska~\cite{S}, who showed that for any $\eps > 0$ the problem remains NP-hard when restricted to $k$-graphs $H$ with $\delta(H) \geq n/k-\eps n$. By contrast, we have seen in this paper that the tight Hamilton cycle problem is NP-hard even when restricted to $k$-graphs $H$ with $\delta(H) \geq n/2 - C$ for a constant $C$, \emph{i.e.}, there is a significant difference in the minimum codegree thresholds needed to render each problem tractable.

We made no attempt to quantify the constant $\eps$ in Theorem~\ref{character} which arises from our proof. However, we conjecture that in fact Theorem~\ref{character} holds for $\eps = \frac{1}{6}$.

\begin{conjecture} \label{conj}
There exists $n_0$ such that the following statement holds. Let $H$ be a 4-graph on $n \geq n_0$ vertices with $\delta(H) \geq n/3$. Then $H$ admits a Hamilton 2-cycle if and only if every bipartition of $V(H)$ is both even-good and odd-good.
\end{conjecture}

If true, the minimum codegree condition of Conjecture~\ref{conj} would be essentially best possible. To see this, fix any $n$ which is divisible by 4 and take disjoint sets $X$, $Y$ and $Z$ each of size $\frac{n}{3} \pm 1$ such that $|X| \neq |Y|$ and $|X \cup Y \cup Z| = n$. Define $H$ to be the $4$-graph on vertex set $V := X \cup Y \cup Z$ whose edges are all sets $S \in \binom{V}{4}$ with $|S \cap X| \equiv |S \cap Y| \pmod 3$. It is easily checked that we then have $\delta(H) \geq n/3 - 4$ and that every bipartition of $V(H)$ is both even-good and odd-good. However, there is no Hamilton $2$-cycle in $H$. Indeed, since $4$ divides $n$, taking every other edge of such a cycle would give a perfect matching $M$ in $H$. Since each edge of $M$ covers equally many vertices of $X$ and $Y$ (modulo 3), the same is true of $M$ as a whole, contradicting the fact that $|X| \not\equiv |Y| \pmod 3$.

If Conjecture~\ref{conj} holds, then by the same argument used to establish Corollary~\ref{algodecide} we may determine in polynomial time whether a $4$-graph $H$ on $n$ vertices with $\delta(H) \geq n/3$ admits a Hamilton 2-cycle. Moreover, we speculate that under the weaker assumption that $\delta(H) \geq n/4$ it may be possible to prove a similar statement to Conjecture~\ref{conj} which considers partitions of $V(H)$ into three parts as well as into two parts. If so, this would allow us to determine in polynomial time whether $H$ contains a Hamilton 2-cycle under this weaker assumption. Such a result would neatly complement Theorem~\ref{NPhard}, which shows that for any $c < 1/4$ it is NP-hard to determine whether a 4-graph $H$ on $n$ vertices with $\delta(H) \geq cn$ contains a Hamilton 2-cycle. However, our proof of Theorem~\ref{character} relies extensively on $\eps$ being small; it seems that significant new ideas and techniques would be needed to prove Conjecture~\ref{conj} or this proposed extension. 

Finally, it would be very interesting to classify the values of $k$ and $\ell$ for which the minimum degree threshold needed to render the $k$-graph Hamilton $\ell$-cycle problem tractable is not asymptotically equal to the threshold which guarantees the existence of a Hamilton $\ell$-cycle in a $k$-graph (we assume for this discussion that $\mbox{P} \neq \mbox{NP}$). Theorems~\ref{codeg} and~\ref{2cycle} show that this is the case for $k=4$ and $\ell=2$. On the other hand, Dahlhaus, Hajnal and Karpi\'nski~\cite{DHK} showed that the thresholds are asymptotically equal for $k=2$ and $\ell=1$, whilst Theorems~\ref{codeg} and~\ref{tightcycle} show that the thresholds are asymptotically equal for $k \geq 3$ and $\ell = k-1$.
Combining Theorem~\ref{NPhard} and Theorem~\ref{codeg} we also find that the thresholds are 
asymptotically equal for any $k$ and $\ell$ such that $k-\ell$ does not divide $k$; all other cases remain open.

\section*{Acknowledgements}

We would like to thank an anonymous referee for a careful review and numerous suggestions for improvements in the presentation of this manuscript.

\bibliographystyle{siam}
\bibliography{cyclesrefs}

\newpage

\appendix

\section{Algorithmic details for Sections~\ref{sec:gen} and~\ref{sec:algo}}

In this section we give further details of each of the procedures from Sections~\ref{sec:gen} and~\ref{sec:algo}, to enable the reader to more easily verify the correctness and claimed running-time of each such procedure. In each case we will adopt the notation of the proof of the corresponding lemma, and so the algorithm provided here should be read in conjunction with the corresponding proof. Furthermore, we adopt the following constant hierarchy for all procedures in this section.

$$1/n\ll 1/D\ll\eps\ll\gamma\ll\beta\ll\omega\ll\vphi\ll\alpha\ll\lambda\ll \mu \ll \kappa\ll\eta\ll c\ll 1\;.$$

\medskip \noindent {\bf Lemma~\ref{lem:con}:} Let $H$ be a $4$-graph on $n$ vertices which satisfies $\delta(H)\geq n/2-\eps n$. If $H$ is not $\kappa$-connecting, then Procedure~\ref{proc:evenpartition}$(H)$ returns a $c$-even-extremal bipartition $\{A, B\}$ of $V(H)$ in time~$O(n^8)$.\medskip

\begin{procedure}[H]
\SetAlgoVlined
\SetAlgoNoEnd
\caption{EvenPartition($H$)}
\KwData{A dense $4$-graph $H$ which is not $\kappa$-connecting.}
\KwResult{A $c$-even-extremal bipartition $\{A, B\}$ of $V(H)$.}
\BlankLine
By exhaustive search, find disjoint pairs $\{a_1,a_2\},\{b_1,b_2\}\in\binom{V(H)}{2}$ such that $H$ contains fewer than $\kappa n^2$ paths of length $2$ and fewer than $\kappa n^4$ paths of length $3$ whose ends are $\{a_1,a_2\}$ and $\{b_1,b_2\}$.\\
Construct the edge-coloured complete $2$-graph $G$ on $V(H)$ as in the proof of Lemma~\ref{lem:con}.\\ 
Find a monochromatic triangle $T^*$ fulfilling Claim~\ref{clm:triangle} (for the given value of $\eta$) by exhaustive search.\\
\Return{$A := N_{\red}(T^*)$ and $B := V(H) \sm A$.}
\label{proc:evenpartition}
\end{procedure}

\medskip \noindent {\bf Lemma~\ref{lem:reservoir}:} Suppose that $1/n \ll \rho \ll \lambda, \kappa$. If $H=(V,E)$ is a $4$-graph of order $n$ which is $\kappa$-connecting, and $G$ is a $2$-graph on the same vertex set $V$ with $\delta(G)\geq n-\lambda n$, then Procedure~\ref{proc:selres}($H,G,\rho$) returns in time $O(n^{16})$ a subset $R\subseteq V$ such that
\begin{enumerate} [label=(\alph*)]
\item $(1-4\rho)\rho n\leq|R|\leq\rho n$,
\item for every $x\in V$ we have $|N_G(x)\cap R|\geq (1-35\lambda)|R|$ and
\item for every disjoint $p_1,p_2\in\binom{V}{2}$ there are at least $\frac{\kappa}{5}|R|$ internally disjoint paths of length at most three in $H[R\cup p_1\cup p_2]$ with ends $p_1$ and $p_2$.
\end{enumerate}

\begin{procedure}[H]
\SetAlgoVlined
\SetAlgoNoEnd
\caption{SelectReservoir($H,G,\rho$)}
\KwData{A $4$-graph $H$ with vertex set $V$, a $2$-graph $G$ with vertex set $V$ and a constant $\rho>0$.}
\KwResult{A reservoir set $R\subseteq V$.}
\BlankLine
$U:=\{\{p,p'\}\subseteq\binom{V}{2}\mid p\cap p'=\emptyset\}$ and $W:=\binom{V}{4}$.\\
$E_1:=\{\{\{p,p'\},S\}\mid\{p,p'\}\in U, S\in W$ and $H[S\cup p\cup p']$ contains a path with ends $p$, $p'\}$.\\
$E_1':=\{\{S,S'\}\in\binom{W}{2}\mid S\cap S'\neq\emptyset\}$.\\
$E_2:=\{\{S,S'\}\in\binom{W}{2}\mid\{u,v\}\in E(G)$ for all $u\in S,v\in S'\}$.\\
Construct graphs $G_1:=(U\cup W,E_1\cup E_1')$ and $G_2:=(W,E_2)$.\\
$R':=$\ref{proc:SelSet}($G_1,G_2,\frac{\rho}{4}n,\kappa,17\lambda,\rho$).\\
$R:=\bigcup_{S\in R'}S$.\\
\Return{$R$.}
\label{proc:selres}
\end{procedure}

\medskip \noindent {\bf Lemma~\ref{lem:absstruc}:} Let $H$ be a $4$-graph on $n$ vertices which satisfies $\delta(H)\geq n/2-\eps n$. If $H$ is not $(\alpha,\beta)$-absorbing, then Procedure~\ref{proc:oddpartition}($H$) returns a $c$-odd-extremal bipartition $\{A, B\}$ of $V(H)$ in time $O(n^6)$.\medskip

\begin{procedure}[H]
\SetAlgoVlined
\SetAlgoNoEnd
\caption{OddPartition($H$)}
\KwData{A dense $4$-graph $H$ which is not $(\alpha,\beta)$-absorbing.}
\KwResult{A $c$-odd-extremal bipartition $\{A, B\}$ of $V(H)$.}
\BlankLine
By exhaustive search, find an edge $\{x,y,x',y'\}$ of $H$ as in Claim~\ref{clm:abs1}.\\
Form the edge-coloured complete graph $K$ on $V(H)$ as in the proof of Lemma~\ref{lem:absstruc}.\\ 
\If{there is a monochromatic triangle $T$ satisfying the conditions of Claim~\ref{clm:abscase1}}
{\Return{$A := N_{\red}(T)$ and $B := V(H) \sm A$.}}
\Else{By exhaustive search, find a normal vertex $v\in V(H)$ which is contained in at most $4\vphi n^2$ red triangles.\\
\Return{$A := N_{\red}(v)$ and $B := V(H) \sm A$.}}
\label{proc:oddpartition}
\end{procedure}

\medskip \noindent {\bf Lemma~\ref{lem:abspath2}:} If $H$ is a $4$-graph of order $n$ with $\delta(H) \geq n/2 - \eps n$ which is $\kappa$-connecting and $(\alpha,\beta)$-absorbing, then Procedure~\ref{proc:abspath}($H$) returns a path $P$ and a graph $G$ on $V(H)$ with the following properties in $O(n^{32})$.
\begin{enumerate}[label=(\roman*)]
\item $P$ has at most $\mu n$ vertices.
\item Every vertex of $V(H) \setminus V(P)$ is contained in at least $n-\lambda n$ edges of $G$. 
\item For any edge $e$ of $G$ which does not intersect $V(P)$ there are at least $2\gamma n$ vertex-disjoint segments of $P$ which are absorbing structures for $e$.
\end{enumerate}

\begin{procedure}[H]
\SetAlgoVlined
\SetAlgoNoEnd
\caption{AbsorbingPath($H$)}
\KwData{A dense, $\kappa$-connecting and $(\alpha,\beta)$-absorbing $4$-graph $H$.}
\KwResult{An absorbing path $P$ in $H$ and a graph $G$ on $V(H)$.}
\BlankLine
Set $W:=V(H)^8$, set $U$ to be the set of all $\beta$-absorbable pairs of vertices of $H$, and form the graph $G := (V(H), U)$.\\
Set $V_1:=U\cup W$, $E_1:=\{\{p,T\} : p \in U, T \in W \mbox{ and $T$ is an absorbing structure for $p$}\}$, $E_1':=\{\{T,T'\} : T, T' \in W \mbox{ and } T\cap T'\neq\emptyset\}$ and form the graph $G_1:=(V_1,E_1\cup E_1')$.\\
Set $\T :=$\ref{proc:SelSet}$(G_1,\emptyset,\beta^2 n,\beta,1,\beta^2)$.\\
Delete all elements from $\T$ which are not an absorbing structure for some $\beta$-absorbable pair.\\
Enumerate $\T$ as $\{T_1,\cdots,T_q\}$ and choose corresponding paths $P_i$ for each $1 \leq i\leq q$.\\
Set $Q := \bigcup_{i=1}^q T_i$ and $X := \{v\in V(H) \sm Q : d_G(v) < (1-\lambda)n\}$.\\
Greedily form a path $P_0$ with $X \subseteq V(P_0) \subseteq V(H) \sm Q$ as in the proof of Lemma~\ref{lem:abspath2}.\\
Greedily choose paths $Q_1, \dots, Q_q$ of length at most three connecting the paths $P_0, \dots, P_q$; each $Q_i$ may be chosen by exhaustive search.\\
\Return{$P := P_0Q_1P_1 \dots Q_qP_q$ and $G$.}
\label{proc:abspath}
\end{procedure}

\medskip \noindent {\bf Theorem~\ref{thm:Erdoes}:} Here we give a proof of the algorithmic statement. For this we define a \emph{multi-$k$-graph $H$} to consist of a vertex set~$V$ and a multiplicity function $m_H: \binom{V}{k} \to \{0\} \cup \N$. We call $m_H(e)$ the \emph{multiplicity} of $e$, and always count `with multiplicity', so, for example, the number of edges of $H$ is $e(H) := \sum_{e \in \binom{V}{k}} m_H(e)$, and the degree of a vertex $v \in V$ is $d_H(v) := \sum_{e \in \binom{V}{k} : v \in e} m_H(e)$. We begin with the following proposition.

\begin{prop} \label{prop:restrict}
Suppose that $1/n \ll 1/\ell, 1/r \ll d, 1/k$. Let $H$ be a multi-$k$-graph on~$n$ vertices with $e(H) \geq d\binom{n}{k}$ in which all multiplicities are at most $r$. Then there exists a set $X \subseteq V(H)$ of size $k\ell$ such that $e(H[X]) \geq d\binom{\ell}{k}$. Moreover, such a set can be found in time~$O(n^k)$.
\end{prop}

\begin{proof}
We proceed by induction on $k$; the base case $k=1$ is trivial. Suppose therefore that $k \geq 2$ and the proposition holds with $k-1$ in place of $k$. In time $O(n^k)$ we may choose vertices $v_1, \dots, v_\ell$ of $H$ each with degree at least $\frac{d}{2} \binom{n}{k-1} + r \ell \binom{n}{k-2}$ in $H$; such vertices must exist as otherwise we would have
$$ \ell \cdot r \binom{n}{k-1} + n \cdot \left(\frac{d}{2} \binom{n}{k-1} + r \ell \binom{n}{k-2}\right) \geq \sum_{v \in V}d_H(v) = k e(H) \geq d k\binom{n}{k},$$
a contradiction. Observe that each $v_i$ then lies in at least $\frac{d}{2} \binom{n}{k-1}$ edges which do not contain any $v_j$ with $j \neq i$. Now form a multi-$(k-1)$-graph $H'$ with vertex set $V' := V(H) \sm \{v_1, \dots, v_\ell\}$ by taking each $(k-1)$-set $S \in \binom{V'}{k-1}$ to have multiplicity $m_{H'}(S) := \sum_{i=1}^\ell m_{H}(S \cup \{v_i\})$. Then by choice of the vertices $v_i$ we have $e(H') \geq \ell \cdot \frac{d}{2} \binom{n}{k-1}$. So by our induction hypothesis (with $r \ell$ in place of $r$) we can find in time $O(n^{k-1})$ a set $X' \subseteq V'$ of size $(k-1)\ell$ such that $H'[X']$ has at least $\ell \frac{d}{2} \binom{\ell}{k-1} \geq d \binom{\ell}{k}$ edges; taking $X := X' \cup \{v_1, \dots, v_\ell\}$ gives the desired set.
\end{proof}

The following corollary is the algorithmic version of Theorem~\ref{thm:Erdoes}.

\begin{coro}[\cite{Er64}]\label{thm:Erdoesalgo}
Suppose that $1/n\ll d, 1/f, 1/k$. Let $F$ be a $k$-partite $k$-graph on $f$ vertices. If $H$ is a $k$-graph on $n$ vertices with $e(H)\geq d\binom{n}{k}$, then we can find a copy of $F$ in $H$ in time $O(n^k)$.
\end{coro}

\begin{proof}
Introduce a constant $\ell$ with $1/n \ll 1/\ell \ll d, 1/f, 1/k$. By Proposition~\ref{prop:restrict} (with $r=\ell$) we may find in time $O(n^k)$ a set $X \subseteq V(H) $ of size $k\ell$ such that $e(H[X]) \geq d\binom{\ell}{k}$. The non-algorithmic part of Theorem~\ref{thm:Erdoes} then implies that $H[X]$ contains a copy of $F$, and we can find such a copy in constant time by exhaustive search.
\end{proof}

\medskip \noindent {\bf Lemma~\ref{lem:longcycle2}:} Suppose that $n$ is even. Let $H$ be a $4$-graph of order $n$ with $\delta(H)\geq n/2-\eps n$ which is $\kappa$-connecting. Also let $P_0$ be a $2$-path in $H$ on at most $\mu n$ vertices, and let $G$ be a $2$-graph on $V(H)$ such that each vertex $v \in V(H) \setminus V(P_0)$ has $d_G(v) \geq (1-\lambda) n$. Then Procedure~\ref{proc:longcycle}$(H,G,P_0)$ returns in time $O(n^{16})$ a $2$-cycle $C$ on at least $(1-\gamma) n$ vertices such that $P_0$ is a segment of $C$ and $G[V(H)\sm V(C)]$ contains a perfect matching.\medskip

\begin{procedure}[H]
\SetAlgoVlined
\SetAlgoNoEnd
\caption{LongCycle($H,G,P_0$)}
\KwData{A dense and $\kappa$-connecting $4$-graph $H$, a very dense $2$-graph $G$ on $V(H)$, and a $2$-path $P_0$ in $H$.}
\KwResult{An long cycle $C$ which contains $P_0$ as a segment and such that $G[V(H)\setminus V(C)]$ contains a perfect matching.}
\BlankLine
Set $V':=V(H)\setminus V(P_0)$, $H':=H[V']$, and $G':=G[V']$.\\
R:=\ref{proc:selres}($H',G',2\gamma/3$).\\
Extend $P_0$ by a single edge at each end to a path $P_0'$.\\
Set $L:=V(H)\setminus(V(P_0')\cup R)$ and $R':=R\setminus V(P_0')$.\\
Extend $P_0'$ greedily by only using vertices from $L$ to a path $P$ of length at least $(1/2-\mu)n$.\\
Set $L:=V(H)\setminus(V(P)\cup R)$ and $R':=R\setminus V(P)$.\\
\While{$|L|>\frac{\gamma}{3}n$}{
\If{$e(H[L])>\mu\binom{|L|}{4}$}{
Use Theorem~\ref{thm:pathindensegraph} to find a path $P'$ in $H[L]$ on at least $\frac{\mu|L|}{4}-1$ vertices.\\
Only using vertices from $R'$ find a path $Q$ of length at most three (by exhaustive search) such that $P:=PQP'$ is a path.\\
Set $L:=V(H)\setminus(V(P)\cup R)$ and $R':=R\setminus V(P)$.\\
}
\Else{
Find $I$, $J$, and $H_0$ as in Claim~\ref{clm:segment} by exhaustive search.\\
Use Theorem~\ref{thm:Erdoes} to find a complete $3$-partite $3$-graph $K$ with all vertex classes of size $D/3$.\\
Derive the complete $4$-partite $4$-graph $K'$ from $K$ and $J$.\\
Find a Hamilton path $Q$ in $K'$.\\
Delete $P[I]$ from $P$ and call the resulting two segments $P_1$ and $P_2$.\\
Find by exhaustive search at most $8$ vertices in $R'$ forming two paths $Q_1$ and $Q_2$ such that $P:=P_1Q_1QQ_2P_2$ is a path.\\
Set $L:=V(H)\setminus(V(P)\cup R)$ and $R':=R\setminus V(P)$.\\
}
}
By exhaustive search find at most $4$ vertices in $R'$ to close $P$ to a cycle $C$.\\
\Return{$C$.}
\label{proc:longcycle}
\end{procedure}

\medskip \noindent {\bf Lemma~\ref{prop:gencase}:} Suppose that $n$ is even. If $H$ is a $4$-graph of order $n$ with $\delta(H)\geq n/2-\eps n$ which is $\kappa$-connecting and $(\alpha,\beta)$-absorbing, then Procedure~\ref{proc:gencase}($H$) returns a Hamilton cycle in $H$ in time $O(n^{32})$.\medskip

\begin{procedure}[H]
\SetAlgoVlined
\SetAlgoNoEnd
\caption{NonExtremalCase($H$)}
\KwData{A dense, $\kappa$-connecting and $(\alpha,\beta)$-absorbing $4$-graph $H$.}
\KwResult{A Hamilton cycle $C$ in $H$.}
\BlankLine
Set $(P,G):=$\ref{proc:abspath}$(H)$.\\
Set $C:=$\ref{proc:longcycle}$(H,G,P)$.\\
Set $M:=\mathrm{MaximumMatching}(G[V(H)\sm V(C)])$.\\
\For{$e\in M$}{
Find a segment $P_e$ in $P$ which is an absorbing structure for $e$ and which is disjoint from each segment $P_{e'}$ chosen for each previously-considered $e' \in M$.\\
Replace $P_e$ in $C$ by a path with vertex set $V(P_e) \cup e$ and the same ends as $P_e$. }
\Return{$C$.}
\label{proc:gencase}
\end{procedure}

\medskip \noindent {\bf Theorem~\ref{2cycle}:} If $H$ is a $4$-graph of order $n$ with $\delta(H) \geq n/2-\eps n$, then Procedure~\ref{proc:hamcycle} $(H,\{A, B\})$ either returns a Hamilton cycle in $H$ or a non-even-good or non-odd-good bipartition of $V(H)$ in time $O(n^{32})$.\medskip

\begin{procedure}[H]
\SetAlgoVlined
\SetAlgoNoEnd
\caption{HamCycle($H$)}
\KwData{A dense $4$-graph $H$.} 
\KwResult{A Hamilton $2$-cycle in $H$ or a certificate for non-existence.}
\BlankLine
\If{$H$ contains a Hamilton cycle (Corollary~\ref{algodecide})}{
  \If{$H$ is not $\kappa$-connecting}
    {Set $\{A, B\}:=$\ref{proc:evenpartition}$(H)$. \\
    \Return{$HamCycleEven(H,\{A, B\})$}.}
  \ElseIf{$H$ is not $(\alpha,\beta)$-absorbing}
    {Set $\{A, B\}:=$\ref{proc:oddpartition}$(H)$. \\
    \Return{$HamCycleOdd(H,\{A, B\})$}.}
  \Else{\Return{\ref{proc:gencase}($H$)}.}}
\Else{
  \Return{A bipartition $\{A, B\}$ of $V(H)$ which is not both even-good and odd-good.}}
\label{proc:hamcycle}
\end{procedure}

\section{Full proofs for the reductions claimed in Section~\ref{sec:tight}}

In this section we give full details of the proofs of the the correctness of the polynomial-time reductions claimed in the proof of Theorem~\ref{tightcycle} in  Section~\ref{sec:tight}. We do this through the following propositions. Throughout this section we write simply `path' and `cycle' to mean tight path and tight cycle, as these are the only types of paths and cycles considered here.

\begin{prop}\label{Fact:HP} For any $k \geq 2$ and any $D$, there is a polynomial-time reduction from $\oHP(k,D)$ to $\oHC(k,D)$ and a polynomial-time reduction from $\oHC(k,D)$ to $\oHP(k,D)$.
\end{prop}

\begin{proof}
First, we show that there is a polynomial-time reduction from $\oHP(k,D)$ to $\oHC(k,D)$. If $D < 2$, then $\oHP(k,D)$ is trivial, so assume that $D \geq 2$. Let $H$ be a $k$-graph with $\Delta(H)\leq D$.  For each of the at most $n^{2k}$ ordered $2k$-tuples $(x_1,\cdots,x_k, y_1,\cdots,y_k)$ such that $x = (x_1,\cdots,x_k)$ and $y=(y_1,\cdots,y_k)$ are both edges of $H$ we construct an altered graph $H_{(x,y)}$ by deleting all edges $e \in E(H)$ with $x_k \in e$ or $y_1 \in e$ and then adding the edges $\{x_{i+1}, x_{i+2}, \dots, x_k, y_1, \dots, y_{i}\}$ for $0 \leq i \leq k$, so $(x_1,\cdots,x_k,y_1,\cdots,y_k)$ is a path in $H_{(x, y)}$. We then test each $H_{(x, y)}$ for a Hamilton cycle. The original graph $H$ contains a Hamilton path starting with $y$ and ending with $x$ if and only if the altered graph $H_{(x, y)}$ contains a Hamilton cycle, and since $D \geq 2$ and $\Delta(H) \leq D$ we have $\Delta(H_{(x, y)}) \leq D$. 

Now, we show that there is a polynomial-time reduction for the other direction as well. Given a $k$-graph $H$ with $\Delta(H)\leq D$, for each of the at most $n^{2k}$ sequences of $2k$ vertices $S=(x_{1},\cdots,x_{k},y_{k},\cdots,y_{1})$ such that $S$ is a path in $H$ we construct a graph $H_S$ by deleting every edge $e$ of $H$ which intersects $S$ except for those edges $e$ such that $e\cap S = \{x_1,\cdots,x_{\ell}\}$ or $e\cap S = \{y_1,\cdots,y_{\ell}\}$ for some $\ell \leq k$. Afterwards we test each $H_S$ for a Hamilton path. Observe that we have $\Delta(H_S) \leq \Delta(H) \leq D$, and that there is a Hamilton cycle in $H$ if and only if one of the altered graphs $H_S$ contains a Hamilton path.
\end{proof}

\begin{prop} \label{redkto2k-1}
For any $k \geq 2$ and any $D$, there is a polynomial-time reduction from $\oHC(k,D)$ to $\oHC(2k-1,2D)$.
\end{prop}

\begin{proof} Given a $k$-graph $H$ with vertex set $V=\{v_1,\cdots,v_n\}$, we construct a $(2k-1)$-graph $\overline{H}_{2k-1}$ as follows. Fix disjoint copies $A=\{v_1^{A},\cdots,v_n^{A}\}$ and $B=\{v_1^{B},\cdots,v_n^{B}\}$ of $V$, and let $\vphi_A:A\rightarrow V$ and $\vphi_B:B\rightarrow V$ be the natural bijections (so $\vphi(v_i^{A}) = v_i = \vphi(v_i^B)$). For convenience we will not always mention the explicit bijections; instead we say that vertices $a \in A$, $b \in B$ and $x \in V$ \emph{correspond} if $\vphi_A(a) = \vphi_B(b) = x$. We take $A \cup B$ to be the vertex set of $\overline{H}_{2k-1}$, and the edges of $\overline{H}_{2k-1}$ are the sets $e\in\binom{A\cup B}{2k-1}$ such that either $\vphi_A(e\cap A)$ is an edge of $H$ and $\vphi_B(e\cap B) \subseteq \vphi_A(e\cap A)$, or $\vphi_B(e\cap B)$ is an edge of $H$ and $\vphi_A(e\cap A) \subseteq \vphi_B(e\cap B)$. 
It is then sufficient to show that
\begin{enumerate}[noitemsep, label=(\roman*)]
\item $H$ contains a Hamilton cycle if and only if $\overline{H}_{2k-1}$ contains a Hamilton cycle, and
\item if $\Delta(H) \leq D$, then $\Delta(\overline{H}_{2k-1}) \leq 2D$.
\end{enumerate}

For~(i), first observe that if  $C = (v_1,\cdots,v_n)$ is a Hamilton cycle in $H$, then $\overline{C} := (v_1^{A},v_1^{B},\cdots,$ $v_n^{A},v_n^{B})$ is a Hamilton cycle in $\overline{H}_{2k-1}$. Indeed, every consecutive subsequence $S$ in $\overline{C}$ of length $2k-1$ either contains $k$ vertices from $A$ or contains $k$ vertices from $B$. These $k$ vertices correspond to $k$ consecutive vertices of $C$ and therefore correspond to some edge $e \in E(H)$, and the remaining $k-1$ vertices of $S$ correspond to a subset of $e$, so $S$ forms an edge of $\overline{H}_{2k-1}$. 
Now suppose that some cyclic ordering of the vertices of $\overline{H}_{2k-1}$ gives a Hamilton cycle $C$ in $\overline{H}_{2k-1}$. It suffices to show that if we delete every vertex of $B$ from this sequence, then every $k$ consecutive vertices $(a_1,\cdots,a_k)$ in the remaining subsequence correspond to an edge of $H$, since this subsequence would then correspond to a Hamilton cycle in $H$. For this, let $Q$ be the subsequence of length $2k-1$ in $C$ beginning with $a_1$, and let $Q'$ be the subsequence of length $2k-1$ in $C$ beginning with the vertex subsequent to $a_1$ in $C$. 
Suppose first that $Q$ contains all of the vertices $a_1,\cdots,a_k$. In this case $\{a_1,\cdots,a_k\}$ must correspond to an edge of $H$, since $Q \in E(\overline{H}_{2k-1})$. Now suppose instead that $Q$ does not contain $a_{k}$. The fact that $Q \in E(\overline{H}_{2k-1})$ then implies that $Q$ contains $k$ vertices of $B$ which correspond to an edge of $H$, and that $a_1,\cdots,a_{k-1}$ each correspond to vertices in $B \cap Q$. 
Since $Q \sm Q' = \{a_1\}$, and $Q'$ is also an edge of $E(\overline{H}_{2k-1})$, we must have $Q' \sm Q = \{a_{k}\}$. It follows that $a_{k}$ corresponds to a vertex of $B\cap Q' = B\cap Q$. Thus the $k$ vertices $a_1,\cdots,a_k$ correspond to the $k$ vertices of $Q \cap B$, and so correspond to an edge of $H$.

For~(ii), fix a set $S \in \binom{A\cup B}{2k-2}$. It is only possible that $S$ is included in an edge of $\overline{H}_{2k-1}$ if either $|A \cap S| = k-2$, $|B \cap S| = k$ and the vertices of $B \cap S$ each correspond to vertices of $A \cap S$, or the same holds with the roles of $A$ and $B$ reversed, or $|A \cap S| = |B \cap S| = k-1$ and at least $k-2$ vertices of $A \cap S$ correspond to vertices of $B \cap S$. In the first case we have $d_{\overline{H}_{2k-1}}(S) = 2$, since there are precisely two vertices which can be added to $S$ to form an edge of $\overline{H}_{2k-1}$, namely the vertices of $A \sm S$ which correspond to vertices of $B \cap S$. Likewise we also have $d_{\overline{H}_{2k-1}}(S) = 2$ in the second case. Finally, suppose that $|A \cap S| = |B \cap S| = k-1$. If exactly $k-2$ vertices of $A \cap S$ correspond to vertices of $B \cap S$, then again there are two edges of $\overline{H}_{2k-1}$ containing $S$, formed by adding the vertex of $A \sm S$ which corresponds to a vertex of $B \cap S$ or vice versa. If instead all $k-1$ vertices of $A \cap S$ correspond to vertices of $B \cap S$, then we can form an edge of $\overline{H}_{2k-1}$ containing $S$ by adding any vertex of $A$ or $B$ which corresponds to a neighbour in $H$ of the corresponding $k-1$ vertices of $H$. Since $\Delta(H) \leq D$ there are at most $2D$ such vertices. So in all cases we have $d_{\overline{H}_{2k-1}}(S) \leq 2D$, as required.
\end{proof}

\begin{prop}\label{Lem:LowMaxDeg}
For any $k \geq 2$ and any $D$, there is a polynomial-time reduction from $\oHC(k,D)$ to $\oHC(2k,D)$.
\end{prop}

\begin{proof}
Given a $k$-graph $H$ with vertex set $V=\{v_1,\cdots,v_n\}$, we take disjoint copies $A=\{v_1^{A},\cdots,v_n^{A}\}$ and $B=\{v_1^{B},\cdots,v_n^{B}\}$ of $V$ and define $\vphi_A$ and $\vphi_B$ as in the proof of Proposition~\ref{redkto2k-1}. We then construct a $2k$-graph $\overline{H}_{2k}$ as follows: the vertex set of $\overline{H}_{2k}$ is $A \cup B$, and the edges of $\overline{H}_{2k}$ are all sets $e \in\binom{A\cup B}{2k}$ such that $\vphi_A(e\cap A)$ and $\vphi_B(e\cap B)$ are both edges of $H$.
It is then sufficient to show that
\begin{enumerate}[noitemsep, label=(\roman*)]
\item $H$ contains a Hamilton cycle if and only if $\overline{H}_{2k}$ contains a Hamilton cycle, and
\item $\Delta(\overline{H}_{2k}) \leq \Delta(H)$.
\end{enumerate}

To show~(i), first assume that $H$ contains a Hamilton cycle $C = (v_1,\cdots,v_n)$. Then $\overline{C} := (v_1^{A},v_1^{B},\cdots,v_n^{A},v_n^{B})$ is a Hamilton cycle in $\overline{H}_{2k}$, since for every consecutive subsequence $S$ of length $2k$ in $\overline{C}$, each of the sets $S \cap A$ and $S \cap B$ corresponds to a consecutive subsequence in $C$ of length $k$, that is, an edge of $H$, so $S$ is an edge of $\overline{H}_{2k}$. For the other direction, suppose that $C = (v_1,\cdots,v_{2n})$ is a Hamilton cycle in $\overline{H}_{2k}$. Then for any $v_i\in A$ the set $Q=\{v_i,\cdots,v_{i+2k-1}\}$ is an edge of $\overline{H}_{2k}$, so $Q\cap A$ contains exactly $k$ vertices from $A$, and moreover these $k$ vertices correspond to an edge of $H$. So if we delete all vertices of $B$ from the sequence $(v_1,\cdots,v_{2n})$, the resulting subsequence corresponds to a Hamilton cycle in $H$.

For~(ii), let $S$ be a set of $2k-1$ vertices of $\overline{H}_{2k}$. It is only possible that $S$ is included in an edge of $H$ if $S\cap A$ is an edge of $H$ and $|S \cap B| = k-1$, or the same with the roles of $A$ and $B$ reversed; without loss of generality we assume the former. A necessary condition for $\{x\} \cup S$ to be an edge of $\overline{H}_{2k}$ is then that $x\in B$ and that $\vphi_B(\{x\} \cup (S\cap B)) \in E(H)$. The number of such vertices $x$ is at most $d_H(\vphi_B(S\cap B)) \leq \Delta(H)$, so $d_{\overline{H}_{2k}}(S) \leq \Delta(H)$.
\end{proof}

It is convenient to note that the problem of deciding whether a $k$-graph with $\Delta(H) \leq D$ contains a Hamilton path reduces to the problem of whether a $k$-graph $H$ with $\Delta(H) \leq D$ whose order is divisible by $k$  contains a Hamilton path. We refer to the latter problem as $\oHP(k,D)^\bot$.

\begin{prop}\label{Fact:HPmod} For any $k \geq 2$ and any $D$, there is a polynomial-time reduction from $\oHP(k,D)$ to $\oHP(k,D)^\bot$.
\end{prop}

\begin{proof}
Let $H$ be a $k$-graph on $n$ vertices; we may assume that $\Delta(H) \geq 2$, as otherwise the Hamilton path problem is trivial. Let $0 \leq r \leq k-1$ and $\ell \geq 0$ be such that $n = \ell k+r$. For every ordered $(k-1)$-set of vertices $T = (v_1,\cdots,v_{k-1})$ in $H$ we form a $k$-graph $H_T$ on $(\ell+1)k$ vertices by adding $k-r$ new vertices $x_1,\cdots,x_{k-r}$ and new edges $\{x_i, \dots, x_{k-r}, v_1, \dots, v_{r+i-1}\}$ for $i \in [k-r]$, so $(x_1,\cdots,x_{k-r}, v_1,\cdots,v_{k-1})$ is a path in $H_T$. 
The resulting graph $H_T$ has order divisible by $k$ and maximum codegree $\Delta(H_T) = \Delta(H)$, and $H$ contains a Hamilton path if and only if one of the altered graphs $H_T$ created in this way contains a Hamilton path.
\end{proof}

\begin{prop} \label{hptohcodd}
For any $k \geq 2$ and any $D$ there is a polynomial-time reduction from $\oHP(k,D)$ to $\HC(2k-1,\lfloor\tfrac{n}{2}\rfloor - (D+1)k)$.
\end{prop}

\begin{proof}
By Proposition~\ref{Fact:HPmod} it suffices to give a polynomial-time reduction from $\oHP(k,D)^\bot$ to $\HC(2k-1,\lfloor\tfrac{n}{2}\rfloor - (D+1)k)$. So let $H$ be a $k$-graph on $n$ vertices, where $k$ divides $n$, and set $\ell:=n/k$. We define a $(2k-1)$-graph $H_{2k-1}$ as follows. Let $U:=V(H)$ and let $X$ be a set of size $|X|=\tfrac{(k-1)n}{k} = \ell(k-1)$ which is disjoint from $U$. Set $A := U \cup X$, and let $B$ be a set of size $|B|=|A|+1 = \ell(2k-1)+1$ which is disjoint from $A$. We take the vertex set of $H_{2k-1}$ to be $A \cup B$, and the edges of $H_{2k-1}$ to be all sets $e\in\binom{A \cup B}{2k-1}$ with $|A\cap e| \notin \{k,k+1\}$, or with $|A\cap e|=k$ and $A\cap e \in E(H)$, or with $|A\cap e|=k+1$ and such that $e' \not\subseteq A \cap e$ for every $e' \in E(H)$. Since $H_{2k-1}$ has $|A| + |B| = 2|A| + 1$ vertices, it then suffices to show that 
\begin{enumerate}[noitemsep, label=(\alph*)]
\item $H$ contains a Hamilton path if and only if $H_{2k-1}$ contains a Hamilton cycle, and
\item if $\Delta(H) \leq D$, then $\delta(H_{2k-1}) \geq |A|-k(D+1)$.
\end{enumerate}
For convenience we define the \emph{type} of an edge $e \in E(H_{2k-1})$ to be the pair $(|e \cap A|, |e \cap B|)$.

For~(a), first assume that there is a Hamilton path $P=(h_1,\cdots,h_n)$ in $H$, and write $X=\{x_1,\cdots,x_{\ell (k-1)}\}$ and $B=\{b_1,\cdots, b_{\ell k}, b'_1, \cdots, b'_{\ell (k-1)+1}\}$. Then 
\begin{align*}
C= &\big(x_1, \cdots, x_{k-1}, b_1, \cdots, b_k, \cdots, x_{(\ell-1)(k-1)+1},
\cdots, x_{\ell(k-1)}, b_{(\ell-1)k+1}, \cdots, b_{\ell k}, \\
& h_{1}, \cdots, h_{k}, b'_{1}, \cdots, b'_{k-1}, \cdots, h_{n-k+1},
\cdots, h_{n}, b'_{(\ell-1)(k-1)+1}, \cdots, b'_{\ell(k-1)}, 
b'_{\ell(k-1)+1}\big)\;.
\end{align*}
is a Hamilton cycle $C$ in $H$.
That is, $C$ begins with edges of type $(k-1,k)$ which include all the vertices of $X$, followed by edges of type $(k,k-1)$ which include all of the vertices of $U$. The final vertex from $B$ then returns $C$ to edges of type $(k-1,k)$ to close the cycle. Note that all vertices of $H_{2k-1}$ are contained in $C$ and that every consecutive subsequence of $2k-1$ vertices of $C$ forms an edge of $H_{2k-1}$.

Now assume instead that $H_{2k-1}$ contains a Hamilton cycle $C$. We introduce the following auxiliary bipartite graph $G$ with vertex classes $A$ and $B$. For every $a\in A$ let $S_a$ be the consecutive subsequence of $C$ of length $2k-1$ starting with $a$ (so $S_a$ is an edge of $H_{2k-1}$). We define the edge set of $G$ to be $E(G):=\{\{a,b\}\mid a\in A,b\in S_a\cap B\}$, and also define $U':=\{u \in U \mid d_{G}(u)\geq k\}$ and $B':=\{b\in B\mid d_{G}(b)\leq k-1\}$.

\begin{claim}\label{Obs:BipAuxGraph}
The following properties hold for $G$.
\begin{enumerate}[noitemsep, label=(\roman*)]
\item $d_{G}(a)\geq k-1$ for all $a\in A$,
\item $d_{G}(x)\geq k$ for all $x\in X$,
\item $d_{G}(b)\leq k$ for all $b\in B$,
\item $|B'|\geq n+1$.
\end{enumerate}
\end{claim}

\begin{proof}
The main observation is that if $e$ is an edge of $C$ such that $|e\cap A| \geq k$, then $e \cap A$ is an edge of $H$, so $e$ is of type $(k,k-1)$ and contains no vertex of $X$. To see this, assume otherwise that there exists an edge $e$ for which this observation does not hold. Then $|e \cap A| > k$, and by construction of $H_{2k-1}$ the edge $e'$ preceding $e$ in $C$ fulfils $|e'\cap A| > k$ as well. Consequently every edge of $C$ has $|e\cap A|>|e\cap B|$, contradicting the fact that $|B| \geq |A|$. This observation immediately implies~(i),~(ii) and~(iii). To show~(iv), observe that if every edge of $C$ contains a vertex of $X$, then for the same reason it would follow that $d_G(a) \geq k$ for every $a \in A$ and $d_G(b) \leq k-1$ for every $b \in B$, and so we would have $|A|k \leq |E(G)| \leq |B|(k-1) = (|A| + 1)(k-1)$, contradicting the fact that $|A| \geq k$. 
So some edge $e$ of $C$ does not contain a vertex of $X$, and it follows that we may choose a set $\Sc$ of disjoint consecutive subsequences of $C$ of length $2k-1$, each beginning with a vertex of $X$, such that $X \subseteq \bigcup \Sc$. Then each $S \in \Sc$ is an edge of $C$ containing a vertex of $X$, so has $|S \cap A| \leq k-1$, and so we have $|\Sc| \geq |X|/(k-1) = \ell$. Furthermore, we have $\bigcup\Sc \cap B \subseteq B'$, so if $|\bigcup\Sc\cap B| \geq n + 1$ then we are done. We may therefore assume that $|\bigcup\Sc \cap B| = n = k\ell$, so $|\Sc|=\ell$. Since $C$ has more than $(4k-3)\ell$ vertices, there must exist $S \in \Sc$ such that the $2k-2$ vertices immediately prior to $S$ in $C$ are not in $\bigcup\Sc$. Let $x\in X$ be the first vertex of $S$, let $S'$ be the consecutive subsequence of $C$ of length $2k-1$ ending with $x$, and let $b$ be the last member of $B$ in $S'$. Since $S'$ is an edge of $C$ containing a vertex of $X$, we have $|S' \cap B| \geq k$. It follows that $b\in B'$, and consequently $|B'|\geq n+1$.
\end{proof}

By making use of Claim~\ref{Obs:BipAuxGraph} we now can double-count the edges of $G$ to get
\begin{align*}
|E(G)| & \geq k|X|+(k-1)|U|+|U'| = (k-1)n + (k-1)n+|U'|\;,\mbox{ and}\\
|E(G)| & \leq k|B|-|B'|\leq (2k-1)n + k - (n+1)\;.
\end{align*}
Hence we have $|U'|\leq k-1$.
For every $x\in X$ the previous $k-1$ vertices of $U$ in $C$ are vertices of $U'$ and therefore $|U'|\geq k-1$. Hence $|U'|=k-1$ and all inequalities used in the above calculation are in fact equalities. Therefore the vertices of $U$ have to appear in a consecutive order in $C$ only interrupted by vertices of $B$. Since there are exactly $k-1$ vertices in $U'$, the minimal segment of $C$ containing all the vertices of $U$ is a path consisting only of edges of the type $(k,k-1)$. So deleting all vertices of $B$ from this segment yields a Hamilton path in $H$.

For~(b) let $S$ be a set of $2k-2$ vertices of $H_{2k-1}$. If $|A\cap S|<k-1$ or $|A\cap S|=k+1$, then we can add any vertex of $A \sm S$ to $S$ to form an edge of $H_{2k-1}$, so $d_{H_{2k-1}}(S) \geq |A|-k-1$. If instead $|A\cap S| > k+1$ or $|A \cap S|=k-1$, then we can add any vertex of $B \sm S$ to form an edge of $H_{2k-1}$, so $d_{H_{2k-1}}(S) \geq |B|-k = |A| - k+1$. The same is true if $|A \cap S|=k$ and $A\cap S$ is an edge of $H$. Finally, if $|A\cap S|=k$ and $A\cap S$ is not an edge of $H$, then $S \cup \{a\}$ is an edge of $H_{2k-1}$ for any $a \in A \sm S$ such that $S' \cup \{a\}$ is not an edge of $H$ for any $S' \subseteq S$ of size $k-1$. Since there are $k$ such sets $S'$, each of which has at most $\Delta(H) \leq D$ neighbours in $H$, it follows that $d_{H_{2k-1}}(S) \geq |A|-(D+1)k$. So in all cases we have $d_{H_{2k-1}}(S) \geq |A|-(D+1)k$ as claimed.
\end{proof}

\begin{prop}
For any $k \geq 2$ and any $D$ there is a polynomial-time reduction from $\oHC(k,D)$ to $\HC(2k,\tfrac{n}{2}-(D+1)k)$.
\end{prop}

\begin{proof}
Let $H$ be a $k$-graph on $n$ vertices, let $A:= V(H)$ and let $B$ be a set of size $n$ which is disjoint from $A$. We form a $2k$-graph $H_{2k}$ with vertex set $V := A \cup B$ whose edges are all sets $e\in\binom{V}{2k}$ such that $|A\cap e| \notin \{k,k+1\}$, or such that $|A\cap e|=k$ and $A \cap e \in E(H)$, or such that $|A\cap e|=k+1$ and $e' \not\subseteq A\cap e$ for every $e' \in E(H)$. Since $H_{2k}$ has precisely $2n$ vertices, it then suffices to show that 
\begin{enumerate}[noitemsep, label=(\roman*)]
\item $H$ contains a Hamilton cycle if and only if $H_{2k}$ contains a Hamilton cycle, and
\item if $\Delta(H) \leq D$, then $\delta(H_{2k}) \geq n-(D+1)k$.
\end{enumerate}
As in Proposition~\ref{hptohcodd}, we define the \emph{type} of an edge $e \in E(H_{2k})$ to be the pair $(|e \cap A|, |e \cap B|)$.

For~(i) assume first that $H_{2k}$ has a Hamilton cycle $C$. Observe that our construction of $H_{2k}$ ensures that if $e$ is an edge of $H_{2k}$ of type $(k, k)$, and $e'$ is an edge of $H_{2k}$ of type $(k+1, k-1)$, then $|e \cap e'| < 2k-1$. It follows that $C$ either has no edges with $|e\cap A| > |e\cap B|$ or has no edges with $|e\cap A| \leq |e\cap B|$. Since $|A|=|B|$ we conclude that every edge in $C$ must have type $(k,k)$. Let $C'$ be the subsequence of $C$ obtained by deleting all vertices of $B$ from $C$, and let $S$ be a subsequence of $k$ consecutive vertices of $C'$, so in particular $S \subseteq A$. Then $S$ is included in an edge of $C$, which is an edge of type $(k,k)$ in $H_{2k}$, and so $S$ is an edge of $H$ by construction of $H_{2k}$. We conclude that $C'$ is a Hamilton cycle in $H$. For the other direction assume that $H$ contains a Hamilton cycle $(v_1,\cdots,v_n)$. Then for any enumeration of $B$ as $b_1, \dots, b_n$, the sequence $(v_1,b_1,\cdots,v_n,b_n)$ is a Hamilton cycle in $H_{2k}$, since every consecutive subsequence consisting of $2k$ vertices contains exactly $k$ vertices from $A$, which form an edge of $H$.

For~(ii), fix a set $S$ of $2k-1$ vertices of $H_{2k}$, and observe that if $|S\cap A|\neq k$, then $S$ is included in at least $n-(k+1)$ edges of $H_{2k}$. If instead $|S\cap A|=k$ and $S \cap A$ is an edge of $H$, then $S$ is included in at least $n - (k-1)$ edges of $H_{2k}$, since we may add any vertex of $B \sm S$ to form an edge of $H_{2k}$. Finally, if $|S\cap A|=k$ and $S \cap A$ is not an edge of $H$, then $S\cup\{a\}$ is an edge of $H_{2k}$ for every vertex $a \in A \sm S$ except for those vertices $a$ such that $S' \cup \{a\}$ is an edge of $H$ for some $S' \subseteq S \cap A$ of size $k-1$. Since there are $k$ such subsets $S'$, each of which has $d_H(S') \leq \Delta(H) \leq D$ neighbours in $H$, it follows that $S$ is included in at least $n-k-kD$ edges of $H_{2k}$. So in all cases we have $d_{H_{2k}}(S) \geq n-k-kD$ as claimed.
\end{proof}

\end{document}